\documentclass[14pt]{amsart}
\usepackage{cases}
\usepackage{amsmath}
\usepackage{amsfonts}
\usepackage{bm}
\usepackage{amsfonts,amsmath,amssymb,amscd,bbm,amsthm,mathrsfs,dsfont}
\usepackage{mathrsfs}
\usepackage{pb-diagram}
\usepackage{color}
\usepackage{amssymb}
\usepackage{xypic}
\usepackage[dvips]{graphicx}

\usepackage[all]{xy}

\newtheorem{Theorem}{Theorem}[section]
\newtheorem{Lemma}[Theorem]{Lemma}
\newtheorem{Construction}[Theorem]{Construction}

\newtheorem{Definition}[Theorem]{Definition}
\newtheorem{Corollary}[Theorem]{Corollary}
\newtheorem{Proposition}[Theorem]{Proposition}
\newtheorem{Example}[Theorem]{Example}
\newtheorem{Remark}[Theorem]{Remark}
\newtheorem{Conjecture}[Theorem]{Conjecture}
\newtheorem{Observation}[Theorem]{Observation}
\newtheorem{Fact}[Theorem]{Fact}
\title [Unfolding  of acyclic sign-skew-symmetric cluster algebras  ]
 {Unfolding of acyclic sign-skew-symmetric cluster algebras and applications to positivity and $F$-polynomials }

\author{Min Huang $\;\;\;\;\;\;$ Fang Li $\;\;\;\;\;\;$}
\address{Min Huang
\newline Department
of Mathematics, Zhejiang University (Yuquan Campus), Hangzhou, Zhejiang
310027,  P.R.China}
\email{minhuang1989@hotmail.com}

\address{Fang Li
\newline Department
of Mathematics, Zhejiang University (Yuquan Campus), Hangzhou, Zhejiang
310027, P.R.China}
\email{fangli@zju.edu.cn}

\date{version of \today}

\newcommand{\lra}{\longrightarrow}

\newcommand{\ra}{\rightarrow}
\newcommand{\sdp}{\times\kern-.2em\vrule height1.1ex depth-.05ex}
\newcommand{\epi}{\lra \kern-.8em\ra}

 \normalbaselines
\begin{document}

\renewcommand{\thefootnote}{\alph{footnote}}
\setcounter{footnote}{-1} 
\renewcommand{\thefootnote}{\alph{footnote}}
\setcounter{footnote}{-1} 

\begin{abstract} In this paper, we build the unfolding approach from acyclic sign-skew-symmetric matrices of finite rank to skew-symmetric matrices of infinite rank, which can be regard as an improvement of that in the skew-symmetrizable case.  Using this approach,  we  give a positive answer to the problem by Berenstein, Fomin and Zelevinsky in \cite{fz3} which asks  whether an acyclic sign-skew-symmetric matrix is always totally sign-skew-symmetric. As applications,  the positivity for cluster algebras in acyclic sign-skew-symmetric case is given;  further, the $F$-polynomials of cluster algebras are proved to have constant term 1 in acyclic sign-skew-symmetric case.

\end{abstract}

\maketitle
\bigskip

\tableofcontents

\section{Introduction and preliminaries }

Cluster algebras are commutative algebras that were introduced by Fomin and Zelevinsky \cite{fz1} in order to give a
combinatorial characterization of total positivity and canonical bases in algebraic groups. The theory of cluster
algebras is related to numerous other fields including Lie theory, representation theory of algebras, the periodicity issue, Teichm$\ddot{u}$ller theory and mathematics physics.

A cluster algebra is defined starting from a totally sign-skew-symmetric matrix. Skew-symmetric and skew-symmetrizable matrices are always totally sign-skew-symmetric, but other examples of totally sign-skew-symmetric matrices are rarely known. A class of non-skew-symmetric and non-skew-symmetrizable $3\times 3$ totally sign-skew-symmetric matrices is given in Proposition 4.7 of \cite{fz1}.

A cluster algebra is a subalgebra of rational function field with a distinguish set of generators, called {\bf cluster variables}. Different cluster variables are related by a iterated procedure called mutations. By construction, cluster variables are rational functions. In \cite{fz1}, Fomin and Zelevinsky in fact proved that they are Laurent polynomials of initial cluster variables, where it was conjectured the coefficients of these Laurent polynomials are non-negative, which is called the positivity conjecture. Since the introduction of cluster algebras by Fomin and Zelevinsky in \cite{fz1}, positivity conjecture is a hot topic. Positivity was proved in the acyclic skew-symmetric case in \cite{KQ}, in the case for the cluster algebras arising from surfaces in \cite{MSW1}, in the skew-symmetric case in \cite{LS} and in the skew-symmetrizable case in \cite{GHKK}. However, the positivity for the totally sign-skew-symmetric case never be studied so far before this work.

The concept of unfolding of a skew-symmetrizable matrix was introduced by Zelevinsky and the authors of \cite{FST} and appeared in \cite{FST} firstly. Demonet \cite{D} proved that any acyclic skew-symmetrizable matrix admits an unfolding. The idea based on unfolding is to use skew-symmetric cluster algebras to characterize skew-symmetrizable cluster algebras. Note that in \cite D and \cite{FST}, all cluster algebras are of finite rank.

Galois covering theory was introduced by Bongartz and Gabriel for skeletal linear categories in \cite{BG}. This theory was extended to   arbitrary linear categories by  Bautista and Liu in \cite{BL}. The basic idea of Galois covering theory is to use infinite dimensional algebras to characterize finite dimensional algebras.

In this paper, we combine the ideas of unfolding of a skew-symmetrizable matrix and Galois covering theory successfully to  introduce the unfolding theory of a sign-skew-symmetric matrix, see Definition \ref{unfolding}. For this, we have to deal with skew-symmetric matrices of {\em infinite rank}. Moreover, given an acyclic sign-skew-symmetric matrix $B$, we introduce a general construction of the unfolding $(Q(B), \Gamma)$. Our main result about unfolding theory is that
\\
\\
{\bf Theorem \ref{mainlemma}. }
{\em If $B\in Mat_{n\times n}(\mathbb Z)$ is an acyclic sign-skew-symmetric matrix, then $(Q(B),\Gamma)$ constructed from $B$ in Construction \ref{construction} is an unfolding of $B$.}
\\

Following this result, we will prove or answer affirmatively some conjectures and problems on cluster algebras given by Fomin and Zelevinsky,etc, as we list as follows.

(I) Berenstein, Fomin and Zelevinsky give an open problem in \cite{fz3} in 2005 that.
\\
{\bf Problem (1.28, \cite{fz3}).}
{\em Is any acyclic seed totally mutable? In other words, is any acyclic sign-skew-symmetric
integer matrix totally sign-skew-symmetric?}

We will give a positive answer to this problem:\\
{\bf Theorem \ref{sign-skew}.}
{\em The acyclic sign-skew-symmetric matrices are always totally sign-skew-symmetric.}

(II) Fomin and Zelevinsky conjectured in \cite{fz1} that,
\\
{\bf Positivity Conjecture, \cite{fz1}.}
{\em Let $\mathcal A(\Sigma)$ be a cluster algebra over $\mathbb {ZP}$. For any cluster $X$, and any
cluster variable $x$, the Laurent polynomial expansion of $x$ in the cluster $X$ has coefficients
which are nonnegative integer linear combinations of elements in $\mathbb P$.}

Combing the positivity theorem on skew-symmetric cluster algebras in \cite{LS} by Lee and Schiffler, we will prove that,
\\
{\bf Theorem \ref{positivity}.}
{\em The positivity conjecture is true for acyclic sign-skew-symmetric cluster algebras.}

(III) Fomin and Zelevinsky conjectured that,
\\
{\bf Conjecture 5.4, \cite{fz4}.}
{\em In a cluster algebra, each $F$-polynomial has constant term $1$.}

Using Theorem \ref{mainlemma}, we will verify this conjecture affirmatively for all acyclic sign-skew-symmetric cluster algebras.
\\
{\bf Theorem \ref{poly}.}
{\em In an acyclic sign-skew-symmetric cluster algebra, each $F$-polynomial has constant term $1$.}

Here we give the outline of our technique for the major target.

 We prove the main theorem (Theorem \ref{mainlemma}) about unfolding following the strategy of \cite{D}. Precisely, for an arbitrary acyclic matrix $B$, we construct an infinite quiver $Q(B)$, briefly written as $Q$, with a group $\Gamma$ action. To check  $(Q,\Gamma)$ as an unfolding of $B$, it suffices to prove that $Q$ has no $\Gamma$-$2$-cycles and $\Gamma$-loops, and then we can do orbit mutations on $Q$ (see Definition \ref{orbitmu}), moreover, to show the preservation of the property  ``{\em no $\Gamma$-$2$-cycles and $\Gamma$-loops}".

  In order to prove this above fact, we build a cluster-tilting subcategory $\mathcal T_0$ of a Frobienus $2$-Calabi-Yau category $\mathcal C^{Q}$ (Corollary \ref{2-cy}) such that the Gabriel quiver of $\underline{\mathcal T}_0$ is $Q$, where $\underline{\mathcal T}_0$ is the subcategory of the stable category $\underline {\mathcal C^Q}$ of $\mathcal C^Q$ corresponding to $\mathcal T_0$. By Proposition \ref{quivermu},  proving this fact can be replaced by proving that there are no $\Gamma$-$2$-cycles and $\Gamma$-loops in the cluster-tilting subcategories obtained by a sequence of mutations from $\mathcal T_0$, which can be verified by using (2)(4)(5) in Theorem \ref{cluster tilting} step by step.

By Theorem \ref{mainlemma}, using Lemma \ref{basiclemma}, we will prove that acyclic sign-skew-symmetric matrices are always totally sign-skew-symmetric (Theorem \ref{sign-skew}), as the positive answer of a problem asked by Berenstein, Fomin and Zelevinsky.

  Due to ($Q, \Gamma)$ as the unfolding of $B$, we can deduce a surjective algebra homomorphism $\pi:\mathcal A(\widetilde \Sigma)\rightarrow \mathcal A(\Sigma)$, where $\widetilde\Sigma$ is the seed associated with $Q$ and $\Sigma$ is the seed associated with $B$, see Theorem \ref{sur} (3).

 The key of our method is that proving the preservation of the positivity conjecture and the conjecture on $F$-polynomials under this surjective homomorphism. Then using the known conclusions for $\mathcal A(\widetilde \Sigma)$ in the skew-symmetric case by Lee and Schiffler \cite{LS} and Zelevinsky, etc \cite{DWZ} respectively, we obtain that the positivity conjecture holds in the acyclic sign-skew-symmetric case (Theorem \ref{positivity}) and the $F$-polynomials have constant terms $1$ in the same case (Theorem \ref{poly}).

The paper is organized as follows. In Section 2, we firstly give the definition of covering and unfolding of a sign-skew-symmetric matrix (Definition \ref{unfolding}), and then construct a strongly finite quiver $Q$ with a group $\Gamma$ action from an acyclic sign-skew-symmetric matrix $B$ (Construction \ref{construction}). We give the statement of the unfolding theorem, see Theorem \ref{mainlemma}. Based on this result, we prove that acyclic sign-skew-symmetric matrix is always totally sign-skew-symmetric (Theorem \ref{sign-skew}). In Section 3, 4, 5, the preparations are given for the proof of Theorem \ref{mainlemma}. Precisely, in Section 3, we study the global dimension of $K$-linear categories and prove that a functorially finite rigid subcategory of a strongly almost finite category, which has global dimension less than 3 and satisfies certain conditions, is a cluster tilting subcategory, see Theorem \ref{keylemma}. Lenzing's Theorem (\cite{L}, Satz 5) is generalized for locally bounded categories (Theorem \ref{nil}). In Section 4, using the theory of colimits of additive categories (Proposition \ref{2CY-Fro}), we improve on the construction in \cite{GLS1} for a strongly almost finite quiver $Q$ to construct a strongly almost finite $2$-Calabi-Yau Frobenius category $\mathcal C^{Q}$ (Corollary \ref{2-cy}). For a strongly almost finite $2$-Calabi-Yau Frobenius category, we discuss the mutation of maximal $\Gamma$-rigid subcategories and their global dimensions in Section 5 (Proposition \ref{globaldim'} and \ref{globaldim}). We prove Theorem \ref{mainlemma} in Section 6. In Section 7, using the Theorem \ref{mainlemma}, we prove that there is a surjective algebra homomorphism $\pi$ from the cluster algebra associated with $Q$ to that associated with $B$ (Theorem \ref{sur}). Relying on the surjective map $\pi$, we prove the positivity (Theorem \ref{positivity}) and the $F$-polynomials have constant $1$ (Theorem \ref{poly}) for acyclic sign-skew-symmetric cluster algebras.

In this paper, we will always assume the background field $K$ is an algebraic closure field with characteristic $0$.
\\

The original definition of cluster algebra given in \cite{fz1} is in terms of exchange pattern. We recall the equivalent definition in terms of seed mutation in \cite{fz2}; for more details, refer to \cite{fz1, fz2}.

An $n\times n$ integer matrix $A=(a_{ij})$ is called {\bf
sign-skew-symmetric} if either $a_{ij}=a_{ji}=0$ or $a_{ij}a_{ji}<0$ for any $1\leq i,j\leq
n$.

An $n\times n$ integer matrix $A=(a_{ij})$ is called {\bf
skew-symmetric} if $a_{ij}=-a_{ji}$ for all $1\leq i,j\leq n$.

An $n\times n$ integer matrix $A=(a_{ij})$ is called {\bf
$D$-skew-symmetrizable} if $d_ia_{ij}=-d_ja_{ji}$ for all $1\leq
i,j\leq n$, where $D$=diag$(d_i)$ is a diagonal matrix with all
$d_i\in \mathbb{Z}_{\geq 1}$.

Let $\widetilde{A}$ be an $(n+m)\times n$ integer matrix whose
principal part, denoted by $A$, is the $n\times n$ submatrix formed
by the first $n$-rows and the first $n$-columns. The entries of
$\widetilde A$ are written by $a_{xy}$, $x\in \widetilde{X}$ and $y\in
X$. We say $\widetilde A$ to be {\bf
sign-skew-symmetric} (respectively, {\bf skew-symmetric}, $D$-{\bf
skew-symmetrizable}) whenever $A$ possesses this property.

For two $(n+m)\times n$ integer
 matrices $A=(a_{ij})$ and
$A'=(a'_{ij})$, we say that $A'$ is
obtained from $A$ by a {\bf matrix mutation $\mu_i$} in direction $i, 1\leq i\leq n$, represented as $A'=\mu_i(A)$, if
\begin{equation}\label{matrixmutation}
a'_{jk}=\left\{\begin{array}{lll} -a_{jk},& \text{if}
~~j=i \ \text{or}\;
k=i;\\a_{jk}+\frac{|a_{ji}|a_{ik}+a_{ji}|a_{ik}|}{2},&
\text{otherwise}.\end{array}\right.\end{equation}

We say $A$ and $A'$ are {\bf mutation equivalent} if $A'$ can be obtained from $A$ by a sequence of matrix mutations.

It is easy to verify that $\mu_i(\mu_i(A))=A$. The
skew-symmetric/symmetrizable property of matrices is invariant under
mutations. However, the sign-skew-symmetric property is not so. For
this reason, a sign-skew-symmetric matrix $A$ is called {\bf totally
sign-skew-symmetric} if any matrix, that is mutation equivalent to
$A$, is sign-skew-symmetric.

 Give a field $\mathds F$ as an extension of the rational number field $\mathds Q$, assume that $u_1,\cdots,u_n,x_{n+1},\cdots,x_{n+m}\in \mathds F$ are $n+m$ algebraically independent over $\mathds Q$ for a positive integer $n$ and a non-negative integer $m$ such that $\mathds F=\mathds{Q}(u_1,\cdots,u_n,x_{n+1},\cdots,x_{n+m})$, the field of rational functions in the  set $\widetilde X=\{u_1,\cdots,u_n,x_{n+1},\cdots,x_{n+m}\}$ with coefficients in $\mathds{Q}$.

A {\bf seed} in $\mathds F$ is a triple $\Sigma=(X, X_{fr},\widetilde B)$, where\\
(a)~$X=\{x_1,\cdots x_n\}$ is a transcendence basis of $\mathds F$ over the fraction field of $\mathds Z[x_{n+1},\cdots,x_{n+m}]$, which is called a {\bf cluster}, whose each $x\in X$ is called a {\bf cluster variable} (see \cite{fz2});\\
(b)~$X_{fr}=\{x_{n+1},\cdots,x_{n+m}\}$ are called the {\bf frozen cluster} or, say, the {\bf frozen part} of $\Sigma$ in $\mathds F$, where all $x\in X_{fr}$ are called  {\bf stable (cluster) variables} or {\bf frozen (cluster) variables};\\
(c)~$\widetilde{X}=X\cup X_{fr}$ is called a {\bf extended cluster};\\
(d)~$\widetilde{B}=(b_{xy})_{x\in \widetilde{X},y\in X}=(B^T\;B_1^T)^T$ is a $(n+m)\times n$ matrix over $\mathds{Z}$ with rows and columns indexed
by $X$ and $\widetilde{X}$, which is totally sign-skew-symmetric. The $n\times n$ matrix $B$ is called the {\bf exchange matrix} and
$\widetilde B$ the {\bf extended exchange matrix} corresponding to the seed $\Sigma$.

Let $\Sigma=(X,X_{fr},\widetilde B)$ be a seed in $\mathds F$ with $x\in X$, the {\bf mutation} $\mu_x$ of $\Sigma$ at $x$ is defined satisfying $\mu_x(\Sigma)=(\mu_x(X), X_{fr}, \mu_{x}(\widetilde B))$ such that

(a)~ The {\bf adjacent cluster} $\mu_{x}(X)=\{\mu_{x}(y)\mid y\in X\} $, where  $\mu_{x}(y)$ is given by the {\bf exchange relation}
\begin{equation}\label{exchangerelation}
\mu_{x}(y)=\left\{\begin{array}{lll} \frac{\prod\limits_{t\in \widetilde{X}, b_{tx}>0}t^{b_{tx}}+\prod\limits_{t\in \widetilde{X}, b_{tx}<0}t^{-b_{tx}}}{x},& \text{if}
~~y=x;\\y,& \text{if}~~ y\neq x. \end{array}\right.\end{equation}
This new variable $\mu_{x}(x)$ is also called a {\em new} {\bf cluster
variable}.

(b)~ $\mu_{x}(\widetilde B)$ is obtained from B by applying the matrix mutation in direction
$x$ and then relabeling one row and one column by replacing $x$ with $\mu_x(x)$.

It is easy to see that the mutation $\mu_x$ is an involution, i.e., $\mu_{\mu_x(x)}(\mu_x(\Sigma))=\Sigma$.

Two seeds $\Sigma'$ and $\Sigma''$ in $\mathds F$ are called {\bf mutation equivalent} if there exists a sequence of mutations $\mu_{y_1},\cdots,\mu_{y_s}$ such that $\Sigma''=\mu_{y_s}\cdots\mu_{y_1}(\Sigma')$.
 Trivially, the mutation equivalence gives an equivalence relation on the set of seeds in $\mathds F$.

Let $\Sigma$ be a seed in $\mathds F$. Denote by $\mathcal S$ the set of all seeds mutation equivalent to $\Sigma$. In particular, $\Sigma\in \mathcal S$. For any $\bar \Sigma\in\mathcal S$, we have $\bar \Sigma=(\bar X,X_{fr},\widetilde {\bar B})$. Denote $\mathcal X=\cup_{\bar\Sigma\in\mathcal S}\bar X$.

\begin{Definition}\label{clusteralgebra}
 Let $\Sigma$ be a seed in $\mathds F$. The {\bf cluster algebra} $\mathcal A=\mathcal A(\Sigma)$, associated with $\Sigma$, is defined to be the $\mathds Z[x_{n+1},\cdots,x_{n+m}]$-subalgebra of $\mathds F$ generated by $\mathcal X$. $\Sigma$ is called the {\bf initial seed} of $\mathcal A$.
\end{Definition}

This notion of cluster algebra was given in \cite{fz1}\cite{fz2}, where it is called the {\bf cluster algebra of geometric type} as a special case of general cluster algebras.

\section{Unfolding and totality of acyclic sign-skew-symmetric matrices}\label{quiver}

Firstly, we recall some terminologies of (infinite) quivers in \cite{BLP}. Let $Q$ be an (infinite) quiver. We denote by $Q_0$ the set of vertices of $Q$; by $Q_1$ the set of arrows of $Q$. For $i\in Q_0$, denote by $i_{+}$ (respectively, $i_{-}$) the set of arrows starting (respectively, ending) in $i$. Say $Q$ to be {\bf locally finite} if $i_{+}$ and $i_{-}$ are finite for any $i\in Q_0$. For $i,j\in Q_0$, let $Q(i, j)$ stands for the set of paths from $i$ to $j$ in $Q$. We say that $Q$ is {\bf interval-finite} if $Q(i, j)$ is finite for any $i,j\in Q_0$. One calls $Q$  {\bf strongly locally finite} if it is locally finite and interval-finite. In this paper, we call $Q$ {\bf strongly almost finite} if it is strongly locally finite and has no infinity paths.

Note that a strongly almost finite quiver $Q$ is always acyclic since it is interval finite.

\subsection{Notions of covering and unfolding}.

For a locally finite quiver $Q$, we can associate an (infinite) skew-symmetric row and column finite matrix $\widetilde B_Q=(b_{ij})_{i,j\in Q_0}$ as follows.

Assume $Q$ has no 2-cycles and no loops. So, if there are arrows from $i$ to $j$ for $i,j\in Q_0$, there are no arrows from $j$ to $i$. Thus, we define $b_{ij}$ to be the number of arrows from $i$ to $j$ and let $b_{ji}=-b_{ij}$. Trivially, $b_{ii}=0$ for $i\in Q_0$. Then $\widetilde B_Q=(b_{ij})_{i,j\in Q_0}$ is obtained. It is easy to see that $\widetilde B_Q$ is a skew-symmetric row and column finite matrix, $Q$ and $\widetilde B_Q$ determinate uniquely each other.

 In case of no confusion, for convenient, we also denote the quiver $Q$ as $(b_{ij})_{i,j\in Q_0}$, or say, directly replace $Q$ by the matrix $\widetilde B_Q$.

A (rank infinite) seed $\Sigma(Q)=(\widetilde X,\widetilde B_Q)$ is associated to $Q$, with cluster $\widetilde X=\{x_i\;|\;i\in Q_0\}$ and exchange matrix $\widetilde B_Q=(b_{ij})_{i,j\in Q_0}$.

Let $Q$ be a locally finite quiver $Q$ with an action of a group $\Gamma$ (maybe infinite). For a vertex $i\in Q_0$, as in \cite{D}, a {\bf $\Gamma$-loop} at $i$ is an arrow from $i\rightarrow h\cdot i$ for some $h\in \Gamma$, a {\bf $\Gamma$-$2$-cycle} at $i$ is a pair of arrows $i\rightarrow j$ and $j\rightarrow h\cdot i$ for some $j\notin \{h'\cdot i\;|\;h'\in \Gamma\}$ and $h\in \Gamma$.
 Denote by $[i]$ the orbit set of $i$ under the action of $\Gamma$. Say {\bf $Q$ has no $\Gamma$-loops} ({\bf $\Gamma$-$2$-cycles}, respectively) {\bf at} $[i]$ if $Q$ has no $\Gamma$-loops ($\Gamma$-$2$-cycles, respectively) at any $i'\in [i]$.

\begin{Definition}\label{orbitmu}
Let $Q=(b_{ij})$ be a locally finite quiver with an group $\Gamma$ action on it. Denote $[i]=\{h\cdot i\;|\;h\in \Gamma\}$ the orbit of vertex $i\in Q_0$. Assume that $Q$ admits no $\Gamma$-loops and no $\Gamma$-$2$-cycles at $[i]$, we define an {\bf adjacent quiver}  $Q'=(b'_{i'j'})_{i',j'\in Q_0}$ from $Q$ to be the quiver by following:

(1)~ The vertices are the same as $Q$,

(2)~ The arrows are defined as
 \[\begin{array}{ccl} b'_{jk} &=&
         \left\{\begin{array}{ll}
             -b_{jk}, &\mbox{if $j\in [i]$ or $k\in [i]$}, \\
              b_{jk}+\sum\limits_{i'\in[i]}\frac{|b_{ji'}|b_{i'k}+b_{ji'}|b_{i'k}|}{2}, &\mbox{otherwise.}
         \end{array}\right.
 \end{array}\]
Denote $Q'$ as $\widetilde\mu_{[i]}(Q)$ and call $\widetilde\mu_{[i]}$  the {\bf orbit mutation} at direction $[i]$ or at $i$ under the action $\Gamma$. In this case, we say that $Q$ can {\bf do orbit mutation at} $[i]$.

\end{Definition}

Since $b_{jk}=-b_{kj}$ for all $j,k\in Q_0$, it is easy to see $b'_{jk}=-b'_{kj}$ for all $j,k\in Q'_0$, that is, $Q'=(b'_{i'j'})_{i',j'\in Q_0}$ as matrix is skew-symmetric.

\begin{Fact}\label{fact}
We can write that $\widetilde \mu_{[i]}(Q)=(\prod\limits_{i'\in [i]}\mu_{i'})(Q)$.
\end{Fact}

\begin{Remark}\label{rem2.3}

(1)~ Since $Q$ is locally finite, the sum in the definition of $b'_{jk}$ is finite and so is well-defined.

(2)~ $\widetilde\mu_{[i]}(Q)$ is also locally finite, since the finiteness of valency of any vertex does not changed under orbit mutation by definition.

(3)~ The conditions no $\Gamma$-loops and no $\Gamma$-$2$-cycles become to no loops and no $2$-cycles when $\Gamma$ is a trivial group, which is necessary in the definition of cluster algebras. Such conditions are crucial in the sequel, e.g. for Lemma \ref{basiclemma} and Theorem \ref{sur} which is the key for the following applications.

(4)~ If $Q$ has no $\Gamma$-loops and $\Gamma$-$2$-cycles at $[i]$, it is easy to see that $\widetilde\mu_{[i]}(Q)$ has no $\Gamma$-loops at $[i]$.
\end{Remark}

Note that if $\Gamma$ is the trivial group $\{e\}$, then the definition of orbit mutation of a quiver is the same as that of quiver mutation (see \cite{fz1}\cite{fz2}) or say $\widetilde\mu_{[i]}=\mu_i$ in this case.

\begin{Definition}\label{unfolding} (i)~
For a locally finite quiver $Q=(b_{ij})_{i,j\in Q_0}$ with a group $\Gamma$ (maybe infinite) action, let $\overline Q_0$ be the orbit sets of the vertex set $Q_0$ under the $\Gamma$-action. Assume that $n=|\overline Q_0|<+\infty$ and $Q$ has no $\Gamma$-loops and $\Gamma$-$2$-cycles.

Define a sign-skew-symmetric matrix $B(Q)=(b_{[i][j]})$ to $Q$ satisfying (1)~ the size of the matrix $B(Q)$ is $n\times n$; (2)~ $b_{[i][j]}=\sum\limits_{i'\in [i]}b_{i'j}$ for $[i],[j]\in \overline Q_0$.

(ii)~  For an $n\times n$ sign-skew-symmetric matrix $B$, if there is a locally finite quiver $Q$ with a group $\Gamma$ such that $B=B(Q)$ as constructed in (i), then we call $(Q,\Gamma)$  a {\bf covering} of $B$.

(iii)~ For an $n\times n$ sign-skew-symmetric matrix  $B$, if there is a locally finite quiver $Q$ with an action of group $\Gamma$ such that $(Q,\Gamma)$  is a covering of $B$ and $Q$ can do arbitrary steps of orbit mutations, then $(Q,\Gamma)$ is called an {\bf unfolding} of $B$; or equivalently, $B$ is called the {\bf folding} of $(Q,\Gamma)$.
\end{Definition}

In the above definition, the part (iii) is a generalization of the notion of unfolding in the skew-symmetrizable case (see \cite{FST}) to the sign-skew-symmetric case.

For $j'\in [j]$, there is $\sigma\in \Gamma$ with $\sigma(j)=j'$ and then $b_{i'j}=b_{\sigma(i')j'}$. Since the restriction of $\sigma$ on $[i]$ is a bijection, we have also $b_{[i][j]}=\sum\limits_{i'\in [i]}b_{\sigma(i')j'}$ for $[i],[j]\in \overline Q_0$. It means in (i), the expression (2) is well-defined.

Because $Q$ has no $\Gamma$-$2$-cycles, fixed $j$, all $b_{i'j}$ have the same sign or equal 0 when $i'$ runs over the orbit $[i]$. In the case $b_{i'j}=0$ for all $i'\in [i]$, then also $b_{j'i}=-b_{ij'}=0$ for all $j'\in [j]$ via $\sigma(j)=j'$. Thus, $b_{[i][j]}=0=b_{[j][i]}$. Otherwise, there is $i_0\in [i]$ with $b_{i_0j}\not=0$, then $b_{ji_0}=-b_{i_0j}\not=0$. Thus, $b_{[j][i]}=\sum\limits_{j'\in [j]}b_{j'i_0}\not=0$ and $b_{[i][j]}=\sum\limits_{i'\in [i]}b_{i'j}\not=0$ have the opposite sign.
 It follows that in (i),  $B=B(Q)=(b_{[i][j]})$ is indeed sign-skew-symmetric.

\begin{Lemma}\label{basiclemma}
If $(Q, \Gamma)$ is a covering of $B$ and $Q$ can do orbit mutation at $[i]$  such that $\widetilde\mu_{[i]}(Q)$ has no $\Gamma$-$2$-cycles for some vertex $i$, then $(\widetilde\mu_{[i]}(Q), \Gamma)$ is a covering of $\mu_{[i]}(B)$.
\end{Lemma}
\begin{proof}
Denote the matrix $\mu_{[i]}(B)$ by $(b'_{[j][k]})$. We have \[\begin{array}{ccl} b'_{[j][k]} &=&
         \left\{\begin{array}{ll}
             -b_{[j][k]}, &\mbox{if $[j]= [i]$ or $[k]=[i]$}, \\
              b_{[j][k]}+\frac{|b_{[j][i]}|b_{[i][k]}+b_{[j][i]}|b_{[i][k]}|}{2}, &\mbox{otherwise.}
         \end{array}\right.
 \end{array}\] If $[j]= [i]$ or $[k]=[i]$, then we have $b'_{[j][k]}=-b_{[j][k]}=-\sum\limits_{j'\in [j]}b_{j'k}=\sum\limits_{j'\in [j]}b'_{j'k};$ otherwise, for any $i'\in [i]$, we have $|\sum\limits_{j'\in [j]}b_{j'i}|=|\sum\limits_{j'\in [j]}b_{j'i}|$. Since $Q$ has no $\Gamma$-$2$-cycles at $[i]$, we have $|\sum\limits_{i'\in [i]}b_{i'k}|=\sum\limits_{i'\in [i]}|b_{i'k}|$ and $|\sum\limits_{j'\in [j]}b_{j'i'}|=\sum\limits_{j'\in [j]}|b_{j'i'}|$, then
 \[\begin{array}{ccl}
b'_{[j][k]} & = &
b_{[j][k]}+\frac{|b_{[j][i]}|b_{[i][k]}+b_{[j][i]}|b_{[i][k]}|}{2}
  =  \sum\limits_{j'\in [j]}b_{j'k}+\frac{|\sum\limits_{j'\in [j]}b_{j'i}|\sum\limits_{i'\in [i]}b_{i'k}+\sum\limits_{j'\in [j]}b_{j'i}|\sum\limits_{i'\in [i]}b_{i'k}|}{2}\\
 & = & \sum\limits_{j'\in [j]}(b_{j'k}+\sum\limits_{i'\in [i]}\frac{|b_{j'i'}|b_{i'k}+b_{j'i'}|b_{i'k}|}{2})
  =  \sum\limits_{j'\in [j]}b'_{j'k}.
\end{array}\]

Therefore, we have $b'_{[j][k]}=\sum\limits_{j'\in [j]}b'_{j'k}$ for $[j],[k]$.

In addition, since $Q$ has no $\Gamma$-$2$-cycles, it is easy to see that $\widetilde\mu_{[i]}(Q)$ admits no $\Gamma$-loops.

Since $\widetilde\mu_{[i]}(Q)$ has no $\Gamma$-$2$-cycles, we have $b'_{j'k}b'_{j''k}\geq 0$ and $b'_{jk'}b'_{jk''}\geq 0$ for all $j',j''\in [j]$, $k',k''\in [k]$. Moreover, since $b_{j'k'}=-b_{k'j'}$ for all $j'\in [j]$, $k'\in [k]$, we have that  $b'_{[j][k]}b'_{[k][j]}=(\sum\limits_{j'\in [j]}b'_{j'k})(\sum\limits_{k'\in [k]}b'_{k'j})\leq 0$. Then it is easy to see that  $b'_{[j][k]}b'_{[k][j]} = 0$ if and only if $b'_{[j][k]}=b'_{[k][j]}=0$. This means $\mu_{[i]}(B)$ is sign-skew-symmetric. It follows that $(\widetilde\mu_{[i]}(Q), \Gamma)$ is a covering of $\mu_{[i]}(B)$.
\end{proof}

This result means the preservability of covering under orbit mutations.

\subsection{Strongly almost finite quiver from an acyclic sign-skew-symmetric matrix}\label{strongly-finite}.

Let $P, Q$ be two quivers.

(1) If there is a full subqiver $P'$ of $P$, a full subquiver $Q'$ of $Q$ and a quiver isomorphism $\varphi: P'\rightarrow Q'$, then we say $P$ and $Q$ to be {\bf glueable} along $\varphi$.

(2)
Let $P$ and $Q$ respectively be corresponding to the skew-symmetric matrices $\left(\begin{array}{cc}
B_{11} & B_{12}\\
B_{21} & B_{22}
\end{array}\right)$ and $\left(\begin{array}{cc}
B'_{11} & B'_{12}\\
B'_{21} & B'_{22}
\end{array}\right)$, where the row/column indices of $B_{22}$ and $B'_{22}$ be respectively corresponding to the vertices of $P'$ and  $Q'$. Since $\varphi$ is an isomorphism, we have $B_{22}=B'_{22}$.

(3) Denote by $P\coprod_{\varphi}Q$ the quiver corresponding to the skew-symmetric matrix $\left(\begin{array}{ccc}
B_{11} & 0 & B_{12}\\
0 & B'_{11} & B'_{12}\\
B_{21} & B'_{21} & B_{22}
\end{array}\right)$.  We call $P\coprod_{\varphi}Q$ the {\bf gluing quiver} of $P$ and $Q$  along $\varphi$.

Roughly speaking, $P\coprod_{\varphi}Q$ is obtained through gluing $P$ and $Q$ along the common full subquiver (up to isomorphism).

For example, for $P'=Q'$ the arrow $\alpha: 1\overset{}{\rightarrow}2$ in both $P$ and $Q$, we can say $P\coprod_{\varphi}Q$ to be the  gluing quiver of $P$ and $Q$ along $\varphi=Id_{\alpha}$.

For any $n\times n$ sign-skew-symmetric matrix $B$, we associate a (simple) quiver $\Delta(B)$
with vertices $1,\cdots,n$ such that for each pair $(i, j)$ with $b_{ij}> 0$, there is exactly one arrow
from vertex $i$ to vertex $j$.
  Trivially, $\Delta(B)$ has no loops and no 2-cycles.

 The matrix $B$ is said in \cite{LLM} to be {\bf acyclic} if $\Delta(B)$ is acyclic and to be {\bf connected} if $\Delta(B)$ is connected.

In the sequel, we will always assume that $B$ is connected and acyclic.

\begin{Construction} \label{construction}
Let $B=(b_{ij})\in Mat_{n\times n}(\mathbb Z)$ be an acyclic matrix. We construct the (infinite) quiver $Q(B)$ (written as $Q$ briefly if $B$ is fixed) via the following steps inductively.
\begin{itemize}
 \item  For each $i=1,\cdots,n$, define a quiver $Q^i$ like this: $Q^i$ has $\sum\limits_{j=1}^n|b_{ji}|+1$ vertices with one vertex labeled by $i$ and $|b_{ji}| (j\neq i)$ various vertices labeled by $j$. If $b_{ji}>0$, there is an arrow from each vertex labeled by $j$ to the unique vertex labeled by $i$; if $b_{ji}<0$, there is an arrow from the unique vertex labeled by $i$ to each vertex labeled by $j$.  Note that if $b_{ji}=0$,  then there is no vertex labelled by $j$.
  \item Let $Q_{(1)}=Q^1$. Call the unique vertex label by $1$ the ``old" vertex, and the other vertices the ``new" vertices.
  \item For a ``new" vertex in $Q_{(1)}$ which is labelled by $i_1$, $Q^{i_1}$ and $Q_{(1)}$ share a common arrow, denoted as $\alpha_1$, which is the unique arrow incident with the ``new'' vertex labelled by $i_1$. Then the gluing quiver $Q^{i_1}\coprod_{Id_{\alpha_1}}Q_{(1)}$ is obtained along the common arrow $\alpha_1$. For another ``new" vertex labelled by $i_2$ in $Q_{(1)}$, we get similarly $Q^{i_2}\coprod_{Id_{\alpha_2}}(Q^{i_1}\coprod_{Id_{\alpha_1}}Q_{(1)})$, where $\alpha_2$ is the unique arrow incident with the ``new" arrow labelled by $i_2$. Using the above procedure step by step for all ``new" vertices in $Q_{(1)}$, we obtain finally a quiver, denoted as $Q_{(2)}$. Clearly,  $Q_{(1)}$ is a subquiver of $Q_{(2)}$. We call the vertices of $Q_{(1)}$ as the ``old" vertices of  $Q_{(2)}$ and the other vertices of  $Q_{(2)}$  as its ``new" vertices.
  \item Inductively, by the same procedure for obtaining $Q_{(2)}$ from $Q_{(1)}$, we can construct $Q_{(m+1)}$ from $Q_{(m)}$ via a series of gluing for any $m\geq 1$.
       Similarly, we call the vertices of $Q_{(m)}$ as the ``old" vertices of  $Q_{(m+1)}$ and the other vertices of  $Q_{(m+1)}$  as its ``new" vertices.
        \item Finally, we define the (infinite) quiver $Q=Q(B)=\bigcup_{m=1}^{+\infty}Q_m$, since $Q_m$ is always a full subquiver of $Q_{m+1}$ for any $m$.
\end{itemize}
\end{Construction}

Define $\Gamma=\{h\in AutQ:\; \text{if}\; h\cdot a_s=a_t \;\text{for}\; a_s,a_t\in Q_0,\; \text{then}\; a_s,a_t\; \text{have the same label} \}$. Trivially, $\Gamma$ is the {\em maximal } subgroup of $AutQ$ whose action on $Q$ preserves the labels of vertices of $Q$.

For $h\in \Gamma$, if there is $v\in Q_0$ such that $h\cdot v=v$, we call $v$ a {\bf fixed point} under $h$. In this case we say $h$ to have fixed points.

\begin{Example}
Let $B=\left(\begin{array}{ccc}
0  & 2  & 2\\
-2 & 0  & 1\\
-1 & -3 & 0
\end{array}\right)$. Then $Q^1=Q_{(1)}$, $Q^2$, $Q^3$ and $Q_{(2)}$ are shown in the Figure 1.
\begin{figure}[h] \centering
  \includegraphics*[0,0][253,140]{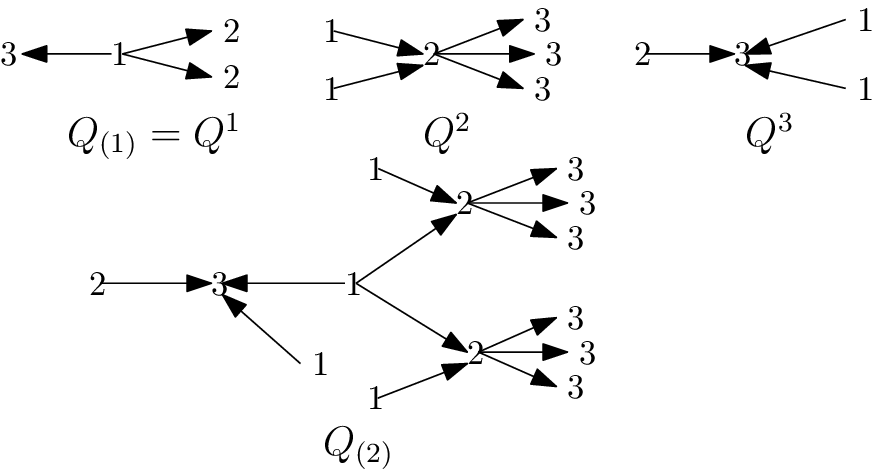}

 {\rm Figure 1 }
\end{figure}
\end{Example}

By the construction of $Q(B)$ as above, we have that
\begin{Observation}\label{extend}
(1)~ As starting of this program, the quiver $Q_{(1)}$ can be chosen to be any quiver $Q^{i}$ for any $i$ in this construction.

(2)~ The underlying graph of $Q=Q(B)$ is a tree, thus is acyclic.

(3)~ For each vertex $a\in Q_0$ labeled by $i$, the subquiver of $Q$ formed by all arrows incident with $a$ is isomorphic to $Q^i$.

(4)~ For any $m\in \mathbb N$, and any automorphism $\sigma$ of $Q_{(m)}$ which preserve the labels of vertices, then there exists $h\in \Gamma$ such that $h|_{Q_{(m)}}=\sigma$ and $h$ has the same order with $\sigma$, where $h|_{Q_{(m)}}$ means the restriction of $h$ to $Q_{(m)}$.

(5)~ In most of cases, $Q$ are infinite quivers with infinite number of vertices. However, in some special cases, $Q$ may be finite quivers. For example, when $B$ is a skew-symmetric matrix corresponding to a finite quiver $Q'$ of type $A$,  we have $Q(B)=Q'$ a finite quiver.

\end{Observation}

In the sequel of this section, all $Q$ are just the quiver $Q(B)$ defined in Construction \ref{construction}. Here we give some results on $Q$ which will be useful in the following context.

\begin{Lemma}\label{ext}
 For $a\in Q_0$, define $\Gamma_{a}:=\{h\in \Gamma\;|\;h\cdot a=a\}$. For any finite subquiver $Q'$ of $Q$, there exist $h_1,\cdots,h_s\in \Gamma$ such that $\{h|_{Q'}:\;h\in\Gamma_{a}\}=\{h_1|_{Q'},\cdots,h_s|_{Q'}\}$ for any $h\in\Gamma_{a}$, where $h|_{Q'}$ means the restriction of $h$ to $Q'$.
\end{Lemma}

\begin{proof}
Since $Q$ is strongly almost finite and $Q'$ is finite, $\{h\cdot Q'\;|\;h\in \Gamma_{a}\}$ is a finite set, written as $\{Q_{\langle 1\rangle},\cdots,Q_{\langle t\rangle}\}$. For each $i=1,\cdots,t$, let $S_i=\{f\in Iso(Q', Q_{\langle i\rangle}):\; f\; \text{preserves the labels of vertices} \}$. It is clear that $S_i$ is a finite set since $Q'$ is finite, then we write that $S_i=\{\sigma_{i1},\cdots,\sigma_{ik}\}$. Since $\sigma_{ij}\in Iso(Q',Q_{\langle i\rangle})$ preserves the labels of vertices, by Observation \ref{extend} (3),  we can lift $\sigma_{ij}$ to an isomorphism $h_{ij}\in \Gamma$ such that $h_{ij}|_{Q'}=\sigma_{ij}$. For any $h\in \Gamma_{a}$, we have $h\cdot Q'=Q_{\langle i\rangle}$ for some $i=1,\cdots,t$, then $h|_{Q'}:Q'\rightarrow Q_{\langle i\rangle}$ is a quiver isomorphism preserving the labels of vertices. Thus, $h|_{Q'}=\sigma_{ij}$ for some $j$. Therefore, we have $\{h|_{Q'}:\;h\in\Gamma_{a}\}\subseteq\{h_{ij}|_{Q'}:\;i=1,\cdots,t; j=1,\cdots,k\}$. Then the result follows.
\end{proof}

\begin{Example}
Let $B=\left(\begin{array}{ccc}
0  & 1  & 1\\
-1 & 0  & 1\\
-1 & -1 & 0
\end{array}\right)$. Then $Q^1=\xymatrix{2 & 1\ar[l]\ar[r] & 3}$, $Q^2=\xymatrix{1\ar[r] & 2\ar[r] & 3}$ and $Q^3=\xymatrix{1\ar[r] & 3 & 2\ar[l]}$.
We have $Q(B)$ to be the following quiver:
$$\xymatrix{\cdots\ar[r] &  3 & 2\ar[l] & 1\ar[l]\ar[r] & 3 & 2\ar[l] & 1\ar[l]\ar[r] &\cdots }$$
\end{Example}

\begin{Proposition}\label{stronglyalmost}
If $B$ is an acyclic sign-skew-symmetric matrix, then $Q(B)$ is a strongly almost finite quiver.
\end{Proposition}

\begin{proof}
By Observation \ref{extend} (3), for each vertex $a$ of $Q$ which is labeled by $i$, the subquiver of $Q$ formed by all arrows incident with $a$ is isomorphic to $Q^i$. Thus, $Q$ is locally finite. Moreover, for each vertex $a$ and $b$, we have $|Q(a,b)|\leq 1$ due to the construction.

 Assume that $Q$ has an infinite path $a_1a_2\cdots a_{n+1}\cdots$. Since all vertices are labeled by $1,\cdots,n$, there are  $a_s$ and $a_t$ admitting the same label, denoted as $i$. Thus, the path from $a_s$ to $a_t$ corresponds to a cycle in $\Delta(B)$. It contradicts to the fact $B$ is acyclic.
\end{proof}

\begin{Lemma}\label{fix}
Keep the forgoing notations. For any $h\in \Gamma$ and finite connected subquiver $Q'$ of $Q$ such that $h\cdot Q'=Q'$, then there exist  fixed points in $Q_0'$ under the action of $h$.
\end{Lemma}

\begin{proof}
By Observation \ref{extend} (2), the underlying graph of $Q$ is acyclic, so the underlying graph of $Q'$ is acyclic. Assume that $h$ has no fixed points. Let $a\in Q'_0$ such that the distance between $a$ and $h\cdot a$ is minimal and let  $w$ be a walk connecting $a$ and $h\cdot a$ with the minimal length, where the distance means the shortest length of walks between $a$ and $h\cdot a$. Such $a$ always exist since $Q'$ is connected.  For any $i\in \mathbb N$, we have $h^i\cdot w$ and $h^{i+1}\cdot w$ share the only common vertex $h^{i+1}\cdot a$ by the choice of $w$ and the fact that the underlying graph of $Q'$ is acyclic. Applying $h, h^2, \cdots$, we get an infinite walk $\cdots(h^2\cdot w)(h\cdot w)w$. However, since $Q'$ is a finite quiver, the infinite walk $\cdots(h^2\cdot w)(h\cdot w)w$ contains at least a cycle as graph. It contradicts to the fact the underlying graph of $Q'$ is acyclic. The result follows.
\end{proof}

\begin{Lemma}\label{finiteorder}Assume  $Q'$ is a finite connected subquiver of $Q$. If $h\in \Gamma$ has fixed points, then there exists $h'\in \Gamma$ of  finite order such that $h'\cdot a= h\cdot a$ for every vertex $a$ of $Q'$.
\end{Lemma}

\begin{proof}
By Observation \ref{extend} (1), $Q_{(1)}$ can be chosen to any $Q^i$ in the construction of $Q$, so we may assume that $h$ fixes a vertex labelled by $1$. As $Q'$ is finite, there exists an $m\in \mathbb N$ such that $Q'$ is a subquiver of the finite quiver $Q_{(m)}$. It is clear that $h|_{Q_{(m)}}$ has finite order, where $h|_{Q_{(m)}}$ means the restriction of $h$ to $Q_{(m)}$. Applying Observation \ref{extend} (4), there exists an  $h'\in \Gamma$ with finite order such that $h'|_{Q_{(m)}}=h|_{Q_{(m)}}$. Our result follows.
\end{proof}

\begin{Lemma}\label{fixedpoint}
If $h\in \Gamma$ has no fixed points, then $h^n$ has no fixed points for any $n\in \mathbb Z$.
\end{Lemma}

\begin{proof}
Assume that $h^n$ has a fixed point $x_1\in Q_0$. By the construction of $Q$, there exists an unique walk $(x_1,x_2,\cdots,x_m,h \cdot x_1)$ from $x_1$ to $h\cdot x_1$ with $x_i\neq x_j$ for $i\neq j$. Applying $h^i$ to the walk for any $i=1,\cdots, n-1$, we get a walk $(h^i\cdot x_1,\cdots,h^{i+1}\cdot x_1)$. Therefore, we obtained a walk $(x_1,x_2,\cdots,h\cdot x_1,h\cdot x_2,\cdots,h^n\cdot x_1=x_1)$. Since the underlying graph of $Q$ has no cycles, and $x_i\neq x_j$ so $h^k\cdot x_i\neq h^k\cdot x_j$ for $i\neq j$, therefore, we have $h^{k+1}\cdot x_2=h^k\cdot x_m$ or $h^{k+1}\cdot x_1=h^k\cdot x_{n-1}$ for some $k$. Thus, we have $h\cdot x_2=x_m$ or $h\cdot x_1=x_{m-1}$. In case $h\cdot x_2=x_m$, since the underlying graph of $Q$ has no cycles and $(x_1,x_2,\cdots,h\cdot x_1,h\cdot x_2,\cdots,h^n\cdot x_1=x_1)$ is a walk, we have $h\cdot x_1=x_1$, which contradicts to $h$ has no fixed points. In case $h\cdot x_1=x_{m-1}$, then $(x_1,\cdots,x_{m-2},h\cdot x_1)$ is a walk, which contradicts to $(x_1,x_2,\cdots,x_m,h\cdot x_1)$ is the unique walk from $x_1$ to $h\cdot x_1$. Therefore, our result follows.
\end{proof}

\begin{Proposition}\label{covering}
If $B\in Mat_{n\times n}(\mathbb Z)$ is an acyclic sign-skew-symmetric matrix, then $(Q(B),\Gamma)$ constructed from $B$ in Construction \ref{construction} is a covering of $B$.
\end{Proposition}

\begin{proof}
Denote $B=(b_{ij})$  and  $Q(B)=Q=(\widetilde b_{a_ia_j})$. Since the vertices of $Q$ has $n$ labels and $\Gamma$ preserves the labels, the number of orbits of vertices under the action of $\Gamma$ is $n$. It means that $B(Q)$, given in Definition \ref{unfolding}, is an $n\times n$ matrix;  denote $B(Q)=(b'_{[a_i][a_j]})$. Assume that $a_j$ is labelled by $j$, by Observation \ref{extend} (3) and the construction of $Q^{j}$, we have $b'_{[a_i][a_j]}=\sum\limits_{a'_i\in [a_i]}\widetilde b_{a'_ia_j}=b_{ij}$. Thus, we have $B(Q)=B$, which means that $(Q,\Gamma)$ is a covering of $B$.
\end{proof}

\subsection{Statement of unfolding theorem and totality of acyclic sign-skew-symmetric matrices}.

Following Proposition \ref{covering}, we give the main result about the existence of unfolding of $B$ as follows, which is indeed important for the major consequences in this paper.

\begin{Theorem}\label{mainlemma}({\bf Unfolding Theorem})
If $B\in Mat_{n\times n}(\mathbb Z)$ is an acyclic sign-skew-symmetric matrix, then $(Q(B),\Gamma)$ constructed from $B$ in Construction \ref{construction} is an unfolding of $B$.
\end{Theorem}

By Definition \ref{unfolding} and Proposition \ref{covering},  it suffices to prove that $\widetilde{\mu}_{[i_s]}\cdots\widetilde{\mu}_{[i_1]}(Q)$ has no $\Gamma$-$2$-cycles for any sequence $([i_1],\cdots,[i_s])$ of orbit mutations. However, for this aim, we will make the preparations in Sections \ref{three}, \ref{cate}, \ref{action}, \ref{quiver}, and then finish the proof of Theorem \ref{mainlemma} at the end of Section \ref{quiver}.

As a direct consequence of Theorem \ref{mainlemma}, we have

\begin{Theorem}\label{sign-skew}
Any acyclic sign-skew-symmetric matrix $B\in Mat_{n\times n}(\mathbb Z)$ is always totally sign-skew-symmetric.
\end{Theorem}

\begin{proof}
It suffices to prove that $\mu_{[i_t]}\cdots\mu_{[i_1]}(B)$ is sign-skew-symmetric for any sequence $([i_1],\cdots,[i_t])$. By Lemma \ref{basiclemma} and Theorem \ref{mainlemma}, there exists a $(Q, \Gamma)$ such that $(\widetilde\mu_{[i_t]}\widetilde\mu_{i_{[t-1]}}\cdots\widetilde\mu_{[i_1]}(Q), \Gamma)$ is a covering of $\mu_{[i_t]}\mu_{[i_{t-1}]}\cdots\mu_{[i_1]}(B)$. Thus, by the definition of covering, our result follows.
\end{proof}

According to this conclusion, we know that any acyclic sign-skew-symmetric matrix can always be an exchange matrix for a cluster algebra.

\section{Global dimensions of $K$-linear categories}\label{three}

\subsection{Terminologies in categories}.

Now we recall some standard terminologies in categories. All subcategories are assumed to be fully faithful if there is no other statement.

\begin{Definition}(p.27, \cite{K})
Let $\mathcal A$ be an additive category. A pair $(f,g)$ of composable morphisms $0\rightarrow X\overset{f}{\rightarrow}Y\overset{g}{\rightarrow}Z\rightarrow 0$ is {\bf exact} if $f$ is a kernel of $g$ and $g$ is a cokernel of $f$. Let $\mathcal E$ be a class of exact pairs closed under isomorphism and satisfies the following axioms Ex0, $Ex1$, $Ex2$ and $Ex2^{op}$. The {\bf admissible epimorphism} mentioned in these axioms are by definition the second components $g$ of $(f,g)\in \mathcal E$. The first components $f$ are {\bf admissible monomorphisms}.
\begin{description}
  \item[$Ex0$] $id_{X}$ is an admissible epimorphism for any $X\in \mathcal A$.
  \item[$Ex1$] The composition of two admissible epimorphism is an admissible epimorphism.
  \item[$Ex2$] For each $h\in Hom_{\mathcal A}(Z',Z)$ and each admissible epimorphism $g\in Hom_{\mathcal A}(Y,Z)$, there is a pullback square
      $$\xymatrix{
       Y'\ar[r]^{g'}\ar[d]_{h'} & Z'\ar[d]_{h} \\
       Y\ar[r]^g                & Z,}$$
       where $g'$ is an admissible epimorphism.
  \item[$Ex2^{op}$] For each $h\in Hom_{\mathcal A}(X',X)$ and each admissible monomorphism $f\in Hom_{\mathcal A}(X,Y)$, there is a pushout square
      $$\xymatrix{
       X\ar[r]^{f}\ar[d]_{h} & Y\ar[d]_{h'} \\
       X'\ar[r]^{f'}            &Y',}$$
       where $f'$ is an admissible monomorphism.
\end{description}
Then $(\mathcal A, \mathcal E)$ (or $\mathcal A$ for shortly) is called an {\bf exact category}.

An exact sequence in an exact category is said to be an {\bf admissible exact sequence}.
\end{Definition}

Given an exact category $\mathcal A$. An object $P$ is called {\bf projective} if $0\rightarrow Hom_{\mathcal A}(P,X)\rightarrow Hom_{\mathcal A}(P,Y)\rightarrow Hom_{\mathcal A}(P,Z)\rightarrow 0$ is exact for all admissible exact sequences $0\rightarrow X\rightarrow Y\rightarrow Z\rightarrow 0$. Dually, the {\bf injective objects} in exact category $\mathcal A$ is defined. $P$ is called {\bf projective-injective} if $P$ is both projective and injective. We say that $\mathcal A$ has {\bf enough projectives} if for each $X\in \mathcal A$, there exists an admissible short exact sequence $0\rightarrow Y\rightarrow P\rightarrow X\rightarrow 0$ with $P$ is projective and $Y\in \mathcal A$. Dually, $\mathcal A$ has {\bf enough injectives} is defined. An exact category $\mathcal A$ is called {\bf Frobenius} if it has enough project/injective objects and the set of project objects is coincide to that of injective objects. See \cite{H}. We call a subcategory $\mathcal T$ of $\mathcal C$ a {\bf co-generator subcategory} if there exists $T\in \mathcal T$ and an admissible short exact sequence $0\rightarrow X\rightarrow T\rightarrow Y\rightarrow 0$ for any $X\in \mathcal C$. Dually, we can define the {\bf generator subcategory} of $\mathcal C$.

 A $Hom$-finite exact $K$-linear category $\mathcal A$ is called {\bf $2$-Calabi-Yau} in sense that $Ext^{1}_{\mathcal A}(X,Y)\cong DExt^{1}_{\mathcal A}(Y,X)$ naturally for every pair $X,Y\in \mathcal A$, where $D$ is the $K$-linear dual. See \cite{IY}.

Recall that a $K$-linear Krull-Schmidt category $\mathcal C$ is called {\bf $Hom$-finite} if $dim_KHom_{\mathcal C}(X,Y)< \infty$ for any $X,Y\in \mathcal C$. We say a $K$-linear Krull-Schmidt category $\mathcal C$ to be {\bf strongly almost finite} if it is $Hom$-finite and there are only finite indecomposable objects $Y$ (up to isomorphism) such that $Hom_{\mathcal C}(X,Y)\neq 0$ or $Hom_{\mathcal C}(Y,X)\neq 0$ for any indecomposable object $X$.

A $K$-linear category $\mathcal C$ is called {\bf locally bounded} in \cite{BG} if it satisfies:

(a)\; each object is indecomposable,

 (b)\; distinct objects are non-isomorphic,

 (c)\; $\sum\limits_{Y\in \mathcal C}(dim_{K}Hom_{\mathcal C}(X,Y)+dim_{K}Hom_{\mathcal C}(Y,X))\leq \infty$. \\Note that the condition (a) is replaced  in \cite{BG} by that $End_{\mathcal C}(X)$ is local for $X\in \mathcal C$.

 Assume $\mathcal C$ is either a strongly almost finite category or a locally bounded category. A {\bf representation} of $\mathcal C$ over $K$ is a covariant functor $F: \mathcal C\rightarrow modK$. A representation $F$ is called {\bf representable} if it is naturally isomorphic to $Hom_{\mathcal C}(X,-)$ for some $X\in \mathcal C$; $F$ is called {\bf finitely presented} if there is an exact sequence $Hom_{\mathcal C}(X,-)\rightarrow Hom_{\mathcal C}(Y,-)\rightarrow F\rightarrow 0$ for some $X,Y\in \mathcal C$. Denote by $mod\mathcal C$ the category consisting of all finitely presented representations of $\mathcal C$. It is well known from \cite{IY} that $mod\mathcal C$ is an abelian category whose projective objects are just finite direct sum of representable representations of $\mathcal C$. Denote by $proj(mod\mathcal C)$ the projective subcategory of $mod\mathcal C$.

  When $\mathcal C$ is a strongly almost finite category, let $S_X$ be the {\bf simple representation} of $\mathcal C$ via $S_X(Y)=Hom_{\mathcal C}(X,Y)/rad(X,Y)$ for any indecomposable object $X$ of $\mathcal C$, where $rad$ is the radical of category $\mathcal C$.

 When $\mathcal C$ is a locally bounded category,  let $S_X$ be the {\bf simple representation} of $\mathcal C$ via $S_X(X)=End_{\mathcal C}(X)/J$ and $S_X(Y)=0$ if $Y\neq X$, for any object $X$ of $\mathcal C$, where $J$ is the Jacobson radical of $End_{\mathcal C}(X)$.

For $\mathcal C$ a strongly almost finite category or a locally bounded category, like as module categories of algebras, we define the {\bf projective dimension} of $S_X$, denoted as $pd.dim(S_X)$, to be that $$pd.dim(S_X)=inf\{n\;|\;0\rightarrow F_n\cdots\rightarrow F_1\rightarrow F_0\rightarrow S_X\rightarrow 0 \;\text{is exact}, F_i\in proj(mod\mathcal C)\};$$
and define the {\bf global dimension} of $\mathcal C$, denoted as $gl.dim(\mathcal C)$, to be that
$$gl.dim(\mathcal C)=sup\{pd.dim(S_X)\;|\;X\;\text{is an indecomposable object of}\;\mathcal C\}.$$

For a strongly almost finite category $\mathcal C$,  a morphism $f\in\mathcal C(X, Y)$ is called {\bf right minimal}
if it has not a direct summand of the form $T \rightarrow 0$ as a complex for some $0\neq T\in \mathcal C$. Dually, the {\bf left minimal morphism} can be
defined. For a subcategory $\mathcal D$ of $\mathcal C$, $f:X\rightarrow Y$ is called a {\bf right $\mathcal D$-approximation} of $Y\in \mathcal C$ if $X\in \mathcal D$ and $\mathcal C(-, X) \overset{\mathcal C(-, f)}{\longrightarrow} \mathcal C(-, Y) \rightarrow 0$ is exact as functors on $\mathcal D$. Moreover, we call that $f$ is a {\bf minimal right $\mathcal D$-approximation} if it is right minimal. $\mathcal D$ is said to be a {\bf contravariantly finite} subcategory of $\mathcal C$ if any $Y\in \mathcal C$ has a right $\mathcal D$-approximation. Dually, a {\bf(minimal) left $\mathcal D$-approximation} and a {\bf covariantly finite} subcategory can be defined. A contravariantly and covariantly finite subcategory is called {\bf functorially finite}. For a subcategory $\mathcal X$ of $\mathcal C$, denote $\mathcal X^{\perp}=\{T\in\mathcal C\;|\; Ext^1_{\mathcal C}(\mathcal X, T)=0\}$ and $^{\perp}\mathcal X=\{T\in\mathcal C\;|\; Ext^1_{\mathcal C}(T, \mathcal X)=0\}$;  $\mathcal X$ is called {\bf rigid} if $Ext^1_{\mathcal C}(\mathcal X,\mathcal X)=0$; and $\mathcal X$ is said to be a {\bf cluster tilting} subcategory if $\mathcal X$ is functorially finite and $\mathcal X=\mathcal X^{\perp}=$ $^{\perp}\mathcal X$. See \cite{AS}, \cite{IY}.

Assume $\mathcal C$ is a strongly almost finite category, for any subcategory $\mathcal T$ of $\mathcal C$, $add(\mathcal T)$ denotes the smallest additive full subcategory of $\mathcal C$ containing $\mathcal T$, that is, the full subcategory of $\mathcal C$ whose objects are the direct sums of direct summands of objects in $\mathcal T$. Let $\mathfrak{Add}(\mathcal C)$ denote the class of all full $K$-linear subcategories of $\mathcal C$, where the subcategories are stable under isomorphisms and direct summands. It is clear that $\mathcal T\in \mathfrak{Add}(\mathcal C)$ if and only if $\mathcal T=add(\mathcal T)$.

\subsection{Global dimension and cluster tilting subcategory}.

In this part, we will prove that a functorially finite rigid subcategory of a strongly almost finite category which has global dimension less than 3 and satisfies certain conditions is a cluster tilting subcategory.
This result generalizes Theorem 3.30 in \cite{D} and Theorem 5.1 (3) in \cite{I}.

We have the following theorem.

\begin{Theorem}\label{keylemma}
Let $\mathcal C$ be a Frobenius, $2$-Calabi-Yau strongly almost finite $K$-category. Assume that $\mathcal C$ is a extension closed subcategory of an abeian category $\mathcal D$. Let $\mathcal T$ be a functorially finite subcategory of $\mathcal C$. If $\mathcal T$ is a rigid subcategory of $\mathcal C$ such that $gl.dim(\mathcal T)\leq 3$, and $\mathcal T$ contains all projective-injective objects of $\mathcal C$, then $\mathcal T$ is a cluster-tilting subcategory of $\mathcal C$.
\end{Theorem}

\begin{proof}
For any $X\in \mathcal C$ such that $Ext^1_{\mathcal C}(X, T)=0$ for all $T\in\mathcal T$, take a projective resolution of $X$, $$\mathbf{P}:P_3\rightarrow P_2\overset{f_2}{\longrightarrow} P_1\overset{f_1}{\longrightarrow} P_0 \overset{f_0}{\longrightarrow} X\rightarrow 0.$$ Applying $Hom_{\mathcal C}(-,\mathcal T)$, we get a sequence of $\mathcal T$ representations, $$0\rightarrow Hom_{\mathcal C}(X,-)\rightarrow Hom_{\mathcal C}(P_0,-)\rightarrow Hom_{\mathcal C}(P_1,-) \rightarrow Hom_{\mathcal C}(P_2,-)\rightarrow Hom_{\mathcal C}(P_3,-).$$ Denote by $H_2$ the homology group of the above sequence at $Hom_{\mathcal C}(P_2,-)$.

We first prove that $Hom_{\mathcal C}(X,-)$ is a projective object in $mod\mathcal T$. Denote $X_1=ker(f_0), X_2=ker(f_1)$ and $X_3=ker(f_2)$, note that $X_1,X_2,X_3$ are in $\mathcal D$. Since $P_2\rightarrow X_2\rightarrow 0$, we have $0\rightarrow Hom_{\mathcal C}(X_2,-)\rightarrow Hom_{\mathcal C}(P_2,-)$. As $gl.dim(mod\mathcal T)\leq 3$ and $Hom_{\mathcal C}(P_2,-)$ is projective by $P_2\in \mathcal T$. Thus, we have $pd.dim(Hom_{\mathcal C}(X_2,-))\leq 2$. Furthermore, we have the exact sequence $$0\rightarrow Hom_{\mathcal C}(X_1,-)\rightarrow Hom_{\mathcal C}(P_1,-)\rightarrow Hom_{\mathcal C}(X_2,-) \overset{g}{\rightarrow} H_2\rightarrow 0.$$ Since
$pd.dim (Hom_{\mathcal C}(X_2,-))\leq 2$ and $pd.dim (H_2)\leq 3$ by $gl.dim (mod\mathcal T)\leq 3$, we get $pd.dim(ker(g))\leq 2$. Thus, we have $pd.dim (Hom_{\mathcal C}(X_1,-))\leq 1$. Applying $Hom_{\mathcal C}(-,\mathcal T)$ to $0\rightarrow X_1\rightarrow P_0\rightarrow X\rightarrow 0$, we obtain $$0\rightarrow Hom_{\mathcal C}(X,-)\rightarrow Hom_{\mathcal C}(P_0,-)\rightarrow Hom_{\mathcal C}(X_1,-) \rightarrow 0$$ as $Ext^1_{\mathcal C}(X,T)=0$. Therefore, $pd.dim (Hom_{\mathcal C}(X,-))= 0$ which means $Hom_{\mathcal C}(X,-)$ is a projective object in $mod\mathcal T$.

Now we show $X\in \mathcal T$. Since $Hom_{\mathcal C}(X,-)$ is a projective object in $mod\mathcal T$, there exists a $T\in\mathcal T$ such that $Hom_{\mathcal C}(X,-)$ isomorphic to $Hom_{\mathcal C}(T,-)$ as representations of $\mathcal T$. Assume $f\in Hom_{\mathcal C}(X,T)$ corresponds to $id_{T}\in Hom_{\mathcal C}(T,T)$. To prove $X\in \mathcal T$, it suffices to prove $X\overset{f}{\cong} T$. For any object $Y\in \mathcal C$, take an injective resolution of $Y$, $0\rightarrow Y\rightarrow I_0\rightarrow I_2$. Applying $Hom_{\mathcal C}(X,-)$ and $Hom_{\mathcal C}(T,-)$ on it respectively, we get the following commuting diagram,

\centerline{\xymatrix{
    0 \ar[r] & Hom_{\mathcal C}(T,Y) \ar[r]\ar[d]_{Hom_{\mathcal C}(f,Y)} & Hom_{\mathcal C}(T,I_0) \ar[r]\ar[d]_{Hom_{\mathcal C}(f,I_0)} & Hom_{\mathcal C}(T,I_1)\ar[d]_{Hom_{\mathcal C}(f,I_1)} \\
    0 \ar[r] & Hom_{\mathcal C}(X,Y) \ar[r]       & Hom_{\mathcal C}(X,I_0) \ar[r]                                 & Hom_{\mathcal C}(X,I_1).
    }}
 Since $I_0,I_1\in \mathcal T$,  $Hom_{\mathcal C}(f,I_0)$ and $Hom_{\mathcal C}(f,I_1)$ are isomorphisms. Then $Hom_{\mathcal C}(f,Y): Hom_{\mathcal C}(T,Y)\rightarrow Hom_{\mathcal C}(X,Y)$ is also an isomorphism.  It is easy to see that $Hom_{\mathcal C}(f,-): Hom_{\mathcal C}(T,-)\rightarrow Hom_{\mathcal C}(X,-)$ is a functor isomorphism over $\mathcal C$. Therefore, by Yoneda embedding Lemma, $X\overset{f}{\cong} T$.

Therefore, since $\mathcal C$ is $2$-Calabi-Yau and  $\mathcal T$ is functorially finite, we know that $\mathcal T$ is a cluster subcategory of $\mathcal C$.
\end{proof}

\subsection{Global dimension, nilpotent element and commutator subgroup}.

In this part, we will generalize Lenzing's Theorem \cite{L} to locally bounded categories.

We recall some terminologies in \cite{L} firstly. Let $\mathcal A$ be an abelian category and $\mathcal B$ be a subcategory of $\mathcal A$ which is closed under extension. $\mathcal B[X]$ is defined as the category whose objects are pairs $(E,f)$ for objects $E\in \mathcal B$ and $f\in Hom(E,E)$, whose morphisms from $(E,f)$ to $(F,g)$ are the morphisms $u:E\rightarrow F$ satisfying $gu=uf$. We call that $0\rightarrow f'\rightarrow f\rightarrow f''\rightarrow 0$ is exact in $\mathcal B[X]$ if there exists a commutative diagram:
$$\xymatrix{
    0 \ar[r] & E' \ar[r]\ar[d]_{f'} & E\ar[r]\ar[d]_f & E''\ar[r]\ar[d]_{f''}\ar[r] &0 \\
    0 \ar[r] & E' \ar[r]            & E\ar[r]         & E''\ar[r] &0,
    }$$
where $0\rightarrow E'\rightarrow E\rightarrow E''\rightarrow 0$ is exact in $\mathcal B$.

\begin{Definition}(Definition 1, \cite{L})
A picture $Tr:Obj(\mathcal B[X])\rightarrow G$ from the class of objects of $\mathcal B[X]$ to a commutative group $G$ is called {\bf trace picture} if it satisfies:

(1) If $0\rightarrow f'\rightarrow f\rightarrow f''\rightarrow 0$ is exact, then $Tr(f)=Tr(f')+Tr(f'')$,

(2) For any $f,f'\in Hom(E,E)$, we have $Tr(f+f')=Tr(f)+Tr(f')$.
\end{Definition}

Recall in \cite{L} that $f:E\rightarrow E\in \mathcal B$ is called {\bf $\mathcal B$-nilpotent} if there exists a filtration $E=E_0\supseteq E_1\supseteq\cdots \supseteq E_s=0$ such that $E_i, E_i/E_{i+1}$ are objects of $\mathcal B$ and $f(E_i)\subseteq E_{i+1}$ for $i=0,\cdots, s$.

\begin{Lemma}\label{zero}(Satz 1, \cite{L})
If $Tr:Obj(\mathcal B[X])\rightarrow G$ is a trace picture and $f:E\rightarrow E$ is $\mathcal B$-nilpotent, then $Tr(f)=0$.
\end{Lemma}
Let $\mathcal D$ be a locally bounded category. Let $V=\bigoplus\limits_{X,Y\in \mathcal D}Hom_{\mathcal D}(X,Y)$ and $[V,V]$ be the $K$-subspace of $V$ generated by $fg-gf$ for all $f,g\in V$. Let $\mathcal P=add(\{Hom_{\mathcal D}(X,-)\;|\;X\in \mathcal D\})$ be the subcategory of $mod\mathcal D$. Then $\mathcal P$ contains all projective objects in $mod\mathcal D$.

For any $X\in \mathcal D$ and $P_X=Hom_{\mathcal D}(X,-)\in \mathcal P(\mathcal D)$, by Yoneda Lemma, we have a natural isomorphism $End(P_X)\overset{\varphi}{\cong} End_{\mathcal D}(X)$. For $f\in End(P_X)$, define $Tr(f)=\varphi(f)+[V,V]\in V/[V,V]$.

\begin{Lemma}\label{com}
(1) Let $X_i\in \mathcal D (i=1,\cdots, n)$ be pairwise non-isomorphic objects and $P=\bigoplus\limits_{i=1}^n P_{X_i}$. If $g=(g_{ij})\in End(P)$, then $Tr(g)=\sum\limits_{i=1}^n Tr(g_{ii})$.

(2) Let $P=\bigoplus\limits_{i=1}^m P_i, P'=\bigoplus\limits_{i=1}^n P'_i$ be objects in $\mathcal P$ and $f=(f_{ij})_{i\in [1,n],j\in [1,m]}:P\rightarrow P'$, $f'=(g'_{ij})_{i\in [1,m],j\in [1,n]}:P'\rightarrow P$ be morphisms. Then $Tr(ff')=Tr(f'f)$.
\end{Lemma}

\begin{proof}
(1) For any $1\leq s,t\leq n$, let $f^{s,t}=(f_{ij})\in End(P)$ with $f_{st}=g_{st}$ and $f_{ij}=0$ for all $(i,j)\neq (s,t)$. When $s\neq t$, we have $\varphi(f_{st})=id_{X_t}\varphi(f_{st})-\varphi(f_{st})id_{X_t}\in [V,V]$, so $Tr(f^{s,t})=Tr(f_{st})=0$. Thus,
 $\begin{array}{ccl} Tr(f^{s,t}) &=&
         \left\{\begin{array}{ll}
             Tr(f_{ss}), &\mbox{if $s=t$}, \\
             0, &\mbox{otherwise.}
         \end{array}\right.
      \end{array}$ Therefore, the result follows.

(2) According to (1), we have $Tr(ff')=\sum\limits_{i=1}^n\sum\limits_{j=1}^mTr(a_{ji}a'_{ij})$ and $Tr(f'f)=\sum\limits_{j=1}^m\sum\limits_{i=1}^nTr(a'_{ij}a_{ji})$. Moreover, as $a_{ji}a'_{ij}-a'_{ij}a_{ji}\in [V,V]$, we have $Tr(ff')=Tr(f'f)$.
\end{proof}

\begin{Corollary}\label{cong}
Let $P=\bigoplus\limits_{i=1}^m P_i, P'=\bigoplus\limits_{i=1}^m P'_i$ be objects in $\mathcal P$ with morphisms $f=(a_{ij})_{i\in [1,m],j\in [1,m]}:P\rightarrow P$, $f'=(a'_{ij})_{i\in [1,m],j\in [1,m]}:P'\rightarrow P'$. If there exists an isomorphism $g=(b_{ij})_{i\in [1,m],j\in [1,m]}:P\rightarrow P'$ such that $f'g=gf$, then $Tr(f)=Tr(f')$.
\end{Corollary}

\begin{proof}
Since $f'g=gf$ and $g$ is isomorphic, we have $f=g^{-1}f'g$. Therefore, by Lemma \ref{com}, $Tr(f)=Tr(g^{-1}f'g)=Tr(g^{-1}gf')=Tr(f')$.
\end{proof}

After the above preparations, we can prove that $Tr$ is a trace picture.

\begin{Theorem}
$Tr$ is a trace picture from $Obj(\mathcal P[X])$ to $V/[V,V]$.
\end{Theorem}

\begin{proof}
(1) For $f,f'\in End(P)$, we have $Tr(f+f')=\varphi(f+f')=\varphi(f)+\varphi(f')=Tr(f)+Tr(f')$.

(2) For $0\rightarrow f'\rightarrow f\rightarrow f''\rightarrow 0$, by the definition, we have the commutative diagram:
$$\xymatrix{
    0 \ar[r] & P' \ar[r]^{x}\ar[d]_{f'} & P\ar[r]^y\ar[d]_f & P''\ar[r]\ar[d]_{f''} &0 \\
    0 \ar[r] & P' \ar[r]^x            & P\ar[r]^y         & P''\ar[r] &0.
    }$$
Since $P''$ is projective, $0\rightarrow P'\rightarrow P\rightarrow P''\rightarrow 0$ splits. Thus, there exist $c\in Hom(P,P'),d\in Hom(P'',P)$ such that $cx=id_{P'}$, $yd=id_{P''}$ and $(x,d):P'\oplus P''\rightarrow P$ is isomorphic with inverse $(c,y)^T$. Therefore, the following diagram is commutative:
$$\xymatrix{
    0 \ar[r] & P' \ar[r]^{(id,0)^T}\ar[d]_{id} & P'\oplus P'' \ar[r]^{(0,id)}\ar[d]_{(x,d)} & P''\ar[r]\ar[d]_{id} &0 \\
    0 \ar[r] & P' \ar[r]^{x}\ar[d]_{f'}        & P\ar[r]^y\ar[d]_{f}                              & P''\ar[r]\ar[d]_{f''} &0 \\
    0 \ar[r] & P' \ar[r]^x\ar[d]_{id}          & P\ar[r]^y\ar[d]_{(c,y)^T}                        & P''\ar[r]\ar[d]_{id} &0\\
    0 \ar[r] & P' \ar[r]^{(id,0)^T}            & P'\oplus P'' \ar[r]^{(0,id)}               & P''\ar[r] &0. }$$
Thus, by Lemma \ref{com} and Corollary \ref{cong}, $$Tr(f)=Tr((c,y)^Tf(x,d))=Tr(\left(\begin{array}{cc}
cfx & cfd \\
yfx & yfd
\end{array}\right))=Tr(cfx)+Tr(yfd)=Tr(f')+Tr(f'').$$
\end{proof}

\begin{Theorem}
Keeps the forgoing notations. If $\mathcal D$ has finite global dimension, then $Tr:Obj(\mathcal P[X])\rightarrow V/[V,V]$ can be extended to a trace picture $Tr:Obj(mod\mathcal D[X])\rightarrow V/[V,V]$.
\end{Theorem}

\begin{proof}
The proof is the same as that of Satz 4 in \cite{L} since $\mathcal D$ has finite global dimension.
\end{proof}

\begin{Theorem}\label{nil}
Keeps the notations as above. Let $\mathcal D$ be a locally bounded category. Assume that $\mathcal D$ has finite global dimension. If $a\in V$ is nilpotent, then $a\in [V,V]$.
\end{Theorem}

\begin{proof}
We follow the strategy of the proof of Satz 5 in \cite{L}. In detail, on one hand, by the definition, we have $Tr(\varphi^{-1}(a))=a+[V,V]$. On the other hand, we have the following filtration $P_X\supseteq \varphi^{-1}(a)P_X\supseteq\cdots\supseteq \varphi^{-1}(a^m)=0$ for some $m$ since $a$ is nilpotent. Thus, $\varphi^{-1}(a)$ is $mod\mathcal D$-nilpotent, by Lemma \ref{zero}, we have $Tr(\varphi^{-1}(a))=0$. Therefore, we have $a+[V,V]=0$, equivalently, $a\in [V,V]$.
\end{proof}

\begin{Remark}
If $\mathcal D$ has finite objects, then Theorem \ref{nil} is Satz 5 of \cite{L}.
\end{Remark}

\section{The category $\mathcal C^Q$ arising from the preprojective algebra $\Lambda_{Q}$}\label{cate}

In this section, we will use the colimit of categories to obtain a Frobenius $2$-Calabi-Yau strongly almost finite category $\mathcal C^Q$ from a strongly almost finite quiver $Q$. This is available for the sequel. It may be regarded as a generalization of the results in \cite{GLS}.

Note that a strongly almost finite quiver $Q$ is always acyclic since it is interval finite.

\subsection{Colimit of categories}.

We study the existence and the properties of colimit of categories in this subsection.

\begin{Definition}(\cite{W}, Definition 2.6.13)
Let $\mathcal A_i,i\in \mathbb N$ be additive categories. Assume that $F^i_j:\mathcal A_i\rightarrow \mathcal A_j$ are additive functors for all $i\leq j$ satisfying (a) $F_i^i=id_{\mathcal A_i}$ the identity functor for all $i\in \mathbb N$; (b) $F_j^kF_i^j=F_i^k$ for all $i\leq j\leq k$.

We say an additive category $\mathcal A$ to be the {\bf colimit} of $(\mathcal A_i,F_i^j)$ if

(1) there exist additive functors $F_i:\mathcal A_i\rightarrow \mathcal A$ for all $i\in \mathbb N$ and natural equivalence $\eta_i^j:F_i\rightarrow F_jF_i^j$ for all $i\leq j$ satisfying $\eta_i^i=id_{F_i}$, and
$$\xymatrix{
       F_i\ar[r]^{\eta_i^j}\ar[d]_{id_{F_i}} & F_jF_i^j\ar[d]^{\eta_j^k(F_i^j)} \\
       F_i\ar[r]^{\eta_i^k}                  & F_kF_i^k,}$$
are commutative for $i\leq j\leq k$;

(2) for any additive category $\mathcal B$ with additive functors $G_i:\mathcal A_i\rightarrow \mathcal B$ and natural equivalence $\varphi_i^j:G_i\rightarrow G_jF_i^j$ for all $i\leq j$ satisfying $\varphi_i^i=id_{G_i}$, and the commutative diagrams:
$$\xymatrix{
       G_i\ar[r]^{\varphi_i^j}\ar[d]_{id_{G_i}} & G_jF_i^j\ar[d]^{\varphi_j^k(F_i^j)} \\
       G_i\ar[r]^{\varphi_i^k}                  & G_kF_i^k,}$$
 for $i\leq j\leq k$,  there exists an unique additive functor (up to natural equivalence) $\xi:\mathcal A\rightarrow \mathcal B$ such that $G_i=\xi F_i$  for all $i\in \mathbb N$ and $\varphi_i^j=\xi \eta_i^j$ for all $i\leq j$.

In this case, we denote $\lim\limits_{\rightarrow}\mathcal A_i=\mathcal A$ or  $\lim\limits_{\rightarrow}\mathcal A_i=(\mathcal A,F_i)$.
\end{Definition}

\begin{Lemma}\label{colim}
Let $\mathcal C_i$, $i\in \mathbb N$ be Frobenius categories and $F_i^j:
\mathcal C_i\rightarrow \mathcal C_j$ be exact functors for $i\leq j$, which satisfies:

 (a) $Hom_{\mathcal C_i}(X,Y) \overset{F_i^j}{\cong} Hom_{\mathcal C_j}(F_i^j(X),F_i^j(Y))$ are naturally isomorphic for all $i\leq j$ and $X,Y\in \mathcal C_i$,

  (b) $F_i^i=id_{\mathcal C_i}$ for all $i\in \mathbb N$, and

   (c) $F_i^k=F_j^kF_i^j$ for all $i\leq j\leq k$.\\ Then,

(1)~ $\mathcal C=\lim\limits_{\rightarrow}\mathcal C_i$ exists;

(2)~ if $\mathcal C$ has enough projective and injective objects, then $\mathcal C$ is a Frobenius category.
\end{Lemma}

\begin{proof}
(1) Firstly, we construct an additive category $\mathcal C$ from $\mathcal C_i$, $i\in \mathbb N$.

The objects of $\mathcal C$ consists of all of the forms $[X_i]$ for any $X_i\in \mathcal C_i$, where it means $X_i\in \mathcal C_i$ when we write $[X_i]\in \mathcal C$.

Now we define the homomorphisms and their compositions. For any two objects $[X_i], [Y_j]\in \mathcal C$,  We define $$Hom_{\mathcal C}([X_i],[Y_j]):=Hom_{\mathcal C_{max\{i,j\}}}(F_i^{max\{i,j\}}(X_i),F_j^{max\{i,j\}}(Y_j)).$$ Trivially,  $f:[X_i]\rightarrow [Y_i]$ means just the morphism $f: X_i\rightarrow Y_i$.

For $f\in Hom_{\mathcal C}([X_i],[Y_j])$ and $g\in Hom_{\mathcal C}([Y_j],[Z_k])$, let $l=max\{i,j,k\}$. Define
  \begin{equation}\label{composition}
 g \circ f:=F^l_{max\{j,k\}}(g)F^l_{max\{i,j\}}(f).
  \end{equation}
  Then, $g\circ f\in Hom_{\mathcal C_l}(F_i^l(X_i),F_k^l(Z_k))\overset{\text{by} (a)}{=}Hom_{\mathcal C_{max\{i,k\}}}(F_i^{max\{i,k\}}(X_i),F_k^{max\{i,k\}}(Z_k))=Hom_{\mathcal C}([X_i],[Z_k])$.

For any $f\in Hom_{\mathcal C}([X_i],[Y_j])$, $g\in Hom_{\mathcal C}([Y_j],[Z_k])$, $h\in Hom_{\mathcal C}([Z_k],[U_m])$, we have
\[\begin{array}{ccl}
h\circ (g\circ f)&=&h\circ(F^l_{j,k}(g)F^l_{i,j}(f))=F^{max\{l,m\}}_{max\{k,m\}}(h)F^{max\{l,m\}}_l(F^l_{max\{j,k\}}(g)F^l_{max\{i,j\}}(f))\\
&=&F^{max\{i,j,k,m\}}_{max\{k,m\}}(h)F^{max\{i,j,k,m\}}_{max\{j,k\}}(g)F^{max\{i,j,k,m\}}_{max\{i,j\}}(f).
\end{array}\]
Similarly, $(h\circ g)\circ f=F^{max\{i,j,k,m\}}_{max\{k,m\}}(h)F^{max\{i,j,k,m\}}_{max\{j,k\}}(g)F^{max\{i,j,k,m\}}_{max\{i,j\}}(f).$
Hence, $h\circ (g\circ f)=(h\circ g)\circ f$ always holds.

 By the definition of composition, it is easy to see the composition is bilinear. Thus, $\mathcal C$ is an additive category.

Now we claim $\mathcal C=\lim\limits_{\rightarrow}\mathcal C_i$.

For any $i\in \mathbb N$, define a functor $F_i:\mathcal C_i\rightarrow \mathcal C$ like this: $F_i(X_i)=[X_i]$ for all $X_i\in \mathcal C_i$ and $F_i(f)=f$ for all $f:X_i\rightarrow Y_i$. It is clear that $F_i$ is an additive functor. For any $f:X_i\rightarrow Y_i$ in $\mathcal C_i$ and $j\geq i$, we have an isomorphism $id_{[F_i^j(X_i)]}:[X_i]\rightarrow [F_i^j(X_i)]$. Moreover, it is easy to verify the  commutative diagram:
$$\xymatrix{
       [X_i]\ar[r]^{id_{[F_i^j(X_i)]}}\ar[d]_{f}    & [F_i^j(X_i)]\ar[d]_{F_i^j(f)} \\
       [Y_i]\ar[r]^{id_{[F_i^j(Y_i)]}}              & [F_i^j(Y_i)].}$$
Thus, we have the natural isomorphism of  functors $\eta_i^j:F_i\rightarrow F_jF_i^j$ such that $(\eta_i^j)_{X_i}=id_{[F_i^j(X_i)]}$ for each $X_i\in \mathcal C_i$. Then $(\eta_i^i)_{X_i}=id_{[X_i]}$ for all $i$ and $X_i\in \mathcal C_i$. It is easy to see that $(\eta_j^k)_{F_i^j(X_i)}(\eta_i^j)_{X_i}=(\eta_i^k)_{X_i}$ for all $i\leq j\leq k$ and $X_i\in \mathcal C_i$.

Assume for an additive category $\mathcal D$, there are functors $G_i:\mathcal C_i\rightarrow \mathcal D$ and natural equivalences $\delta_i^j:G_i\rightarrow G_jF_i^j$ for $i\leq j$, satisfying $(\delta_i^i)_{X_i}=id_{G_i(X_i)}$  and $(\delta_j^k)_{F_i^j(X_i)}(\delta_i^j)_{X_i}=(\delta_i^k)_{X_i}$  for all $i\leq j\leq k$ and $X_i\in \mathcal C_i$. We need to find an unique functor $H:\mathcal C\rightarrow \mathcal D$ such that
$G_i=HF_i$ for all $i\in \mathbb N$ and $\delta_i^j=H\eta_i^j$ for all $i\leq j$.

We define the functor $H:\mathcal C\rightarrow \mathcal D$ satisfying $H([X_i]):=G_i(X_i)$ for all $X_i\in \mathcal C_i$ and for all morphisms $f:[X_i]\rightarrow [Y_j]$, defining that
 $ H(f):=
  \left\{\begin{array}{lll} G_j(f)(\delta_i^j)_{X_i}, &\text{if}\; i\leq j, \\
(\delta_j^i)_{Y_j}^{-1}G_i(f), &\text{otherwise.}
\end{array}\right.$

Then $H(id_{[X_i]})=G_i(id_{X_i})=id_{G_i(X_i)}=id_{H([X_i])}$. For $f:[X_i]\rightarrow [Y_j]$ and $g:[Y_j]\rightarrow [Z_k]$, by (\ref{composition}), we have  $g \circ f:=F^l_{max\{j,k\}}(g)F^l_{max\{i,j\}}(f)$ for $l=max\{i,j,k\}$.
Since $\delta_{max\{i,j\}}^l$ is a natural equivalence, we obtain $$(\delta_{max\{i,j\}}^l)_{F^{max\{i,j\}}_j(Y_j)}G_{max\{i,j\}}(f)=(G_lF_{max\{i,j\}}^l)(f)(\delta_{max\{i,j\}}^l)_{F^{max\{i,j\}}_i(X_i)}.$$

If $i\leq j\leq k$, we have
$$H(g\circ f)=H(gF_j^k(f))=G_k(gF_j^k(f))(\delta_i^k)_{X_i}=G_k(g)G_k(F_j^k(f))(\delta_i^k)_{X_i};$$
\[\begin{array}{ccl} H(g)H(f) & = & (G_k(g)(\delta_j^k)_{Y_j})(G_j(f)(\delta_i^j)_{X_i})
  =   G_k(g)G_k(F_j^k(f))(\delta_j^k)_{F_i^j(X_i)}(\delta_i^j)_{X_i}\\
 & = &  G_k(g)G_k(F_j^k(f))(\delta_i^k)_{X_i}=H(g\circ f).
\end{array}\]

If $i\leq k\leq j$, we have $$H(g\circ f)=H((F_k^j)^{-1}(gf))=G_k((F_k^j)^{-1}(gf))(\delta_i^k)_{X_i};$$
$$H(g)H(f)  =  ((\delta_k^j)^{-1}_{Z_k}G_j(g))(G_j(f)(\delta_i^j)_{X_i})
  = (\delta_k^j)^{-1}_{Z_k}G_j(gf)(\delta_k^j)_{F_i^k(X_i)}(\delta_i^k)_{X_i},$$
Since $\delta_k^j$ is a natural transformation, we have $G_j(gf)(\delta_k^j)_{F_i^k(X_i)}=(\delta_k^j)_{Z_k}G_k((F_k^j)^{-1}(gf))$. Thus $H(g\circ f)=H(g)H(f)$.

In the other cases, we can similarly prove $H(g\circ f)=H(g)H(f)$. Thus, $H$ is a functor.

For any $X_i,Y_i\in \mathcal C_i$ and $f:X_i\rightarrow Y_i$, we have $HF_i(X_i)=H([X_i])=G_i(X_i)$ and $HF_i(f)=H(f)=G_i(f)(\delta_i^i)_{X_i}=G_i(f)$, then $G_i=HF_i$ for all $i\in \mathbb N$.

If there is another functor $H_1: \mathcal C\rightarrow \mathcal D$ satisfying $G_i=H_1F_i$ for all $i\in \mathbb N$, then
 $H_1([X_i])=H_1F_i(X_i)=G_i(X_i)=H([X_i])$ for each $[X_i]\in \mathcal C$.

 Let $f:[X_i]\rightarrow [Y_j]\in \mathcal C$. Since $H_1\eta_s^t=\delta_s^t$, $(H_1\eta_s^t)_{X_s}=H_1(id_{[F_s^t(X_s)]})=(\delta_s^t)_{X_s}$ for $s\leq t$. Then $H_1(id^{-1}_{[F_s^t(Y_s)]})=(\delta_s^t)^{-1}_{Y_s}$ for $s\leq t$. If $i\leq j$, then $f: [X_i]\overset{id_{[F_i^j(X_i)]}}{\rightarrow} [F_i^j(X_i)]\overset{f}{\rightarrow} [Y_j]$.  Thus, $$H_1(f)=H_1(fid_{[F_i^j(X_i)]})=H_1(f)H_1(id_{[F_i^j(X_i)]})=G_j(f)(\delta_i^j)_{X_i}=H(f).$$ Similarly, if $i\geq j$, $f:[X_i]\overset{f}{\rightarrow} [F_j^i(Y_j)]\overset{id^{-1}_{[F_j^i(Y_j)]}}{\rightarrow}[Y_j]$, then $H_1(f)=(\delta_i^j)^{-1}_{Y_j}G_i(f)=H(f)$. Hence $H$ are uniquely determined.

(2) Define the admissible short exact sequences in $\mathcal C$ are sequences isomorphic to the form $0\rightarrow [X_i]\overset{f}{\rightarrow} [Y_i]\overset{g}{\rightarrow} [Z_i]\rightarrow 0$, where $0\rightarrow X_i\overset{f}{\rightarrow} Y_i\overset{g}{\rightarrow} Z_i\rightarrow 0$ is an admissible short exact sequence in $\mathcal C_i$. Now we prove that it gives an exact structure on $\mathcal C$.

For admissible short exact sequence $0\rightarrow [X_i]\overset{f}{\rightarrow} [Y_i]\overset{g}{\rightarrow} [Z_i]\rightarrow 0$ and $h:[Y_i]\rightarrow [C_j]$ such that $h\circ f=0$. According to construction, $0\rightarrow X_i\overset{f}{\rightarrow} Y_i\overset{g}{\rightarrow} Z_i\rightarrow 0$ is an admissible short exact sequence in $\mathcal C_i$. We have $h: F_i^{max\{i,j\}}(Y_i)\rightarrow F_j^{max\{i,j\}}(C_j)$ and $h\circ f=hF_i^{max\{i,j\}}(f)=0$. Since $F_i^{max\{i,j\}}$ is exact, $0\rightarrow F_i^{max\{i,j\}}(X_i)\overset{F_i^{max\{i,j\}}(f)}{\longrightarrow} F_i^{max\{i,j\}}(Y_i)\overset{F_i^{max\{i,j\}}(g)}{\longrightarrow} F_i^{max\{i,j\}}(Z_i)\rightarrow 0$ is an admissible short exact sequence in $\mathcal C_{max\{i,j\}}$, and as $\mathcal C_{max\{i,j\}}$ is exact, we have $h=h'F_i^{max\{i,j\}}(g)$ for $h':F_i^{max\{i,j\}}(Z_i)\rightarrow F_j^{max\{i,j\}}(C_j)$. Thus, $h=h'\circ g$ in $\mathcal C$. Therefore, $g$ is the cokernel of $f$ in $\mathcal C$. Dually, $f$ is the kernel of $g$ in $\mathcal C$.

$Ex0$ follows immediately by $\mathcal C_i$ are exact.

For $Ex1$, let $f_1:[X_i]\rightarrow [Y_i]$ and $f_2:[Y_i]\rightarrow [Z_j]$ be admissible epimorphisms. We have $f_1: X_i\rightarrow Y_i$, $f_2:F_i^{max\{i,j\}}(Y_i)\rightarrow F_j^{max\{i,j\}}Z_j$ are admissible epimorphisms in $\mathcal C_i$ and $\mathcal C_{max\{i,j\}}$, and $f_2\circ f_1=f_2F_i^{max\{i,j\}}(f_1)$. As $F_i^{max\{i,j\}}$ is exact, then $F_i^{max\{i,j\}}(f_1)$ is an admissible epimorphism in $\mathcal C_{max\{i,j\}}$. Thus, $f_2F_i^{max\{i,j\}}(f_1)$ is an admissible epimorphism in $\mathcal C_{max\{i,j\}}$. Therefore, $f_2\circ f_1$ is admissible epimorphism in $\mathcal C$.

For $Ex2$, for each $h: [Z'_j]\rightarrow [Z_i]$ and each admissible epimorphism $g: [Y_i]\rightarrow [Z_i]$. We have $h:F_j^{max\{i,j\}}Z'_j\rightarrow F_i^{max\{i,j\}}(Z_i)$ and $g:Y_i\rightarrow Z_i$ is an admissible epimorphism in $\mathcal C_i$, thus $F_i^{max\{i,j\}}(g):F_i^{max\{i,j\}}(Y_i)\rightarrow F_i^{max\{i,j\}}(Z_i)$ is an admissible epimorphism in $\mathcal C_{max\{i,j\}}$. Since $\mathcal C_{max\{i,j\}}$ is exact, there exists the following pullback square
     $$\xymatrix{
       Y'_{max\{i,j\}}\ar[r]^{g'}\ar[d]_{h'}    & F_j^{max\{i,j\}}(Z'_j)\ar[d]_{h} \\
       F_i^{max\{i,j\}}(Y_i)\ar[r]^{F_i^{max\{i,j\}}(g)}   & F_i^{max\{i,j\}}(Z_i)}$$
in $\mathcal C_{max\{i,j\}}$ such that $h'$ is an admissible epimorphism. Thus, we have the following pullback square
     $$\xymatrix{
       [Y'_{max\{i,j\}}]\ar[r]^{g'}\ar[d]_{h'}    & [Z'_j]\ar[d]_{h} \\
       [Y_i]\ar[r]^{g}                 & [Z_i]}$$
in $\mathcal C$ such that $h'$ is an admissible epimorphism in $\mathcal C$. Thus, $Ex2$ also holds.

Dually, we can prove $Ex2^{op}$.

Therefore, $\mathcal C$ is an exact category.

Finally, we claim that $\mathcal C$ is Frobenius.
We show that for $P_i\in \mathcal C_i$, $[P_i]\in \mathcal C$  is projective if and only if $F_i^j(P_i)$ is projective in $\mathcal C_j$ for all $j\geq i$.

 Assume $F_i^j(P_i)$ is projective in $\mathcal C_j$ for all $j\geq i$. For any admissible short exact sequence $0\rightarrow [X_j]\overset{f}{\rightarrow} [Y_j]\overset{g}{\rightarrow} [Z_j]\rightarrow 0$, then $0\rightarrow X_j\overset{f}{\rightarrow} Y_j\overset{g}{\rightarrow} Z_j\rightarrow 0$ is an admissible short exact sequence in $\mathcal C_j$. Since $F_j^{max\{i,j\}}$ is exact, we have $0\rightarrow F_j^{max\{i,j\}}(X_j)\rightarrow F_j^{max\{i,j\}}(Y_j)\rightarrow F_j^{max\{i,j\}}(Z_j)\rightarrow 0$ is an admissible short exact sequence in $\mathcal C_{max\{i,j\}}$. As $F_i^{max\{i,j\}}(P_i)$ is projective in $\mathcal C_{max\{i,j\}}$, so $Hom_{\mathcal C_{max\{i,j\}}}(F_i^{max\{i,j\}}(P_i),F_j^{max\{i,j\}}(Y_j))\rightarrow Hom_{\mathcal C_{max\{i,j\}}}(F_i^{max\{i,j\}}(P_i),F_j^{max\{i,j\}}(Z_j))\rightarrow 0$ is exact. Thus, $$Hom_{\mathcal C}([P_i],[Y_j])\rightarrow Hom_{\mathcal C}([P_i],[Z_j])\rightarrow 0$$ is exact. Therefore, $[P_i]$ is projective in $\mathcal C$.

 Conversely, assume $[P_i]$ is projective. For any $j\geq i$ and admissible short exact sequence $0\rightarrow X_j\overset{f}{\rightarrow} Y_j\overset{g}{\rightarrow} Z_j\rightarrow 0$ in $\mathcal C_j$, then $0\rightarrow [X_j]\overset{f}{\rightarrow} [Y_j]\overset{g}{\rightarrow} [Z_j]\rightarrow 0$ is an admissible short exact sequence in $\mathcal C$. As $[P_i]$ is projective, so $Hom_{\mathcal C}([P_i],[Y_j])\rightarrow Hom_{\mathcal C}([P_i],[Z_j])\rightarrow 0$ is exact. Thus, $Hom_{\mathcal C_j}(F_i^j(P_i),Y_j)\rightarrow Hom_{\mathcal C_j}(F_i^j(P_i),Z_j)\rightarrow 0$ is exact. Therefore, $F_i^j(P_i)$ is projective in $\mathcal C_j$.

Dually, for $I_i\in \mathcal C_i$,  $[I_i]\in \mathcal C$ is injective if and only if $F_i^j(I_i)$ is injective in $\mathcal C_j$ for all $j\geq i$.

Furthermore, since $\mathcal C_i$ are Frobenius for all $i$, it follows that $\mathcal C$ is Frobenius.
\end{proof}

We call the family $\{F_i^j\}_{i\in \mathbb N}$ satisfying the conditions of this lemma a {\bf system of exact functors} for $\mathcal C_i$, $i\in \mathbb N$.

In this proof of Lemma \ref{colim}, for any $i\in \mathbb N$, the additive functor $F_i:\mathcal C_i\rightarrow \mathcal C$ satisfying $F_i(X_i)=[X_i]$ for all $X_i\in \mathcal C_i$ and $F_i(f)=f$ for all $f:X_i\rightarrow Y_i$, is called the {\bf embedding functor} from $\mathcal C_i$ to $\mathcal C$.

\begin{Remark}
(1) In Lemma \ref{colim}, the condition (a) ensures that $F_i^j(X)\cong F_i^j(Y)$ if and only if $X\cong Y$. More precisely, if $F_i^j(X)\cong F_i^j(Y)$, by (a), then  for any $M\in \mathcal C_i$, we have the natural isomorphism $Hom_{\mathcal C_i}(X,M)\cong Hom_{\mathcal C_i}(Y,M)$. Then by Yoneda Lemma, we have $X\cong Y$.

(2) For every $i\leq j$, since $F_i^j$ is exact and additive, we have a group homomorphism $F_i^j:Ext^1_{\mathcal C_i}(X,Y)\rightarrow Ext^1_{\mathcal C_j}(F_i^j(X),F_i^j(Y))$.

(3) If $\mathcal C_i$ are abelian categories satisfying the same conditions, then we can also prove that $\lim\limits_{\rightarrow}\mathcal C_i$ exists and is an abelian category.
\end{Remark}

In the above remark, if $F_i^j$ are isomorphisms and $\mathcal C_i$ are $2$-Calabi-Yau categories, we say that $F_i^j:\mathcal C_i\rightarrow \mathcal C_j$ is {\bf compatible} with the $2$-Calabi-Yau structures if the following diagram is commutative,
      $$\xymatrix{
       Ext^1_{\mathcal C_i}(X_i,Y_i)\ar[r]^{\backsimeq}\ar[d]_{F_i^j} & DExt^1_{\mathcal C_i}(Y_i,X_i) \\
       Ext^1_{\mathcal C_j}(F_i^j(X_i),F_i^j(Y_i))\ar[r]^{\backsimeq}            &DExt^1_{\mathcal C_j}(F_i^j(Y_i),F_i^j(X_i))\ar[u]_{DF_i^j}.}$$

\begin{Lemma}\label{2-C-Y}
Keeps the condition of Lemma \ref{colim}. Assume that for all $i\leq j$,  $Ext^1_{\mathcal C_i}(X,Y)\overset{F_i^j}{\cong} Ext^1_{\mathcal C_j}(F_i^j(X),F_i^j(Y))$ and $\mathcal C_i$ are $2$-Calabi-Yau categories. If $F_i^j:\mathcal C_i\rightarrow \mathcal C_j$ are compatible with the $2$-Calabi-Yau structures, then $\mathcal C$ is a $2$-Calabi-Yau category.
\end{Lemma}

\begin{proof}
Let $[X_i],[Y_j]\in \mathcal C$. For any $\xi:0\rightarrow [Y_j]\rightarrow [Z_k]\rightarrow [X_i]\rightarrow 0\in Ext^1_{\mathcal C}([X_i],[Y_j])$ and $\eta:0\rightarrow [X_i]\rightarrow [Z'_s]\rightarrow [Y_j]\rightarrow 0\in Ext^1_{\mathcal C}([Y_j],[X_i])$. According to the construction of $\mathcal C$, for any $t\geq max\{i,j,k,s\}$, $\widetilde\xi:0\rightarrow F_j^t(Y_j)\rightarrow F_k^t(Z_k)\rightarrow F_i^t(X_i)\rightarrow 0\in Ext^1_{\mathcal C_t}(F_i^t(X_i),T_j^t(Y_j))$ and $\widetilde\eta:0\rightarrow F_i^t(X_i)\rightarrow F_s^t(Z'_s)\rightarrow F_j^t(Y_j)\rightarrow 0\in Ext^1_{\mathcal C_t}(F_j^t(Y_j),F_i^t(X_i))$. Sine $\mathcal C_t$ is $2$-Calabi-Yau, there is a non-degenerate bilinear form $\gamma_t:Ext_{\mathcal C_t}^1(X,Y)\times Ext^1_{\mathcal C_t}(Y,X)\rightarrow K$.

Define a bilinear map $$\gamma:Ext^1_{\mathcal C}([X_i],[Y_j])\times Ext^1_{\mathcal C}([Y_j],[X_i])\rightarrow K, (\xi,\eta)\rightarrow\gamma_t(\widetilde\xi,\widetilde\eta).$$

Firstly, we show that $\gamma$ is well-defined. Assume that $t'$ is the other integer such that $t'\geq max\{i,j,k,s\}$. Without loss of generality, we may assume that $t'\geq t$. Then,
$$F_t^{t'}(\widetilde\xi)\in Ext^1_{\mathcal C_{t'}}(F_i^{t'}(X_i),F_j^{t'}(Y_j))\;\;\;\;\;\; \text{and}\;\;\;\;\;\; F_t^{t'}(\widetilde\eta)\in Ext^1_{\mathcal C_{t'}}(F_j^{t'}(Y_j),F_i^{t'}(X_i)).$$
Since $F_t^{t'}$ is compatible with the $2$-Calabi-Yau structure, we have $\gamma_{t'}(F_t^{t'}(\widetilde\xi),F_t^{t'}(\widetilde\eta))=\gamma_t(\widetilde\xi,\widetilde\eta)$. For any $\xi=\xi'$, which means the admissible exact sequences $\xi$ and $\xi'$ are isomorphic. We can choose $t$ bigger enough, such that $\widetilde \xi$ and $\widetilde \xi'$ are isomorphic as admissible exact sequences in $\mathcal C_t$, by the well-defined of $\gamma_t$, we have $\gamma(\xi,\eta)=\gamma_t(\widetilde \xi,\widetilde \eta)=\gamma_t(\widetilde \xi',\widetilde\eta)=\gamma(\xi',\eta)$. Dually, if $\eta=\eta'$, we have $\gamma(\xi,\eta)=\gamma(\xi,\eta')$.

It remains to prove that $\gamma$ is non-degenerate. Assume that there is a $\xi\in Ext^1_{\mathcal C}([X_i],[Y_j])$ such that $\gamma(\xi,\eta)=0$ for all $\eta\in Ext^1_{\mathcal C}([Y_j],[X_i])$. Thus, we have $\gamma_t(\widetilde \xi, \widetilde \eta)=0$ for all $\widetilde\eta\in Ext^1_{\mathcal C_t}([F_j^t(Y_j)],[F_i^t( X_i)])=0$. Since $\gamma_t$ is non-degenerate, we have $\widetilde \xi=0$, it implies $\xi=0$. So, $\gamma$ is non-degenerate.

By the non-degenerate bilinear form $\gamma$, we have isomorphisms $Ext^1_{\mathcal C}([Y_j],[X_i])\cong DExt^1_{\mathcal C}([X_i],[Y_j])$ for all $[X_i],[Y_j]\in \mathcal C$. It is easy to see that such isomorphisms are functorial since $\gamma$ is constructed by $\{\gamma_t\}$ and $\{\gamma_t\}$ can induce functorial isomorphisms.
\end{proof}

For exact categories $\mathcal A$ and $\mathcal B$ with an exact functor $\pi: \mathcal A\rightarrow \mathcal B$, let $\mathcal M$ be an additive subcategory of $\mathcal B$. Define $\mathcal C_{\mathcal M}:=\pi^{-1}(\mathcal M)$, that is, the subcategory of $\mathcal A$ generated by all objects whose images under $\pi$ are in $\mathcal M$.

\begin{Proposition}\label{2CY-Fro}
Let $\mathcal C_i$ and $\mathcal D_i$, $i\in \mathbb N$ be exact categories, and exact functors $\pi_i:\mathcal C_i\rightarrow \mathcal D_i$. Assume the families $\{F_i^j\}_{i\in \mathbb N}$ and $\{G_i^j\}_{i\in \mathbb N}$ are respectively  systems of exact functors for $\mathcal C_i$ and $\mathcal D_i$, $i\in \mathbb N$ satisfying $\pi_j F_i^j=G_i^j\pi_i$ for all $i\leq j$. Let $\mathcal C=\lim\limits_{\rightarrow}\mathcal C_i$ and $\mathcal D=\lim\limits_{\rightarrow}\mathcal D_i$ as constructed in Lemma \ref{colim} with the embedding functors $F_i:\mathcal C_i\rightarrow \mathcal C$ and $G_i:\mathcal D_i\rightarrow \mathcal D$. Let $\mathcal M_i$, $i\in \mathbb N$ be additive subcategories of $\mathcal D_i$. Assume that $S:=\{[M_i]\in \mathcal D\;|\;i\in \mathbb N \;\text{and}\;M_i\in \mathcal M_i\;\text{satisfying}\;G_i^j(M_i)\in \mathcal M_j, \forall j\geq i\}\neq \emptyset$. Let $\mathcal M$ be the additive subcategory of $\mathcal D$ generated by $S$. Then the following statements hold:

  (1)~  There is a $\pi:\mathcal C\rightarrow \mathcal D$ such that $\pi F_i=G_i\pi_i$ for all $i$.

(2)~ If all $\mathcal C_{\mathcal M_i}$ are extension closed subcategories of $\mathcal C_i$, then $\mathcal C_{\mathcal M}$ is an extension closed subcategory of $\mathcal C$.

(3)~ If all $\mathcal C_{\mathcal M_i}$ are closed under factor, then $\mathcal C_{\mathcal M}$ is closed under factor.

(4)~ Under the conditions of (2) and (3), assume that all $\mathcal C_{\mathcal M_i}, i\in \mathbb N$ are $2$-Calabi-Yau Frobenius categories. Denote by $\mathcal I_i$ the class consisting of all injective objects of $\mathcal C_{\mathcal M_i}$ satisfying  $F_i^j(I_i)\in \mathcal I_j$ for all $j\geq i$ and $I_i\in \mathcal I_i$, and by $\mathcal I$ the subcategory consisting of all objects of $\mathcal C_{\mathcal M}$ isomorphic to $[I_i]$ for all $I_i\in \mathcal I_i$. If $\mathcal I$ is a generator and co-generator subcategory of $\mathcal C_{\mathcal M}$ and $F_i^j:Ext^1_{\mathcal C_i}(X,Y)\rightarrow Ext^1_{\mathcal C_j}(F_i^j(X),F_i^j(Y))$ are group isomorphisms for all $i\leq j$ and compatible with the $2$-Calabi-Yau structures, then $\mathcal C_{\mathcal M}$ is also a $2$-Calabi-Yau Frobenius category whose projective-injective objects just form  $\mathcal I$.
\end{Proposition}

\begin{proof}
(1) By Lemma \ref{colim}, there exist $F_i:\mathcal C_i\rightarrow \mathcal C$, $G_i:\mathcal D_i\rightarrow \mathcal D$ and natural functor isomorphisms $\delta_i^j:G_i\rightarrow G_jG_i^j$ such that $(\delta_i^i)_{Y_i}=id_{G_i(Y_i)}$ and $(\delta_j^k)_{G_i^j(Y_i)}(\delta_i^j)_{Y_i}=(\delta_i^k)_{Y_i}$ for all $i\leq j\leq k$ and $Y_i\in \mathcal D_i$. Since $G_i^j\pi_i=\pi_j F_i^j$, we have the natural functor isomorphism $\delta_i^j\pi_i:G_i\pi_i \rightarrow G_jG_i^j\pi_i=(G_j\pi_j) F_i^j$ for $i\leq j$. For any $i\leq j\leq k$ and $X_i\in \mathcal C_i$, we have $(\delta_i^i\pi_i)_{X_i}=id_{G_i\pi_i(X_i)}$ and $(\delta_j^k\pi_j)_{F_i^j(X_i)}(\delta_i^j\pi_i)_{X_i}=(\delta_j^kG_i^j\pi_i)_{X_i}(\delta_i^j\pi_i)_{X_i}=(\delta_i^k\pi_i)_{X_i}$. Thus, by the universal property of colimit, there is a functor $\pi:\mathcal C\rightarrow \mathcal D$ such that $\pi F_i=G_i\pi_i$.

(2) For any $[X_i],[Z_k]\in\mathcal C_{\mathcal M}$ and short exact sequence $0\rightarrow [X_i]\rightarrow [Y_j]\rightarrow [Z_k]\rightarrow 0$ in $\mathcal C$, then there exists a short exact sequence $0\rightarrow X'_s\rightarrow Y'_s\rightarrow Z'_s\rightarrow 0$ in $\mathcal C_s$ such that $0\rightarrow [X_i]\rightarrow [Y_j]\rightarrow [Z_k]\rightarrow 0$ is isomorphic to $0\rightarrow [X'_s]\rightarrow [Y'_s]\rightarrow [Z'_s]\rightarrow 0$. (Since $F_i^j$ ($j\geq i$) are exact, we may assume that $s\geq i,j,k$.) Thus, $X'_s\cong F_i^s(X_i)$, $Y'_s\cong F_j^s(Y_j)$ and $Z'_s\cong F_k^s(Z_k)$. Since $[X_i],[Z_k]\in\mathcal C_{\mathcal M}$, so $F_s^t(X'_s)=F_s^tF_i^s(X_i)=F_i^t(X_i)\in \mathcal C_{\mathcal M_t}$ and $F_s^t(Z'_s)=F_s^tF_k^s(Z_k)=F_k^t(Z_k)\in \mathcal C_{\mathcal M_t}$ for all $t\geq s$, and since $\mathcal C_{\mathcal M_t}$ is extension closed, we have $F_s^t(Y'_s)\in \mathcal C_{\mathcal M_t}$ for all $t\geq s$. Therefore, $[Y'_s]\cong[F_j^s(Y_j)]\cong[Y_j]\in \mathcal C_{\mathcal M}$. Thus, $\mathcal C_{\mathcal M}$ is extension closed.

(3) For any $[X_i]\in \mathcal C_{\mathcal M}$ and surjective morphism $[X_i]\overset{f}{\rightarrow} [Y_j]$ in $\mathcal C$. For any $s\geq max(i,j)$, then $F_i^s(X_i)\in \mathcal C_{\mathcal M_s}$ and $F_i^s(X_i)\overset{f}{\rightarrow} F_j^s(Y_j)$ is surjective in $\mathcal C_s$, and since $\mathcal C_{\mathcal M_s}$ is closed under factor. Thus, $F_j^s(Y_j)\in \mathcal C_{\mathcal M_s}$. Hence, $[F_j^{max(i,j)}(Y_j)]\cong [Y_j]\in \mathcal C_{\mathcal M}$. Therefore, $\mathcal C_{\mathcal M}$ is closed under factor.

(4) By Lemma \ref{2-C-Y}, $\mathcal C$ is a $2$-Calabi-Yau category, and by (2) $\mathcal C_{\mathcal M}$ is closed under extension. Thus, $\mathcal C_{\mathcal M}$ is $2$-Calabi-Yau.

Firstly, we show any object $[I_i]\in \mathcal I$ is injective in $\mathcal C_{\mathcal M}$. For any exact sequence $0\rightarrow [X_j]\rightarrow [Y_j]\rightarrow [Z_j]\rightarrow 0$, then $0\rightarrow X_j\rightarrow Y_j\rightarrow Z_j\rightarrow 0$ is exact in $\mathcal C_{\mathcal M_j}$.
If $i\leq j$, since $F_i^j(I_i)$ is injective in $\mathcal C_{\mathcal M_j}$, we have $0\rightarrow Hom_{\mathcal C_j}(Z_j, F_i^j(I_i))\rightarrow Hom_{\mathcal C_j}(Y_j, F_i^j(I_i))\rightarrow Hom_{\mathcal C_j}(X_j, F_i^j(I_i))\rightarrow 0$ is exact. Thus, $0\rightarrow Hom_{\mathcal C}([Z_j], [I_i])\rightarrow Hom_{\mathcal C}([Y_j], [I_i])\rightarrow Hom_{\mathcal C}([X_j], [I_i])\rightarrow 0$ is exact. If $i\geq j$, since $F_j^i$ is exact, $0\rightarrow F_j^i(X_j)\rightarrow F_j^i(Y_j)\rightarrow F_j^i(Z_j)\rightarrow 0$ is exact in $\mathcal C_i$. As $I_i$ is injective in $\mathcal C_i$, we have $0\rightarrow Hom_{\mathcal C_i}(F_j^i(Z_j), I_i)\rightarrow Hom_{\mathcal C_i}(F_j^i(Y_j), I_i)\rightarrow Hom_{\mathcal C_i}(F_j^i(X_j), I_i)\rightarrow 0$ is exact. Thus, $0\rightarrow Hom_{\mathcal C}([Z_j], [I_i])\rightarrow Hom_{\mathcal C}([Y_j], [I_i])\rightarrow Hom_{\mathcal C}([X_j], [I_i])\rightarrow 0$ is exact. Therefore, $[I_i]$ is injective.

Now we prove that any injective object $[X_i]$ is in $\mathcal I$. Since $\mathcal I$ is a co-generator subcategory of $\mathcal C_{\mathcal M}$, there exists $[I_j]\in \mathcal I$ and exact sequence $0\rightarrow [X_i]\rightarrow [I_j]\rightarrow [Y_k]\rightarrow 0$. By (3), $\mathcal C_{\mathcal M}$ is closed under factor, $[Y_k]\in \mathcal C_{\mathcal M}$. As $[X_i]$ is injective in $\mathcal C_{\mathcal M}$, we have the exact sequence is split. Thus, $[I_j]\cong [X_i]\oplus [Z_k]\cong[F_i^{max(i,k)}(X_i\oplus Z_k)]$. Let $s=max(i,j,k)$. Since $Hom_{\mathcal C}([I_j],[F_i^{max(i,k)}(X_i\oplus Z_k)])=Hom_{\mathcal C_s}(F_j^s(I_j), F_i^s(X_i)\oplus F_k^s(Z_k)).$ Thus, $F_j^s(I_j)\cong F_i^s(X_i)\oplus F_k^s(Z_k)$ in $\mathcal C_s$. For any $t\geq s$, since $Hom_{\mathcal C_s}(X,Y)\cong Hom_{\mathcal C_t}(F_s^t(X),F_s^t(Y))$ for all $X,Y\in \mathcal C_s$, we have $F_j^t(I_j)\cong F_i^t(X_i)\oplus F_k^t(Z_k)$ in $\mathcal C_t$, moreover, as $[I_j]\in \mathcal I$, $F_i^t(I_j)$ is injective in $\mathcal C_{\mathcal M_t}$. Thus, $F_i^t(X_i)$ is injective in $\mathcal C_{\mathcal M_t}$. Therefore, $[X_i]\cong[F_s^t(X_i)]\in \mathcal I$.

Therefore, $\mathcal I$ is the subcategory formed by all injective objects of $\mathcal C_{\mathcal M}$. As $\mathcal C_{\mathcal M}$ is $2$-Calabi-Yau, $[I]\in \mathcal C_{\mathcal M}$ is projective if and only if $[I]$ is injective. Moreover, $\mathcal I$ is the generator and co-generator subcategory, then $\mathcal C_{\mathcal M}$ has enough projective and injective objects. The result follows.
\end{proof}

\subsection{Strongly almost finite Frobenius $2$-Calabi-Yau category $\mathcal C^Q$}.

In this part, we  use Proposition \ref{2CY-Fro} to construct a strongly almost finite Frobenius $2$-Calabi-Yau category $\mathcal C^Q$ via a strongly almost finite quiver $Q$. Moreover, we study some properties of $\mathcal C^Q$.

Assume that $Q$ is a strongly almost finite quiver. Let $\overline Q$ be the double quiver of $Q$, more precisely, it is obtained by adding an new arrow $\alpha^{*}:j\rightarrow i$ whenever there is an arrow $\alpha:i\rightarrow j$. For $K\overline Q$ the path algebra associated with $\overline Q$, let $\mathcal I$ be the idea of $K\overline{Q}$ generated by $\sum\limits_{s(\alpha)=i}\alpha^{*}\alpha-\sum\limits_{t(\alpha)=i}\alpha\alpha^{*}$ for all $i\in Q_0$. Denote by $\Lambda_Q$ the quotient algebra of $K\overline Q/\mathcal I$. As like as preprojective algebras from a finite quiver, we call $\Lambda_Q$ the {\bf preprojective algebra} of $Q$, although in this case the algebra $\Lambda_Q$ may not possess identity when $Q$ has infinite vertices. Thanks to $KQ$ as a subalgebra of $\Lambda_Q$, we denote $\pi: mod\Lambda_Q\rightarrow modKQ$ the restriction functor. For any additive subcategory $\mathcal M$ of $modKQ$, let $\mathcal C_{\mathcal M}:=\pi^{-1}(\mathcal M)$.

Recall that, for any finite acyclic quiver $Q'$, Ringel \cite{R} defined a category $C_{Q'}(1,\tau_{Q'})$ as follows. The objects are of the form $(X,f)$, where $X\in modKQ'$ and $f\in Hom_{KQ'}(X,\tau_{Q'}(X))$. Here $\tau_{Q'}$ is the Auslander-Reiten translation in $modKQ'$. The homomorphism from $(X,f)$ to $(X',f')$ is a $KQ'$-module homomorphism $g:X\rightarrow X'$ such that the following diagram
$$\xymatrix{
       X\ar[r]^{f}\ar[d]_{g} & \tau_{Q'}(X)\ar[d]^{\tau_{Q'}(g)} \\
       X'\ar[r]^{f'}   & \tau_{Q'}(X')}$$
commutes. Ringel (Theorem $B$, \cite{R}) proved that $\mathcal C_{Q'}(1,\tau_{Q'})$ isomorphic to $mod\Lambda_{Q'}$. More precisely, there exists an isomorphic functor $\Phi_{Q'}:mod\Lambda_{Q'}\rightarrow \mathcal C_{Q'}(1,\tau_{Q'})$ such that $\Phi_{Q'}(X)=(Y,f)$ implies $\pi_{Q'}(X)=Y$. We extend such construction to strongly almost finite quiver $Q$ to obtain the category $\mathcal C_{Q}(1,\tau)$, briefly written as $\mathcal C(1,\tau)$, where $\tau$ is the Auslander-Reiten translation in $modKQ$.

\begin{Theorem}
Keep the forgoing notations. Let $Q$ be a strongly almost finite quiver. Then $\mathcal C(1,\tau)$ is categorically isomorphic to $mod\Lambda_Q$ with  an isomorphic functor $\Phi:mod\Lambda_Q\rightarrow \mathcal C(1,\tau)$ satisfying $\pi(X)=Y$ for $\Phi(X)=(Y,f)$.
\end{Theorem}

\begin{proof}
For any $Y=((Y_i)_{i\in Q_0},(Y_{\alpha})_{\alpha\in Q_1})\in modKQ$, we have $\tau(Y)=((\tau(Y)_i)_{i\in Q_0},(\tau(Y)_{\alpha})_{\alpha\in Q_1})$. Since $Y$ and $\tau(Y)$ are finite dimensional, we can find a finite subquiver $Q'$ of $Q$ such that $Y,\tau(Y)\in modKQ'$ under the restriction of $KQ$-action. According the proof of Theorem B', \cite{R}, we can define linear maps $g_{\alpha}:\tau(Y)_{t(\alpha)}\rightarrow Y_{s(\alpha)}$ for all $\alpha\in Q'_1$ inductively on the given partial order of vertices of $Q'$ due to the orientation of arrows.

First, for all $\alpha_0\in Q'_1$ with $t(\alpha_0)=i_0$ a sink (i.e. a minimal vertex on the partial order), we can give the exact sequence
$$0\rightarrow \tau(Y)_{i_0}\overset{(g_{\alpha_0})_{t(\alpha_0)=i_0}}{\longrightarrow} \bigoplus\limits_{t(\alpha_0)=i_0}Y_{s(\alpha_0)}\overset{(Y_{\alpha_0})_{t(\alpha_0)=i_0}^{\top}}{\longrightarrow} Y_{i_0}.$$
to obtain $g_{\alpha_0}$ for all $\alpha_0$ with $t(\alpha_0)=i_0$.
Note in this sequence $(g_{\alpha_0})_{t(\alpha_0)=i_0}$ means to arrange it in a row and $(Y_{\alpha_0})_{t(\alpha_0)=i_0}^{\top}$ means to arrange it in a column such that the operations can be made. In the next sequence the description is similar.

 Next, using the following exact sequence:
 $$0\rightarrow \tau(Y)_i\overset{((g_{\beta})_{t(\beta)=i},(\tau(Y)_{\gamma}))_{s(\gamma)=i}}{\longrightarrow} \bigoplus\limits_{t(\beta)=i}Y_{s(\beta)}\oplus \bigoplus\limits_{s(\gamma)=i}\tau(Y)_{t(\gamma)}\overset{\left(
                                                                                                             \begin{array}{c}
                                                                                                               (Y_{\beta})^{\top}_{t(\beta)=i} \\
                                                                                                               (g_{\gamma})^{\top}_{s(\gamma)=i} \\
                                                                                                             \end{array}
                                                                                                           \right)
 }{\longrightarrow} Y_i,$$
 for any $i\in Q'_0$ we can decide all $g_{\beta}$  with $t(\beta)=i$ inductively by those $g_{\gamma}$ with $s(\gamma)=i$. Note that such exact sequence always exists by Theorem B', \cite{R} since $Q'$ is finite acyclic. For all $\alpha\in Q_1\setminus Q'_1$, set $g_{\alpha}=0$. Then, we obtain $(g_{\alpha})_{\alpha\in Q_1}$ on $Q'$.

Use this $(g_{\alpha})_{\alpha\in Q_1}$ above from $Q'$, we can construct,  relying on $g_{\alpha}$,  the functors $\Psi:\mathcal C(1,\tau)\rightarrow mod\Lambda_Q$ and $\Phi:mod\Lambda_Q\rightarrow \mathcal C(1,\tau)$ as follows.

 Precisely, given $(Y,f)\in \mathcal C(1,\tau)$, define $\Psi(Y,f)=((X_i)_{i\in Q_0}, (X_{\alpha})_{\alpha\in \overline Q_1})$ as follows: $X_i=Y_i$ for each $i\in Q_0$ and $X_{\alpha}=Y_{\alpha}$, $X_{\alpha^{*}}=g_{\alpha}f_{t(\alpha)}$ for each $\alpha \in Q_1$. As in the proof of Theorem B', \cite{R}, we can prove that $X\in mod\Lambda_Q$. Conversely, given $X\in mod\Lambda_Q$, there exist uniquely $f=(f_i)_{i\in Q_0}$ such that $X_{\alpha^{*}}=g_{\alpha}f_{t(\alpha)}$ for each $\alpha\in Q_1$. Let $Y=\pi(X)$ for $\pi$ the restriction functor from $mod\Lambda_Q$ to $modKQ$. Then $f:Y\rightarrow \tau(Y)$ is a $modKQ$ homomorphism, and define $\Phi(X)=(Y,f)$.

It can be proved similarly as Theorem B' of \cite{R} that $\Phi\Psi=id_{\mathcal C(1,\tau)}$ and $\Psi\Phi=id_{mod\Lambda_Q}$. Therefore, our result follows.
\end{proof}

Let $Q$ be a strongly almost finite quiver. Following \cite{GLS}, for any $a\leq b\in \mathbb N$ and $i\in Q_0$, let $I_i$ be the indecomposable representation of $KQ$ corresponding to vertex $i$, define
$$I_{i,[a,b]}=\bigoplus\limits_{j=a}^b\tau^j(I_i),$$ and
$$e_{i,[a,b]}=\left(\begin{array}{cccc}
0 & 1       &    & \\
  & \ddots  & \ddots &\\
  &  & 0 & 1\\
  &  &   & 0
\end{array}\right):I_{i,[a,b]}\rightarrow \tau(I_{i,[a,b]}).$$

Let $Q'$ be a finite acyclic quiver, recall that an object $M$ of $modKQ'$ is called {\bf terminal} (2.2., \cite{GLS}) if the following hold:

(1) $M$ is pre-injective;

(2) If $X$ is an indecomposable module of $KQ'$ such that $Hom(M,X)\neq 0$, then $X\in add(\mathcal M)$;

(3) $I_i\in add(\mathcal M)$ for all indecomposable injective $KQ'$-modules $I_i$.

If $M$ is a terminal module, define $t_i(M):=max\{j\geq 0\;|\;\tau^j(I_i)\in add(M)\setminus \{0\}\}$. See \cite{GLS}.

Note that if $Q'$ is not type $A$ with linear orientation, then $M=DKQ'\oplus \tau(DKQ')$ is a terminal module of $KQ'$ with $t_i(M)=1$ for all $i\in Q'_0$.

If $Q'$ is an acyclic finite quiver with $n$ vertices, denote $\pi_{Q'}:mod\Lambda_{Q'}\rightarrow modKQ'$ be the restriction functor, where $\Lambda_{Q'}$ is the preprojective algebra of $KQ'$. Assume $M$ is a terminal module of $KQ'$. Let $\mathcal C'_{add(M)}:=\pi_{Q'}^{-1}(add(M))$.
We obtain  the following:

\begin{Theorem}\label{GLSmain}(Theorem 2.1., Lemma 5.5., Lemma 5.6., \cite{GLS})

(1) $\mathcal C'_{add(M)}$ is a Frobenius $2$-Calabi-Yau category with $n$ indecomposable $\mathcal C'_{add(M)}$-projective-injective objects $$\{(I_{i,[0,t_i(M)]},e_{i,[0,t_i(M)]})\;|\;i=1,\cdots,n\};$$

(2) $\mathcal C'_{add(M)}$ is closed under extension;

(3) $\mathcal C'_{add(M)}$ is closed under factor.
\end{Theorem}

\begin{Theorem}\label{infversion2-c-y}
Let $Q$ be a strongly almost finite quiver. Denote the vertices of $Q$ by $1,2,\cdots$, and $Q^{[i]}$ be the sub-quiver of $Q$ generated by vertices $1,\cdots, i$ for any $i\in \mathbb N$. Let $G_i^j:modKQ^{[i]}\rightarrow modKQ^{[j]}$ be the embedding functors for all $i\leq j\in \mathbb N$. Assume that $M_i$ are terminal modules of $KQ^{[i]}$ and $\mathcal N$ is the additive subcategory of $modKQ$ generated by all objects $M\in add(M_i)$ satisfying $G_i^j(M)\in add(M_j)$ for all $j\geq i$. Denote by $\mathcal I_i$ the subcategory of injective objects in $\mathcal C'_{add(M_i)}$ and by $\mathcal I$ the additive subcategory of $\mathcal C_{\mathcal N}$ generated by all objects $I\in \mathcal I_i$ satisfying $G_i^j(I)\in \mathcal I_j$ for all $j\geq i$. Assume that $\mathcal I$ is a generator and co-generator  subcategory of $\mathcal C_{\mathcal N}$. Then,

(1)~ $\mathcal C_{\mathcal N}$ is closed under extension and under factor respectively;

(2)~ $\mathcal C_{\mathcal N}$ is a Frobenius $2$-Calabi-Yau category whose  projective-injective subcategory is just $\mathcal I$.
\end{Theorem}

\begin{proof}
Let $\Lambda_i:=\Lambda_{Q^{[i]}}$ be preprojective algebras of $KQ^{[i]}$ and $F_i^j:mod\Lambda_i\rightarrow mod\Lambda_j$ be the embedding functors for all $i\leq j\in \mathbb N$. It is easy to see that the conditions of Lemma \ref{colim} are satisfied for $mod\Lambda_i$ and $F_i^j$. Therefore, $\mathcal C:=\lim\limits_{\rightarrow}mod\Lambda_i$ exists.

Now we prove that $mod\Lambda_Q$ is equivalent to $\mathcal C$. Let $\widetilde H:\mathcal C\rightarrow mod\Lambda_Q$ be the functor such that $\widetilde H([X_i])=X_i$ for each $[X_i]\in \mathcal C$ and $\widetilde H':mod\Lambda_Q\rightarrow \mathcal C$ be the functor such that $\widetilde H'(X_i)=[X_i]$ for each $X_i\in mod\Lambda_Q$, where $i$ is the least number such that $X_i\in mod\Lambda_{Q^{[i]}}$. It is clear that $\widetilde H$ and $\widetilde H'$ are equivalent functors between $mod\Lambda_Q$ and $\mathcal C$.

Similarly, $\lim\limits_{\rightarrow}modKQ^{[i]}$ exists, denoted by $\mathcal D$, and $modKQ$ is equivalent to $\mathcal D$. Moreover, $H:\mathcal D\rightarrow modKQ$ via $[X_i]\mapsto X_i$ is an equivalence functor.

It is clear that $\pi_jF_i^j=G_i^j\pi_i$. Thus, according to Proposition \ref{2CY-Fro} (1), there is a functor $\widetilde \pi$ such that $\widetilde \pi F_i=G_i\pi_i$ for all $i$. It is easy to verify the following commutative diagram:
$$\xymatrix{
       \mathcal C\ar[r]^{\widetilde H}\ar[d]_{\widetilde \pi}    & mod\Lambda_Q\ar[d]_{\pi} \\
       \mathcal D\ar[r]^{H}   & modKQ.}$$
Denote by $\mathcal M$  the additive subcategory of $\mathcal D$ generated by all objects $[M_i]$ with $M_i\in \mathcal M_i$ and $G_i^j(M_i)\in \mathcal M_j$ for all $j\geq i$. Thus, $H$ induces an equivalence between $\mathcal M$ and $\mathcal N$. By the above commutative diagram,  $\widetilde H$ induces an equivalence between $\mathcal C_{\mathcal M}$ and $\mathcal C_{\mathcal N}$.

(1) By Theorem \ref{GLSmain}, $\mathcal C'_{add(M_i)}$ are closed subcategories of $mod\Lambda_{Q^{[i]}}$ under extension.
Using Proposition \ref{2CY-Fro} (2), we have $\mathcal C_{\mathcal M}$ is a closed subcategory of $\mathcal C$ under extension. Thus, $\mathcal C_{\mathcal N}$ is an extension closed subcategory of $mod\Lambda_Q$.

 By Theorem \ref{GLSmain}, $\mathcal C'_{add(M_i)}$ are closed under factor. Using Proposition \ref{2CY-Fro} (3), we have $\mathcal C_{\mathcal M}$ is closed under factor. Thus, $\mathcal C_{\mathcal N}$ is an extension closed subcategory of $mod\Lambda_Q$.

(2) Let $\widetilde{\mathcal I}$ be the subcategory of $\mathcal C_{\mathcal M}$ which contains all objects $[I]$ with $I\in \mathcal I_i$ and $G_i^j(I)\in \mathcal I_j$ for all $j\geq i$. Then $\widetilde H$ induces equivalence between $\widetilde {\mathcal I}$ and $\mathcal I$. Since $\mathcal I$ is a generator and co-generator subcategory of $\mathcal C_{\mathcal N}$, $\widetilde{\mathcal I}$ is a generator and co-generator subcategory of $\mathcal C_{\mathcal M}$. By the proof of Theorem 3 of \cite{GLS2}, $F_i^j$ are compatible with the $2$-Calabi-Yau structures. Therefore, by Proposition \ref{2CY-Fro} (4), our result follows.
\end{proof}

From the given quiver $Q$ above, let $\mathcal M$ be the additive subcategory of $modKQ$ generated by $I_j$ and $\tau(I_j)$ for all injective objects $I_j$ and $j\in Q_0$. Then, we will always denote $\mathcal C^Q:=\mathcal C_{\mathcal M}$.

\begin{Corollary}\label{2-cy}
Let us keep the forgoing notations. Assume that $Q$ is a strongly almost finite quiver but not a linear quiver of type $A$. Then,

(1) $\mathcal C^Q$ is a $2$-Calabi-Yau Frobenius category;

(2) The projective-injective subcategory is $\{(I_{j,(0,1)},e_{j,[0,1]})\;|\;j\in Q_0\}$.
\end{Corollary}

\begin{proof}
(1)~  When $Q$ is finite, the result follows immediately by Theorem \ref{GLSmain}. So, we can assume that $Q$ is an infinite quiver.
 Following this, since $Q$ is not a linear quiver of type $A$, we can renumber the vertices of $Q$ such that $Q^{[i]}$ is not of type $A_i$ with linear orientation for each $i\in \mathbb N$. Let $M_i:=DKQ^{[i]}\oplus \tau(DKQ^{[i]})$. Thus, $M_i$ is a terminal module of $KQ^{[i]}$ with $t_j(M_i)=1$ for all $j\in Q^{[i]}_0$. As $Q$ is strongly almost finite, by the similar proof of Corollary 8.4 and Corollary 8.8 of \cite{GLS}, we have $\{(I_{j,(0,1)},e_{j,[0,1]})\;|\;j\in Q_0\}$ is a co-generator and generator subcategory of $\mathcal C^{Q}$. Thus, $\mathcal C^{Q}$ satisfies the condition of Theorem \ref{infversion2-c-y}. Therefore, our result follows.

(2)~It follows by Theorem \ref{infversion2-c-y} (2).
\end{proof}

We say that a group $\Gamma$ {\bf acts on} a quiver $Q$ if there is a group homomorphism $\varphi:\Gamma\rightarrow Aut(Q)$, where $Aut(Q)$ is the automorphism of $Q$ as a directed graph. In this case, we have $h\cdot i=\varphi(h)(i)$  for $i\in Q_0$.

\begin{Definition}\label{def3.1}
Let $\mathcal C$ be a strongly almost finite Frobenius exact category and $\Gamma$ be an (infinite) group. We call that $\mathcal C$ admits a {\bf $\Gamma$-action} if there are categorical isomorphisms $\varphi(h):\mathcal C\rightarrow \mathcal C$ for $h\in \Gamma$ satisfying

(a)~ $\varphi(e)=id_{\mathcal C}$, where $e$ is the identity element of $\Gamma$ and $id_{\mathcal C}$ is the identity functor;

(b)~ $\varphi(hh')=\varphi(h)\varphi(h')$ for any pair $h,h'\in \Gamma$.
\end{Definition}

We write $\varphi(h)(X)$ as $h\cdot X$ briefly for any $X\in \mathcal C$.

\begin{Definition}(Definition 3.17, \cite{D})
Let $X\in \mathcal C$. An epimorphism $f:X\rightarrow Y$ is said to be a {\bf left rigid quasi-approximation} of $X$ if the following conditions are satisfied:
\begin{itemize}
  \item $add(\{h\cdot Y\;|\;h\in \Gamma\})$ is rigid;
  \item $Y$ has an injective envelope, without direct summand in $add(\{h\cdot X\;|\;h\in \Gamma\})$;
  \item If $add(\{h\cdot Z\;|\;h\in \Gamma\})$ is rigid, then every morphism from $X$ to $Z$ without invertible matrix coefficient, factorizes through $f$.
\end{itemize}
\end{Definition}

Let $Q$ be a strongly almost finite quiver which admits a group $\Gamma$ action, it is easy to see $\Gamma$ acts in $\mathcal C^Q$. Assume that $Q$ is not a linear quiver of type $A$. We will show that every projective object in $\mathcal C^Q$ has a left rigid quasi-approximation.

We first recall a crucial lemma of how to judge a projective object has a left rigid quasi-approximation.

\begin{Lemma}\label{left-rigid}(Lemma 3.20, \cite{D})
Let $\mathcal D$ be an abelian category endowed with an action of $\Gamma$. Let $\mathcal C$ be an
exact full subcategory of $\mathcal D$. Let $P$ be a projective indecomposable object of $\mathcal C$. If
\begin{itemize}
  \item every monomorphism of $\mathcal C$ to a projective object is admissible,
  \item every monomorphism of $\mathcal C$ from $P$ to an indecomposable object is admissible,
  \item $P$ has a simple socle $S$ in $\mathcal D$,
  \item $add(\{h\cdot S\;|\;h\in \Gamma\})$ is rigid in $\mathcal D$,
  \item the cokernel of $S \hookrightarrow P$ is in $\mathcal C$,
\end{itemize}
then this cokernel is a left rigid quasi-approximation of $P$ in $\mathcal C$.
\end{Lemma}

\begin{proof}
The proof of (Lemma 3.20, \cite{D}) for finite group can be transferred to that for infinite group $\Gamma$.
\end{proof}

The following corollary is inspired by Corollary of 4.42, \cite{D}.

\begin{Corollary}\label{quasi-approxi}
Let $Q$ be a strongly almost finite quiver with a group $\Gamma$-action. $Q$ is not a linear quiver of type $A$. Then all projective objects of $\mathcal C^Q$ have left rigid quasi-approximations.
\end{Corollary}

\begin{proof}
Let $P=(I_{i,(0,1)},e_{i,(0,1)})$ be a projective object in $\mathcal C^Q$. By Theorem \ref{infversion2-c-y} (2), $\mathcal C^Q$ is closed under factor, so all every monomorphism of $\mathcal C^Q$ to a projective object is admissible. Since $\mathcal C^Q$ is Frobenius, every monomorphism of $\mathcal C^Q$ from $P$ to an indecomposable object is split, so it is admissible. By Lemma 8.1, \cite{GLS}, $P$ has simple socle $(S_i,0)$. Moreover, since $Q$ has no $\Gamma$-loops and $\Gamma$-$2$-cycles, $add(\{h\cdot S\;|\;h\in \Gamma\})$ is rigid. In addition, $\mathcal C^Q$ is closed under factor,  the cokernel of $S\hookrightarrow P$ is in $\mathcal C^Q$. Applying Lemma \ref{left-rigid}, our result follows.
\end{proof}

For a Krull-Schmidt $K$-category $\mathcal T$, we define the Gabriel quiver $Q(\mathcal T)$ of $\mathcal T$ as follows: the vertices are the equivalent classes of indecomposable objects. For two vertices $[T_1]$ and $[T_2]$, the number of arrows from $[T_1]$ to $[T_2]$ equal to the dimension of $Irr(T_1,T_2):=rad(T_1,T_2)/rad^2(T_1,T_2)$ in the category $\mathcal T$, where $rad$ is the radical of $\mathcal T$. Note that if $\mathcal T=add(\bigoplus\limits_{i=1}^n T_i)$ with $\{T_i\;|\; i=1,\cdots,n\}$ are pairwise not isomorphic, then the Gabriel quiver of $\mathcal T$ is the same as the Gabriel quiver of the finite dimensional algebra $End_{\mathcal T}(\bigoplus\limits_{i=1}^n T_i)$.

Let $\mathcal T_0=add\{(I_{i,[0,1]},e_{i,[0,1]}), (I_{i,[1,1]},e_{i,[1,1]})\;|\;i\in Q_0\}$ as an subcategory of $\mathcal C^Q$.

\begin{Lemma}\label{mor}(Lemma 6.3, \cite{GLS})
Let $Q$ be a strongly almost finite quiver which is not a linear quiver of type $A$. For any $(X,f)\in \mathcal C^Q$, we have

(1) $Hom_{KQ}(\tau(I_i),X)\overset{\cong}{\rightarrow} Hom_{\mathcal C^Q}((I_{i,[0,1]},e_{i,[0,1]}),(X,f)): g\longmapsto (\tau^{-1}(f)\tau^{-1}(g), g)$;

(2) $\{g\in Hom_{KQ}(\tau(I_i),X)\;|\;fg=0\}\overset{\cong}{\rightarrow} Hom_{\mathcal C^Q}((I_{i,[1,1]},e_{i,[1,1]}),(X,f)): g\longmapsto g$;

(3) $\{g\in Hom_{KQ}(I_i,X)\;|\;fg=0\}\overset{\cong}{\rightarrow} Hom_{\mathcal C^Q}((I_{i,[0,0]},e_{i,[0,0]}),(X,f)): g\longmapsto g$.
\end{Lemma}

\begin{Lemma}\label{T_0}
Keep the notations as above. Let $Q$ be a strongly almost finite quiver which is not a linear quiver of type $A$.

(1)~ For any object $C$ in $\mathcal C^Q$, there are only finite indecomposable objects $T$ (up to isomorphism) in $\mathcal T_0$ such that $Hom_{\mathcal C^Q}(T,C)\neq 0$ or $Hom_{\mathcal C^Q}(C,T)\neq 0$. In particular, $\mathcal T_0$ is functorially finite.

(2)~ $\mathcal T_0$ is a maximal rigid subcategory of $\mathcal C^Q$.

(3)~ The Gabriel quiver of $\mathcal T_0$ is equal to the Gabriel quiver of $\{\tau(I_i), I_i\;|\;i\in Q_0\}$ by adding an arrow from $[I_i]$ to $[\tau(I_i)]$ for each $i\in Q_0$, more precisely, the vertex $[(I_{i,(0,1)},e_{i,(0,1)})]$ in the Gabriel quiver of $\mathcal T_0$ corresponds to the vertex $[I_i]$ in the Gabriel quiver of $\{\tau(I_i),I_i\;|\;i\in Q_0\}$ and $[(I_{i,(1,1)},e_{i,(1,1)})]$ to $I_i$.

(4)~ The Gabriel quiver of $\underline{\mathcal T_0}$ is isomorphic to $Q$, where $\underline{\mathcal T_0}$ is the subcategory of the stable category $\underline{\mathcal C^Q}$ of $\mathcal C^Q$ generated by all objects $X\in \mathcal T_0$.

\end{Lemma}

\begin{proof}
(1)~ Since $C\in \mathcal C^Q$ is finite dimensional, there are only finite vertices $i\in Q_0$ such that $C_i\neq 0$. As $Q$ is strongly almost finite, there are only finite indecomposable objects $T$ in $\mathcal T_0$ such that $T_i\neq 0$ for each $i\in Q_0$ with $C_i\neq 0$. Thus, our result follows.

(2)~ By Lemma 7.1, \cite{GLS}, since $Q$ is strongly almost finite, we have $\mathcal T_0$ is a rigid subcategory. Now we show it is actually maximal rigid.

For any $i\in Q_0$, it is easy to verify that $$0\rightarrow (I_{i,[0,0]},e_{i,[0,0]})\overset{(1,0)^t}{\rightarrow} (I_{i,[0,1]},e_{i,[0,1]})\overset{(0,1)}{\rightarrow}(I_{i,[1,1]},e_{i,[1,1]})\rightarrow 0$$ is exact. For any $(X,f)\in \mathcal C^Q$, since $(I_{i,[0,1]},e_{i,[0,1]})$ is projective-injective, applying $Hom_{\mathcal C^Q}(-,(X,f))$ to the exact sequence, we get $$Hom_{\mathcal C^Q}((I_{i,[0,1]},e_{i,[0,1]}),(X,f))\rightarrow Hom_{\mathcal C^Q}((I_{i,[0,0]},e_{i,[0,0]}),(X,f))\rightarrow Ext^1_{\mathcal C^Q}((I_{i,[1,1]},e_{i,[1,1]}),(X,f))\rightarrow 0.$$ Thus, using Lemma \ref{mor}, $Ext^1_{\mathcal C^Q}((I_{i,[1,1]},e_{i,[1,1]}),(X,f))=0$ if and only if $$Hom_{KQ}(\tau(I_i),X)\rightarrow \{g\in Hom_{KQ}(I_i,X)\;|\;fg=0\}:g\longmapsto \tau^{-1}(fg)$$ is surjective. Therefore, $(X,f)\in \mathcal T_0^{\bot}$ if and only if $$Hom_{KQ}(\tau(I_i),X)\rightarrow \{g\in Hom_{KQ}(I_i,X)\;|\;fg=0\}:g\longmapsto \tau^{-1}(fg)$$ is surjective for any $i\in Q_0$. By the definition of $\mathcal C^Q$, we may assume that $X=I\oplus \tau(I')$. Moreover, since $Hom_{KQ}(I,\tau(I'))=0$, in fact $f: \tau(I')\rightarrow \tau(I)$.
Since $Hom_{KQ}(I_i,\tau(I'))=0$, $\{g\in Hom_{KQ}(I_i,X)\;|\;fg=0\}=Hom_{KQ}(I_i,X)=Hom_{KQ}(I_i,I)$. Thus, we have
$$Hom_{KQ}(\tau(I),I\oplus \tau(I'))\rightarrow Hom_{KQ}(I,I):g\longmapsto \tau^{-1}(fg)$$ is surjective. Therefore, we have $I=0$; or $I'=I$ and then $f=1:\tau(I)\rightarrow \tau(I)$. In both cases, we have $(X,f)\in \mathcal T_0$. We have that $\mathcal T_0$ is a maximal rigid subcategory.

(3)~ Since we have Lemma \ref{mor}, the proof is similar as Lemma 7.6 of \cite{GLS}.

(4)~ It follows by (3).
\end{proof}

\begin{Remark}
(i)~ The (2) and (3) in Lemma \ref{mor} are respectively the generalizations of   Corollary 7.4 and Lemma 7.6 in \cite{GLS}.

(ii)~ Let $\mathcal P_{\mathcal C^Q}$ be the subcategory of $Hom\mathcal C^Q$ generated by all projective objects. Then $\underline{\mathcal C^Q}=\mathcal C^Q/\mathcal P_{\mathcal C}$. Since $\mathcal P_{\mathcal C^Q}$ is a subcategory of $\mathcal T_0$, we have $\underline{\mathcal T_0}=\mathcal T_0/\mathcal P_{\mathcal C^Q}$.
\end{Remark}

\section{Subcategory of $mod\Lambda_Q$ with an (infinite) group action.}\label{action}

Assume $\mathcal C$ is a strongly almost finite subcategory with a group $\Gamma$ action.  Let $\Gamma'=(h)$ be the subgroup of $\Gamma$ generated by $h\in \Gamma$. Define a category $\mathcal C_{h}$ as follows: the objects are $X\in \mathcal C$ and $Hom_{\mathcal C_h}(Y,X):=\bigoplus\limits_{h'\in\Gamma'}Hom_{\mathcal C}(Y, h'\cdot X)$ for $X,Y\in \mathcal C_h$. For all $f:X_i\rightarrow h'\cdot X_j$ and $g:X_j\rightarrow h''\cdot X_k$, their composition $g\circ f$ is defined to be $h'(g)f:X_i\rightarrow h'h''\cdot X_k$. Clearly, $\mathcal C_h$ is well-defined as a  $K$-linear category. For any subcategory of $\mathcal T$ of $\mathcal C$, denote by $\mathcal C_{h}(\mathcal T)$ the minimal subcategory of $\mathcal C_h$ which contains all objects $T$ of $\mathcal T$. Note that $\mathcal C_{h}$ is not a Krull-Schmidt category in general.

Let $Q$ be a strongly almost finite quiver and $\Gamma$ be an (infinite) subgroup of $Aut(Q)$. In this section, we study a strongly almost finite $2$-Calabi-Yau Frobenius subcategory $\mathcal C$ of $mod\Lambda_Q$ which is stable under the action of group $\Gamma$. Since $\Gamma$ is a subgroup of $Aut(Q)$, it is clear the action of $\Gamma$ on $\mathcal C$ is {\bf exact}, that is, $\varphi(h)$ are exact functors for all $h\in \Gamma$. We set some assumptions to the action of $\Gamma$ on $\mathcal C$:
\begin{itemize}
  \item for each indecomposable $X\in \mathcal C$ and $h\in \Gamma$, there exists a $h'\in \Gamma$ such that $h'\cdot X\cong h\cdot X$ and $\mathcal C_{h'}$ is $Hom$-finite.
  \item every projective object in $\mathcal C$ has a left rigid quasi-approximation.
\end{itemize}

In this section, we study some useful properties of such subcategory $\mathcal C$ of $mod\Lambda_Q$ given above. However, for its applications in Section \ref{quiver}, indeed, we will only consider $\mathcal C=\mathcal C^Q$ as its special kind. In fact, for $\mathcal M$ the additive subcategory of $modKQ$ generated by $I_j$ and $\tau(I_j)$ for all injective modules $I_j$ ($j\in Q_0$), we have $\mathcal C^Q=\pi^{-1}(\mathcal M)$ which is stable under the action of $\Gamma$ since $\mathcal M$ is stable under the action of $\Gamma$. By Corollary \ref{2-cy}, $\mathcal C^Q$ is a strongly almost finite $2$-Calabi-Yau Frobenius category.

Using the action of $\Gamma$ on $Q$, we give precisely the action of $\Gamma$ on $\mathcal C^Q$ satisfying that for each $X=(X_i,f_{\alpha})_{i\in \overline{Q}_0,\alpha\in \overline{Q}_1}\in \mathcal C^Q$, $h\in \Gamma$, define
\begin{equation}\label{action}
h\cdot X=(X_{h\cdot i},f_{h\cdot\alpha})_{i\in \overline{Q}_0,\alpha\in \overline{Q}_1}
\end{equation}
 where $\overline{Q}$ is the double quiver of $Q$.

\subsection{Maximal $\Gamma$-stable rigid subcategories and mutations }\label{proofmain}.

In this subsection, we study the mutation of maximal $\Gamma$-stable rigid subcategories of $\mathcal C$. We follow the strategy of Laurent Demonet \cite{D} through replacing finite group by infinite group. So, in this part, the definitions and results are parallel with that in \cite{D}.

\begin{Definition}(Definition 2.32, \cite{D})
A subcategory $\mathcal T$ of $\mathcal C$ is said to be {\bf $\Gamma$-stable} if $h\cdot X\in \mathcal T$ for every $X\in \mathcal T$ and $h\in \Gamma$.
\end{Definition}

Denote by $\mathfrak{Add}(\mathcal C)^{\Gamma}$ the family of all $\Gamma$-stable subcategories of $\mathfrak{Add}(\mathcal C)$.

\begin{Definition}(Definition 3.22, \cite{D})
Let $\mathcal T\in \mathfrak{Add}(\mathcal C)^{\Gamma}$ and let $X\in \mathcal T$ be indecomposable. A {\bf $\Gamma$-loop} of $\mathcal T$ at $X$ is an irreducible morphism $X\rightarrow h\cdot X$ for some $h\in \Gamma$. A {\bf $\Gamma$-$2$-cycle} at $X$ is a pair of irreducible morphisms $X\rightarrow Y$ and $Y\rightarrow h\cdot X$ for some $h\in \Gamma$.
\end{Definition}

It can be seen that the proofs of these results below can be verified similarly as in \cite{D}. Hence, these proofs are omitted here.

\begin{Lemma}\label{3.6}(Lemma 3.6, \cite{D})
Let $\mathcal T\in \mathfrak{Add}(\mathcal C)^{\Gamma}$ be a rigid subcategory of $\mathcal C$ and $X\in \mathcal C$ such that $add(\{h\cdot X\;|\;h\in \Gamma\})$ is rigid. If
$0\rightarrow X\overset{f}{\rightarrow}T\overset{g}{\rightarrow}Y\rightarrow 0$
is an admissible short exact sequence and $f$ is a left $\mathcal T$-approximation, then the category
$add(\mathcal T\cup \{h\cdot Y\;|\;h\in\Gamma\})$
 is rigid.
\end{Lemma}

Although the description of the above lemma is different with Lemma 3.6 in \cite{D}, but its proof is exactly parallel with the proof of Lemma 3.6 in \cite{D}.

\begin{Lemma}\label{leftrightmu}(Proposition 3.7, \cite{D})
Let $\mathcal T\in \mathfrak{Add}(\mathcal C)^{\Gamma}$ be a rigid subcategory of $\mathcal C$ and $X\in \mathcal C$ be indecomposable such that $X\notin\mathcal T$. Suppose that $\mathcal T$ contains all projective objects of $\mathcal C$, and $add(\mathcal T\cup\{h\cdot X\;|\;h\in \Gamma\})$ is rigid. If $X$ has minimal left and right $T$-approximations, then up to isomorphisms, there exist two admissible short exact sequences:

\centerline{$0\rightarrow X\overset{f}{\rightarrow}T\overset{g}{\rightarrow}Y\rightarrow 0$\;\;\;\text{and}\;\;\;$0\rightarrow Y'\overset{f'}{\rightarrow}T'\overset{g'}{\rightarrow}X\rightarrow 0$}
such that

(1)~ $f$ and $f'$ are minimal left $\mathcal T$-approximations;

(2)~ $g$ and $g'$ are minimal right $\mathcal T$-approximations;

(3)~ $add(\mathcal T\cup \{h\cdot Y\;|\;h\in\Gamma\})$ and $add(\mathcal T\cup \{h\cdot Y'\;|\;h\in\Gamma\})$ are rigid;

(4)~ $Y\notin \mathcal T$ and $Y'\notin \mathcal T$;

(5)~ $Y$ and $Y'$ are indecomposable;

(6)~ $add(\{h\cdot X\;|\;h\in \Gamma\})\cap add(\{h\cdot Y\;|\;h\in \Gamma\})=0$ and $add(\{h\cdot X\;|\;h\in \Gamma\})\cap add(\{h\cdot Y'\;|\;h\in \Gamma\})=0$.

\end{Lemma}

In the case $\Gamma$ is an infinite group given in Definition \ref{def3.1},  the proof of this lemma is exactly parallel with that of Proposition  3.7 in \cite{D}.

In order to be used in the sequel, write \\
 \centerline{ $\mu_X^r(\mathcal T, X)=add(\mathcal T\cup \{h\cdot Y\;|\;h\in\Gamma\})$ \;\;\;and\;\;\; $\mu_X^l(\mathcal T, X)=add(\mathcal T\cup \{h\cdot Y'\;|\;h\in\Gamma\})$,}
 where $\mu_X^r$ and $\mu_X^l$ are called respectively the {\bf right} and {\bf left mutation} of $\mathcal T$ at $X$, $(X,Y)$ and $(X,Y')$ are called respectively a {\bf right exchange pair} and {\bf left exchange pair}.

Recall that in \cite{D}, two indecomposable objects $X,Y\in \mathcal C$ are called {\bf neighbours}  if they satisfy one of the following equivalent conditions:
\begin{equation}
\begin{array}{ccl} dim Ext_{\mathcal C}^1(h\cdot X, Y) &=&
         \left\{\begin{array}{ll}
             1, &\mbox{if $h\cdot X\cong X$}, \\
             0, &\mbox{otherwise.}
         \end{array}\right.
      \end{array}
\end{equation}
\begin{equation}
\begin{array}{ccl} dim Ext_{\mathcal C}^1(X, h\cdot Y) &=&
         \left\{\begin{array}{ll}
             1, &\mbox{if $h\cdot Y\cong Y$}, \\
             0, &\mbox{otherwise.}
         \end{array}\right.
      \end{array}
\end{equation}

\begin{Proposition}\label{mutationrigid}(Proposition 3.14, \cite{D})
Let $X,Y\in \mathcal C$ be neighbours such that $add(\{h\cdot X\;|\;h\in \Gamma\})$ and $add(\{h\cdot Y\;|\;h\in \Gamma\})$ are rigid.
Let $0\rightarrow X\overset{f}{\rightarrow} W\overset{g}{\rightarrow} Y\rightarrow 0$ be a non-splitting admissible short exact sequence (which is unique up to isomorphism
because $X$ and $Y$ are neighbours). Then,

(1) $add(\{h\cdot X, h\cdot M\;|\;h\in \Gamma\})$ and $add(\{h\cdot Y, h\cdot M\;|\;h\in \Gamma\})$ are rigid;

(2) $X,Y\notin add\{h\cdot M\;|\;h\in \Gamma\}$;

(3) if there is a $\Gamma$-stable subcategory $\mathcal T\in \mathfrak{Add}(\mathcal C)^{\Gamma}$ such that $add(\mathcal T, \{h\cdot X\;|\;h\in \Gamma\})$ and $add(\mathcal T, \{h\cdot Y\;|\;h\in \Gamma\})$ are maximal $\Gamma$-stable rigid, then $f$ is a minimal left $\mathcal T$-approximation and $g$ is a minimal right $\mathcal T$-approximation.
\end{Proposition}

\begin{Corollary}\label{neighbour}(Corollary 3.15, \cite{D})
Let $\mathcal T\in \mathfrak{Add}(\mathcal C)^{\Gamma}$ and $X\in \mathcal C$ be an indecomposable object such that $X\notin \mathcal T$ and $add(\mathcal T, \{h\cdot X\;|\;h\in \Gamma\})$ is maximal $\Gamma$-stable rigid. Assume that $\mathcal T$ contains all projective objects and $X$ has left and right $\mathcal T$-approximation. Then, the following are equivalent:

(1) There exists an indecomposable object $Y\in \mathcal C$ such that $$\mu_X^l(add(\mathcal T, \{h\cdot X\;|\;h\in\Gamma \}))=add(\mathcal T, \{h\cdot Y\;|\;h\in\Gamma \})$$ and $X, Y$ are neighbours.

(2) There exists an indecomposable object $Y'\in \mathcal C$ such that $$\mu_X^r(add(\mathcal T, \{h\cdot X\;|\;h\in\Gamma \}))=add(\mathcal T, \{h\cdot Y'\;|\;h\in\Gamma \})$$ and $X, Y'$ are neighbours.

In this case, we have $Y\cong Y'$ and denote that
$$\mu_X(add(\mathcal T, \{h\cdot X\;|\;h\in\Gamma \}))=\mu_X^l(add(\mathcal T, \{h\cdot X\;|\;h\in\Gamma \}))=\mu_X^r(add(\mathcal T, \{h\cdot X\;|\;h\in\Gamma \})),$$
then $\mu_Y(\mu_X(add(\mathcal T, \{h\cdot X\;|\;h\in\Gamma \})))=add(\mathcal T, \{h\cdot X\;|\;h\in\Gamma \}).$
\end{Corollary}

Note that the final formula in this corollary can be regarded as the analogue of the involution of a mutation of matrix.

\begin{Lemma}\label{loops-neighbour}(Lemma 3.25 \cite{D})
Let $\mathcal T'\in \mathfrak{Add}(\mathcal C)^{\Gamma}$ and $X$ be indecomposable in $\mathcal C$ such that $add(\mathcal T'\cup\{h\cdot X\;|\;h\in \Gamma\})$ is rigid. Assume that $\mathcal T'$ contains all projective objects and $X\notin \mathcal T'$ has left and right minimal $\mathcal T'$-approximations. Let $(X,Y)$ be a left (respectively, right) exchange pair associated to $\mathcal T'$ such that $add(\mathcal T'\cup\{h\cdot X\;|\;h\in \Gamma\})=\mathcal T$. The following are equivalent:

(1) $\mathcal T$ has no $\Gamma$-loops at $X$.

(2) For any indecomposable $X'\in \{h\cdot X\;|\;h\in \Gamma\}$, every non-invertible morphism from $X$ to $X'$ factorizes through $\mathcal T'$.

(3) $X$ and $Y$ are neighbours.
\end{Lemma}

In the case $\Gamma$ is infinite,  the proofs of Proposition \ref{mutationrigid}, Corollary \ref{neighbour} and Lemma \ref{loops-neighbour} are parallel respectively with that of Proposition 3.14, Corollary 3.15 and Lemma 3.25 in \cite{D}.

In summary, we have:

\begin{Lemma}\label{mutation}
Let $\mathcal T'\in \mathfrak{Add}(\mathcal C)^{\Gamma}$ and $X$ be indecomposable in $\mathcal C$ such that $\mathcal T:=add(\mathcal T'\cup\{h\cdot X\;|\;h\in \Gamma\})$ is maximal $\Gamma$-stable rigid. Assume that $\mathcal T'$ contains all projective objects and $X\notin \mathcal T'$ has left and right minimal $\mathcal T'$-approximations. If $\mathcal T$ has no $\Gamma$-loops at $X$, then

(1) there exists an indecomposable object $Y\in\mathcal C$ such that $(X,Y)$ are neighbours,

(2) $\mu_{X}(\mathcal T)=add(\mathcal T',\{h\cdot Y\;|\;h\in \Gamma\})$ is $\Gamma$-stable rigid.
\end{Lemma}

\subsection{Global dimension of maximal $\Gamma$-stable rigid subcategories $\mathcal T$ and categories $\mathcal C_{h}(\mathcal T)$}

In this subsection, we study the global dimension of maximal $\Gamma$-stable rigid subcategories and the categories $\mathcal C_{h}(\mathcal T)$. Fix $h\in \Gamma$ such that $\mathcal C_h$ is $Hom$-finite, denote $\Gamma'=(h)$ the subgroup of $\Gamma$ generated by $h$.

Let $\mathcal T$ be a $\Gamma$-stable rigid subcategory of $\mathcal C$ which contains all projective objects. From now on, we always assume that $\mathcal T=add\{h\cdot \bigoplus\limits_{i=1}^n X_i\;|\;h\in \Gamma\}$ for indecomposable objects $X_i$ such that $X_i\notin \{h\cdot X_j\;|\;h\in \Gamma\}$ for $i\neq j$ until the end the this subsection.

\begin{Proposition}\label{globaldim'}(Proposition 3.26, \cite{D})
Keep the forgoing notations. Suppose that $\mathcal T=add\{h\cdot \bigoplus\limits_{i=1}^n X_i\;|\;h\in \Gamma\}$ is a maximal $\Gamma$-stable rigid subcategory of $\mathcal C_h(\mathcal T)$ without $\Gamma$-loops. Then,
 \[\begin{array}{ccl} gl.dim (mod\mathcal T)
  \left\{\begin{array}{ll}\leq
2, &\mbox{if any object of $\mathcal T$ is projective}, \\
=3, &\mbox{otherwise.}
\end{array}\right.
\end{array}\]
\end{Proposition}

\begin{proof} The proof in the case of a finite group is suitable to $\Gamma$ as an infinite group.
\end{proof}

For $X\in \mathcal C_h(\mathcal T)$  which is also indecomposable in $\mathcal T$, denote $A:=End_{\mathcal C_h(\mathcal T)}(X)$ the endomorphism algebra of $X$ in $\mathcal C_h(\mathcal T)$ and
\begin{equation}\label{J}
J=J(A):=\{(f_{h'})_{h'\in \Gamma'}\in A\;|\;f_{h'} \;\text{is not an isomorphism for any}\;h'\in \Gamma'\;\}.
\end{equation}

\begin{Lemma}\label{radical}
Keep the notations as above. Assume that $\mathcal C_h$ is $Hom$-finite. Then $J$ is the Jacobson radical of $A$.
\end{Lemma}

\begin{proof}
Since $A=End_{\mathcal C(\mathcal T)}(X)$ is local, $J$ is an ideal of the algebra $A$. By assumption, $\mathcal C_h$ is $Hom$-finite, then $End_{\mathcal C_h(\mathcal T)}(X)=\bigoplus\limits_{h'\in \Gamma'}Hom_{\mathcal C}(X,h'\cdot X)$ is finite dimensional for any $X\in \mathcal C$. Following this, we know that either  (a) $\Gamma'=(h)$ is a finite group, or (b) $\Gamma'$ is an infinite group but the subgroup $\Gamma'_X:=\{h'\in \Gamma'\;|\;h'\cdot X\cong X\}$ is a trivial group. Otherwise, if $\Gamma'$ is an infinite group but $\Gamma'_X$ is a non-trivial subgroup of $\Gamma'$, then $\Gamma'_X$ is also infinite and thus, $End_{\mathcal C_h}(X)$ is infinite dimensional, which is a contradiction.

(a) In case $\Gamma'$ is a finite group. Assume that $h$ has finite order $m$. Now we show that $J$ is a nilpotent ideal of $A$. Any $f=(f_0,\cdots,f_{m-1})\in J$ determinate a morphism $F=(F_{ij})_{i,j\in[1,m]}\in End_{\mathcal C}(\bigoplus\limits_{k=0}^{m-1}h^k\cdot X)$ with $F_{ij}=h^{i-1}\cdot f_{j-i}$, where the plus $j-i$ is in $\mathbb Z_m$ and $f_i:X\rightarrow h^i\cdot X$. According to the composition in $\mathcal C_h(\mathcal T)$, it is easy to see that $f^k$ is the first row of $F^k$ for any $k\in \mathbb N$. Thus, to prove that $f$ is nilpotent, it suffices to prove that $F$ is nilpotent. It suffices to show that all eigenvalues of $F$ as a linear map are equal to $0$ since $K$ is algebraical closed. For any $\lambda(\neq 0)\in K$, since $f_i$ are non-isomorphic and $End_{\mathcal C}(X)$ is local, we have $\lambda-F_{ii}$ are isomorphic and $-F_{ij}$ ($i\neq j$) are non-isomorphic. Since $End_{\mathcal C}(X)$ is local, an isomorphic endomorphism plus a non-isomorphic endomorphism is isomorphic. Therefore, we can do invertible row and column operations on $\lambda-F$ to obtain an isomorphic endomorphism. Thus, we have $\lambda-F$ is isomorphic for any $\lambda\neq 0$, which implies $F$ is nilpotent and $f$ is nilpotent. Therefore, $J$ is a nilpotent ideal of $A$. In addition, $A$ is finite dimensional, $J$ is included in the Jacobson radical of $A$.

To show $J$ is the Jacobson radical of $A$, it suffices to prove that $A/J$ is semi-simple.

Let $\Gamma'_X=\{h'\in \Gamma'\;|\;h'\cdot X\cong X\}$. Since $\Gamma'$ is a cyclic group, we may assume that $\Gamma'_X=(h^k)$ for some $k\in \mathbb N$. Denote the order of $h^k$ be $s$. Fixed an isomorphism $\varphi_1:X\rightarrow h^k\cdot X$ such that $\varphi_1^s-id_X$ is nilpotent. (We can always choose such $\varphi_1$ since for any isomorphism $\varphi:X\rightarrow h^k\cdot X$, assume that $\varphi^s-\lambda_0id_X$ is nilpotent, then we can set $\varphi_1=\sqrt[s]{\lambda_0}\varphi$.) Let $\varphi_i=\varphi_1^i:X\rightarrow h^{ki}\cdot X$ for $i=1,\cdots, s$. Let $\pi:End_{\mathcal C}(X)\rightarrow End_{\mathcal C}(X)/M\overset{\cong}{\rightarrow} K$ be the canonical surjective algebra homomorphism, where $M$ is the unique maximal ideal of $End_{\mathcal C}(X)$.

We define an algebra homomorphism $\theta:A/J\rightarrow K[\Gamma'_X]$ like this: for any $f:X\rightarrow h^i\cdot X$ ($0\leq i< sk$), define $\theta(f)=0$ if $h^i\notin \Gamma'_X$, $\theta(f)=\pi(\varphi^{-1}_{\frac{i}{k}}f\varphi_s)h^i$ if $h^i\in \Gamma'_X$. Now we prove that $\theta$ defines an algebra isomorphism from $A/J$ to $K[\Gamma'_X]$.

It is clear that $\theta$ is a well defined map, since if $f$ is non-isomorphic, $\theta(f)=0$. For $f:X\rightarrow h^{ki}\cdot X$ ($0\leq i< s$) and $f':X\rightarrow h^{kj}\cdot X$ ($0\leq j< s$), denote $\theta(f)=\lambda_ih^{ki}$ and $\theta(f')=\lambda_jh^{kj}$. Thus, we have $\varphi_{i}^{-1}f\varphi_s=\lambda_iid_X+g_i$ and $\varphi^{-1}_{j}f'\varphi_s=\lambda_jid_X+g_j$ for nilpotent homomorphisms $g_i$ and $g_j$. Consider the following commutative diagram:
$$\xymatrix{
X\ar[r]\ar[d]^{\varphi_s}   & X\ar[r] \ar[d]^{\varphi_i}     & X\ar[d]_{\varphi_{t}} \\
X\ar[r]^{f}                 & h^{ki}\cdot X \ar[r]^{h^{ki}\cdot f'}      & h^{ki+kj}\cdot X ,}$$
where $t\in [1,s]$ such that $t\equiv i+j(mod\; s)$. Denote $\varphi^{-1}_t(h^{ki}\cdot f')\varphi_i=\varphi^{-t}_1(h^{ki}\cdot f')\varphi_i=\lambda id_X+g$ for a nilpotent homomorphism $g$. Therefore, $\theta(f'\circ f)=\pi(\varphi^{-1}_{t}(h^{ki}\cdot f')f\varphi_s)=\lambda\lambda_ih^{ki+kj}$. Thus, to show $\theta(f\circ f')=\theta(f)\theta(f)$, it suffices to show that $\lambda=\lambda_j$. Consider the following commutative diagram:
$$\xymatrix{
X\ar[r]^{\varphi_i}\ar[d]   & h^{ki}X\ar[r]^{h^{ki}\cdot \varphi_s} \ar[d]     & h^{ki}\cdot X\ar[d]_{h^{ki}\cdot f'} \\
X\ar[r]^{\varphi_i}         & h^{ki}X\ar[r]^{h^{ki}\cdot \varphi_j}            & h^{ki+kj}\cdot X ,}$$ since $\varphi^{-1}_{j}f'\varphi_s=\lambda_jid_X+g_j$, we have
$$(h^{ki}\cdot \varphi^{-1}_{j})(h^{ki}\cdot f')(h^{ki}\cdot \varphi_s)=h^{ki}\cdot (\lambda_jid_X)+h^{ki}\cdot g_j=\lambda_jid_{h^{ki}\cdot X}+h^{ki}\cdot g_j.$$ Therefore,
$$\varphi^{-1}_i(h^{ki}\cdot \varphi^{-1}_{j})(h^{ki}\cdot f')(h^{ki}\cdot \varphi_s)\varphi_i=\varphi^{-1}_i(\lambda_jid_{h^{ki}\cdot X}+h^{ki}\cdot g_j)\varphi_i=\lambda_jid_{X}+\varphi^{-1}_i(h^{ki}\cdot g_j)\varphi_i.$$
According to our construction, $\varphi_s-id_X$ is nilpotent, so $$\varphi^{-1}_i(h^{ki}\cdot \varphi^{-1}_{j})(h^{ki}\cdot f')\varphi_i-\lambda_jid_X=\varphi_1^{-(i+j)}(h^{ki}\cdot f')\varphi_i-\lambda_jid_X$$ is nilpotent. Moreover, $i+j\equiv t(mods)$ and $\varphi_s-id_X=\varphi_1^s-id_X$ is nilpotent. Thus, $\varphi_1^{-t}(h^{ki}\cdot f')\varphi_i-\lambda_jid_X$ is nilpotent. Since $\varphi_1^{-t}(h^{ki}\cdot f')\varphi_i-\lambda id_X$ is nilpotent, we obtain $\lambda=\lambda_j$. Therefore, $\theta$ is an algebra homomorphism. Since $\theta(\varphi_i)=\pi(\varphi^{-1}_i\varphi_i\varphi_s)h^{ki}=h^{ki}$ for all $i=0,\cdots,s-1$. Thus, $\theta$ is surjective. It is clear that $dim (A/J)\leq s$ and $dim (K[\Gamma'_X])=s$. Therefore, we obtain $\theta$ is an algebra isomorphism.

Since $char(K)=0$, $A/J\cong K[\Gamma'_X]$ is semi-simple. In addition, $J$ is nilpotent and $A$ is finite dimensional, we get $J$ is the Jacobson radical of $A$.

(b) In case $\Gamma'$ is an infinite group but $\Gamma'_X=\{h'\in \Gamma'\;|\;h'\cdot X\cong X\}$ is a trivial group. Since $End_{\mathcal C}(X)$ is local, $J$ is the unique maximal ideal of $A$. Thus, $J$ is the Jacobson radical of $A$.
\end{proof}

By the definition, in order to calculate the global dimension of $\mathcal C_h(\mathcal T)$, we have to calculate the projective dimension of any simple representations $S_X$, where $S_X$ is the simple top of indecomposable projective representation $Hom_{\mathcal C_h}(X,-)$ of $\mathcal C_h(\mathcal T)$ for an indecomposable object $X$. We divide the discussion into two parts: (I) $X$ is not projective as an object of $\mathcal C$, (II) $X$ is projective as an object of $\mathcal C$.

(I) In the case that $X$ is not projective as an object of $\mathcal C$.

Let $X,Y$ be neighbours and $0\rightarrow X\rightarrow T\rightarrow Y\rightarrow 0$ be the admissible short exact sequence defined in Lemma \ref{leftrightmu}. For $h'\in \Gamma'$, applying $Hom_{\mathcal C}(-,h'\cdot X)$ to the admissible short exact sequence, we get the exact sequence:  $$Hom_{\mathcal C}(X, h'\cdot X) \overset{\delta_{h'}}{\longrightarrow} Ext^1_{\mathcal C}(Y, h'\cdot X)\rightarrow Ext^1_{\mathcal C}(Y, h'\cdot T)=0,$$ with $\delta_{h'}$ as the connecting morphism, where the last equality is from the rigidity of $\mu_X(\mathcal T)$ and $Y, h'\cdot T\in \mu_X(\mathcal T)$. Thus, $$\bigoplus\limits_{h'\in \Gamma'}Hom_{\mathcal C}(X, h'\cdot X) \overset{(\delta_{h'})_{h'\in \Gamma'}}{\longrightarrow} \bigoplus\limits_{h'\in \Gamma'}Ext^1_{\mathcal C}(Y, h'\cdot X)\rightarrow 0.$$

 For any $f:X\rightarrow h'\cdot X$ and $\xi\in Ext^1_{\mathcal C}(Y, h''\cdot X)$, define $f\xi$ via the following push out diagram:
\begin{equation}\label{modaction}
\xymatrix{
       \xi:  0\ar[r] &h''\cdot X\ar[r]\ar[d]_{h''\cdot f} &Z\ar[r]\ar[d]  & Y\ar[r]\ar[d]_{id_Y} &0 \\
       f\xi: 0\ar[r] &h''h'\cdot X\ar[r]            &Z'\ar[r]       & Y\ar[r]              &0,}
       \end{equation}
where $Z'=(h''h'\cdot X)\prod_{(h''\cdot X)}Z$.

\begin{Lemma}\label{simple}
Keep the notations as above. Assume that $\mathcal C_h$ is $Hom$-finite and $X$ is indecomposable non-projective in $\mathcal C$. Then,

(1)~
 $\bigoplus\limits_{h'\in \Gamma'}Ext^1_{\mathcal C}(Y, h'\cdot X)$ admits a left $A$-module structure defined via the expansion of the action given in the above (\ref{modaction}).

(2)~ $A=\bigoplus\limits_{h'\in \Gamma'}Hom_{\mathcal C}(X, h'\cdot X) \overset{\delta}{\longrightarrow} \bigoplus\limits_{h'\in \Gamma'}Ext^1_{\mathcal C}(Y, h'\cdot X)\rightarrow 0$, where $\delta=(\delta_{h'})_{h'\in \Gamma'}$. Then $\delta$ is an $A$-module homomorphism and $ker(\delta)$ is the Jacobson radical of $A$, where $A$ is viewed as a left regular $A$-module.
\end{Lemma}

\begin{proof}
(1) It is clear that $id_X\xi=\xi$ for all $\xi\in Ext^1_{\mathcal C}(Y, h'\cdot X)$. For any $f:X\rightarrow h'\cdot X$, $f':X\rightarrow h''\cdot X$ and $\xi\in Ext^1_{\mathcal C}(Y,h'''\cdot X)$, we have the following commutative diagram:
$$\xymatrix{
       \xi:      0\ar[r] &h'''\cdot X\ar[r]\ar[d]_{h'''\cdot f'}    &Z\ar[r]\ar[d]   & Y\ar[r]\ar[d]_{id_Y} &0 \\
       f'\xi:    0\ar[r] &h'''h''\cdot X\ar[r]\ar[d]_{h'''h''\cdot f} &Z'\ar[r]\ar[d]  & Y\ar[r]\ar[d]_{id_Y} &0 \\
       f(f'\xi): 0\ar[r] &h'''h''h'\cdot X\ar[r]               &Z''\ar[r]       & Y\ar[r]              &0,}$$
where $Z'=(h'''h''\cdot X)\coprod_{h'''\cdot X}Z$ and $Z''=(h'''h''\cdot Z)\coprod_{(h'''h''\cdot X)}Z'$. Thus, $Z''=(h'''h''\cdot Z)\coprod_{(h'''h''\cdot X)}Z'=(h'''h''\cdot Z)\coprod_{(h'''h''\cdot X)}(h'''h''\cdot X)\coprod_{h'''\cdot X}Z=(h'''h''\cdot Z)\coprod_{h'''\cdot X}Z$. Therefore, we have $f(f'\xi)=(ff')\xi$. Our result follows.

(2) According to Lemma \ref{radical}, $J=\{(f_{h'})_{h'\in \Gamma'}\in A\;|\;f_{h'} \;\text{is not an isomorphism}\;\}$.

For any $h'\in \Gamma'$ and $f:X\rightarrow h'\cdot X$, we have $\delta_{h'}(f)\in Ext^1_{\mathcal C}(Y,h'\cdot X)$ by the following diagram:
$$\xymatrix{
                        0\ar[r] &X\ar[r]\ar[d]_{f} &T\ar[r]\ar[d]  & Y\ar[r]\ar[d]_{id_Y} &0 \\
       \delta_{h'}(f): 0\ar[r] &h'\cdot X\ar[r]    &T'\ar[r]       & Y\ar[r]              &0,}$$
where $T'=(h'\cdot X)\coprod_{X}T$.

For any $h''\in \Gamma'$ and $f':X\rightarrow h''\cdot X$, we have the commutative diagram:
$$\xymatrix{
                           0\ar[r] &X\ar[r]\ar[d]_{f}                  &T\ar[r]\ar[d]   & Y\ar[r]\ar[d]_{id_Y} &0 \\
           \delta_{h'}(f):0\ar[r] &h'\cdot X\ar[r]\ar[d]_{h'\cdot f'} &T'\ar[r]\ar[d]  & Y\ar[r]\ar[d]_{id_Y} &0 \\
       f'(\delta_{h'}(f)):0\ar[r] &h'h''\cdot X\ar[r]        &T''\ar[r]       & Y\ar[r]              &0,}$$
where $T'=(h'\cdot X)\coprod_{X}T$ and $T''=(h'h''\cdot X)\coprod_{(h'\cdot X)}T'$. Thus, $T''=(h'h''\cdot X)\coprod_{X}T$. Therefore, we get $f'(\delta_{h'}(f))=\delta_{h'}(f'\circ f)$. Hence, $\delta=(\delta_{h'})_{h'\in \Gamma'}$ is an $A$-module homomorphism.

Now we prove that $ker(\delta)=J$, where $J$ is the Jacobson radical of $A$.

 Give any non-isomorphic morphism $f_{h'}:X\rightarrow h'\cdot X$. If $h'\cdot X\not\cong X$, then $Ext^1_{\mathcal C}(Y,h'\cdot X)=0$. Hence, $\delta_{h'}(f_{h'})=0$. If $h'\cdot X\cong X$, then $h'\in \Gamma'_X$ has finite order $m$. Thus, $f_{h'}^m\in Hom_{\mathcal C}(X,X)$. Since $Hom_{\mathcal C}(X,X)$ is local, and $f_{h'}$ is non-isomorphic, so $f_{h'}^m$ is nilpotent. Therefore, $f_{h'}$ is nilpotent. Moreover, since $dim Ext^1_{\mathcal C}(Y,h'\cdot X)=1$, $\delta(f_{h'})=\delta_{h'}(f_{h'})=0$. Thus, $J\subseteq ker(\delta)$.

For any isomorphic morphism $f_{h'}:X\rightarrow h'\cdot X$, we have the commutative diagram:
$$\xymatrix{
                        0\ar[r] &X\ar[r]\ar[d]_{f_{h'}} &T\ar[r]\ar[d]  & Y\ar[r]\ar[d]_{id_Y} &0 \\
       \delta_{h'}(f_{h'}): 0\ar[r] &h'\cdot X\ar[r]    &T'\ar[r]       & Y\ar[r]              &0.}$$
If $\delta_{h'}(f_{h'})=0$, then $\delta_{h'}(f_{h'})=0$ is split. Further, since $f_{h'}$ is isomorphic, $0\rightarrow X\rightarrow T\rightarrow Y\rightarrow 0$ is split, which makes a contradiction. Thus, $\delta(f_{h'})=\delta_{h'}(f_{h'})\neq 0$. Hence, $ker(\delta)\subseteq J$.

Then, we obtain $ker(\delta)=J$.
\end{proof}

(II) In the case that  $X$ is projective as an object in $\mathcal C$.

By the assumption of the beginning of this section, let $X\twoheadrightarrow Y$ be the left rigid quasi-approximation of $X$. Denote by $M$ the cokernel of $\bigoplus\limits_{h'\in \Gamma'}Hom_{\mathcal C}(Y,h'\cdot X)\rightarrow \bigoplus\limits_{h'\in \Gamma'}Hom_{\mathcal C}(X,h'\cdot X)$.
Thus, we obtain the exact sequence $\bigoplus\limits_{h'\in \Gamma'}Hom_{\mathcal C}(Y,h'\cdot X)\rightarrow \bigoplus\limits_{h'\in \Gamma'}Hom_{\mathcal C}(X,h'\cdot X)\overset{\varepsilon}{\rightarrow} M\rightarrow 0.$

\begin{Lemma}\label{simplepro}
Keep the notations as above. Assume that $\mathcal C_h(\mathcal T)$ is $Hom$-finite and $X$ is an indecomposable projective object in $\mathcal C$.
Then $M$ given above is an $A$-module and $M\cong A/J$, where $J$ is the Jacobson radical of $A$.
\end{Lemma}

\begin{proof}
According to Lemma \ref{radical}, $J=\{(f_{h'})_{h'\in \Gamma'}\in A\;|\;f_{h'} \;\text{is not an isomorphism}\;\}$. It suffices to prove that the kernel of $\varepsilon$ equals to $J$. It is clear that $ker(\varepsilon)=J$ by the definition of left rigid quasi-approximation.
\end{proof}

Now we can already calculate the global dimension of $\mathcal C_h(\mathcal T)$.

\begin{Proposition}\label{globaldim}
Keep the forgoing notations. Suppose that $\mathcal T=add\{h\cdot \bigoplus\limits_{i=1}^n X_i\;|\;h\in \Gamma\}$, with  indecomposables $X_i (i=1,\cdots,n)$, is a maximal $\Gamma$-stable rigid subcategory of $\mathcal C$ without $\Gamma$-loops. Assume  $\mathcal C_h(\mathcal T)$ is $Hom$-finite for $h\in \Gamma$, then
 \[\begin{array}{ccl} gl.dim (\mathcal C_h(\mathcal T))
  \left\{\begin{array}{ll}\leq
2, &\mbox{if all objects in $\mathcal T$ are projective}, \\
=3, &\mbox{otherwise.}
\end{array}\right.
\end{array}\]
\end{Proposition}

\begin{proof}
The $X_i (i=1,\cdots,n)$ are also all indecomposables of $\mathcal C_h(\mathcal T)$. So, $gl.dim(\mathcal C_h(\mathcal T))=max\{pd.dim(S_{X_i})\;|i=1,\cdots,n\},$ where $S_{X_i}: \mathcal C_h(\mathcal T)\rightarrow K-vec$ with $S_{X_i}(X)=Hom_{\mathcal C_h(\mathcal T)}(X_i, X)/J(X_i, X)$.    Then, we have to compute  $pd.dim(S_{X_i})$ for any fixed $X_i$.

Suppose firstly that $X_i$ is not projective.
Denote $\mathcal T'_{(i)}=add(\{h'\cdot X_j\;|\;h'\in \Gamma,j\neq i\})$ and $\Gamma'=(h)\subseteq \Gamma$.
 Since $\mathcal T$ has no $\Gamma$-loops, by Lemma \ref{mutation}, there exists $Y\in \mathcal C$ such that $(X_i,Y)$ are neighbours and $\mu_X(\mathcal T)=add(\mathcal T'_{(i)},\{h'\cdot X\;|\;h'\in \Gamma\})$ is $\Gamma$-stable rigid. Moreover, using Corollary \ref{neighbour}, we get two admissible short exact sequences,
$$0\rightarrow X_i\rightarrow T'\rightarrow Y\rightarrow 0\;\;\;\;\text{and}\;\;\;\;0\rightarrow Y\rightarrow T''\rightarrow X_i\rightarrow 0$$ for $T',T''\in \mathcal T'_{(i)}$.

For each $h'\in \Gamma'$ and $X\in\mathcal T$, applying $Hom_{\mathcal C}(-,h'\cdot X)$ to the above sequences, we get the two long exact sequences,
$$0\rightarrow Hom_{\mathcal C}(Y,h'\cdot X)\rightarrow Hom_{\mathcal C}(T',h'\cdot X)\rightarrow Hom_{\mathcal C}(X_i,h'\cdot X)\rightarrow Ext^1_{\mathcal C}(Y,h'\cdot X)\rightarrow Ext^1_{\mathcal C}(X_i,h'\cdot X)=0,$$
and
$0\rightarrow Hom_{\mathcal C}(X_i,h'\cdot X)\rightarrow Hom_{\mathcal C}(T'',h'\cdot X)\rightarrow Hom_{\mathcal C}(Y,h'\cdot X)\rightarrow Ext^1_{\mathcal C}(X_i,h'\cdot X)=0.$

 Combing the two long exact sequences, and by the arbitrary of $h'$, we can obtain  that
\begin{equation}\label{projdim}
0\rightarrow Hom_{\mathcal C_h(\mathcal T)}(X_i,-)\rightarrow Hom_{\mathcal C_h(\mathcal T)}(T'',-)\rightarrow Hom_{\mathcal C_h(\mathcal T)}(T',-)\rightarrow Hom_{\mathcal C_h(\mathcal T)}(X_i,-)\rightarrow F\rightarrow 0,
 \end{equation}
 where the functor $F: \mathcal C_h(\mathcal T)\rightarrow K\text{-}vec$ with $X\mapsto \bigoplus\limits_{h'\in \Gamma'}Ext^1_{\mathcal C}(Y, h'\cdot X)$.

For indecomposable $X\not\cong X_i$ in $\mathcal C_h(\mathcal T)$, we have $F(X)=\bigoplus\limits_{h'\in \Gamma'}Ext^1_{\mathcal C}(Y,h'\cdot X)=0$ since $\mu_{X_i}(\mathcal T)$ is $\Gamma$-stable rigid and then $Y,h'\cdot X\in\mu_{X_i}(\mathcal T)$. And,  $S_{X_i}(X)=Hom_{\mathcal C_h(\mathcal T)}(X_i, X)/J(X_i, X)=0$ since in this case $Hom_{\mathcal C_h(\mathcal T)}(X_i, X)=J(X_i, X)$ by \cite{ASS}. Then $F(X)=S_{X_i}(X)$.

In the case $X\cong X_i$,
by Lemma \ref{simple}, $F(X)\cong F(X_i)=\bigoplus\limits_{h'\in \Gamma'}Ext^1_{\mathcal C}(Y,h'\cdot X_i)\cong End_{\mathcal C_h(\mathcal T)}(X_i)/J$. And, $S_{X_i}(X)\cong S_{X_i}(X_i)=Hom_{\mathcal C_h(\mathcal T)}(X_i,X_i)/J(X_i,X_i)=End_{\mathcal C_h(\mathcal T)}(X_i)/J$. Hence, also, $F(X)=S_{X_i}(X)$.

In summary, we have $F\cong S_{X_i}$.
  As a consequence from (\ref{projdim}), $proj.dim(S_{X_i})\leq 3$. Since $X_i\notin \mathcal T'_{(i)}$ and $T''\in \mathcal T'_{(i)}$, we get $Hom_{mod\mathcal C_h(\mathcal T)}(Hom_{\mathcal C_h(\mathcal T)}(T'',-), S_{X_i})=0.$ Therefore,
  $$Ext^3_{mod\mathcal C_h(\mathcal T)}(S_{X_i},S_{X_i})\cong Hom_{mod\mathcal C_h(\mathcal T)}(Hom_{\mathcal C_h(\mathcal T)}(X_i,-), S_{X_i})\neq 0.$$ Thus, $proj.dim(S_{X_i})=3.$

Now, suppose that $X_i$ is projective, then $X_i$ is also injective in $\mathcal C^Q$. By the assumption at the beginning of this section, let $X_i\twoheadrightarrow Y$ be a left rigid quasi-approximation. According to the definition, the injective envelope of $Y$ does not intersect with $add(\{h'\cdot X\;|\;h'\in \Gamma\})$ and, by Lemma 2.38 of \cite{D}, there is an admissible short exact
sequence $$0\rightarrow Y\overset{f}{\rightarrow} T'\rightarrow Z\rightarrow 0,$$ where $f$ is a left $\mathcal T'_{(i)}$-approximation. Using Lemma \ref{3.6},  $add(\mathcal T, \{h'\cdot Z\;|\;h'\in\Gamma\})$ is $\Gamma$-stable rigid. Moreover, since $\mathcal T$ is maximal $\Gamma$-stable rigid, we have $Z\in \mathcal T$. Applying $Hom_{\mathcal C}(-,h'\cdot X)$ to the admissible sequence, we obtain
$$0\rightarrow Hom_{\mathcal C}(Z,h\cdot X)\rightarrow Hom_{\mathcal C}(T',h\cdot X)\rightarrow Hom_{\mathcal C}(Y, h\cdot X)\rightarrow Ext^1_{\mathcal C}(Z,h\cdot X)=0.$$
Thus, $0\rightarrow Hom_{\mathcal C_h(\mathcal T)}(Z,-)\rightarrow Hom_{\mathcal C_h(\mathcal T)}(T',-)\rightarrow Hom_{\mathcal C_h(\mathcal T)}(Y,-)\rightarrow 0.$

Let $G:\mathcal C_h(\mathcal T)\rightarrow K\text{-}vec$ be the functor satisfying that
$$G(X)=coker(\bigoplus\limits_{h'\in \Gamma'}Hom_{\mathcal C}(Y,h'\cdot X)\rightarrow \bigoplus\limits_{h'\in \Gamma'}Hom_{\mathcal C}(X_i, h'\cdot X)).$$
Thus, $0\rightarrow Hom_{\mathcal C_h(\mathcal T)}(Y,-)\rightarrow Hom_{\mathcal C_h(\mathcal T)}(X_i,-)\rightarrow G\rightarrow 0.$
Therefore, we get
\begin{equation}\label{proj}
0\rightarrow Hom_{\mathcal C_h(\mathcal T)}(Z,-)\rightarrow Hom_{\mathcal C_h(\mathcal T)}(T',-)\rightarrow Hom_{\mathcal C_h(\mathcal T)}(X_i,-)\rightarrow G\rightarrow 0.
\end{equation}

As $X_i\twoheadrightarrow Y$ is a left rigid quasi-approximation, we get $G(X)=0=S_{X_i}(X)$ when $X\not\cong X_i$ in $\mathcal C_h(\mathcal T)$. By Lemma \ref{simplepro}, we have $G(X_i)\cong A/J=S_{X_i}(X_i)$. Thus, $G\cong S_{X_i}$. As a consequence of (\ref{proj}), $pd.dim(S_{X_i})=pd.dim(G)\leq 2$.

Then the result follows.
\end{proof}

\section{Proof of the unfolding theorem of acyclic sign-skew-symmetric matrices}\label{quiver}

In this section, we use the previous preparation to prove that any acyclic sign-skew-symmetric cluster algebra has an unfolding corresponding to a skew-symmetric cluster algebra. As application, we have given a positive answer for the problem by Berenstein, Fomin, Zelevinsky on the totality of acyclic sign-skew-symmetric matrices in Theorem \ref{sign-skew}.

Let $B\in Mat_{n\times n}(\mathbb Z)$ be an acyclic sign-skew-symmetric matrix and $Q$ be the strongly almost finite quiver constructed in section \ref{strongly-finite}, $\Gamma$ be the group acts on $Q$. Let $\mathcal C^Q$ be the strongly almost finite Frobenius $2$-Calibi-Yau category associated to $Q$ which has been defined in section \ref{cate} before Corollary \ref{2-cy}. The action of $\Gamma$  on $\mathcal C^Q$ is exactly given from the action of $\Gamma$ acts on $Q$ as shown in (\ref{action}).

In this section, we firstly focus on the action of $\Gamma$ on the $\mathcal C^Q$.
 We prove that the action satisfies the assumptions at the beginning of section \ref{action}. By Corollary \ref{left-rigid}, every projective object in $\mathcal C^Q$ has a left rigid quasi-approximation. It remains to prove that for an indecomposable $X\in \mathcal C^Q$ and $h\in \Gamma$, there exists a $h'\in \Gamma$ such that $h'X\cong hX$ and $\mathcal C^Q_{h'}$ is $Hom$-finite. To see this, we need the following preparation.

Assume that $h\in \Gamma$ has no fixed points, for any indecomposable object $X\in \mathcal C^Q$, set $Q^X$ to be the sub-quiver of $Q$  generated by all vertices $a$ such that $X(a)\neq 0$, and $\Gamma'_{X}:=\{h'\in \Gamma'\;|\;h'X\cong X\}$. It is easy to see that there is a group homomorphism $\omega:\Gamma'_{X}\rightarrow Aut(Q^X), h'\rightarrow h'|_{Q^X}$, where $h'|_{Q^X}$ means the restriction of $h'$ to $Q^X$. Since $h$ has no fixed points, by Lemma \ref{fixedpoint}, we have $ker(\omega)=\{e\}$. Thus, $\Gamma'_X$ is a finite group.

\begin{Lemma}\label{Hom-finite} The category  $\mathcal C^Q_h$ is $Hom$-finite if one of two conditions holds:
(i)~ the order of $h$ is finite, (ii)~  $h\in \Gamma$ acts on $Q$ has no fixed points.
\end{Lemma}

\begin{proof}
If the order of $h$ is finite, then $\Gamma'=(h)$ is finite. Since $\mathcal C^Q$ is $Hom$-finite, $\mathcal C^Q_h$ is a $Hom$-finite category.

If $h$ has no fixed points. We may assume that $X$ is indecomposable. To show $Hom_{\mathcal C^Q_h}(Y,X)$ is finite dimensional. It suffices to prove there are only finite $h'\in \Gamma'$ such that $Q^{h'(X)}$ has common vertices with $Q^Y$.

Otherwise, since $Q^Y$ is a finite quiver, there are infinite $h'\in \Gamma'$ such that $Q^{h'\cdot X}$ contains a vertex of $Q^Y$, assume to be $1'$. Assume that $1'$ is a common vertex of $Q^{h'\cdot X}$ for $h'\in S_1$ for $S_1\subseteq \Gamma'$. Thus, $\{{h'}^{-1}\cdot 1\;|\;h'\in S_1\}\subseteq Q^X_0$. Moreover, as $Q^X$ is a finite quiver, there are infinite $h'\in S_1$ such that ${h'}^{-1}\cdot 1'$ have the same image. Without loss of generality, assume that ${h'}^{-1}\cdot 1'=1$ for a vertex $1\in Q^X_0$ and for all $h'\in S_1$. For a vertex $2\in Q^X_0$ with an arrow $\alpha:1\rightarrow 2$ (if we the orientation is inverse, the proof is the same), there is an arrow from $h'\cdot 1=1'$ to $h'\cdot 2$ for all $h'\in S_1$. Since $Q$ is strongly almost finite, there are only finite arrows starting from $1'$. Thus, there is a infinite subset $S_2\subseteq S_1$ such that $h'\cdot 2$ have common image for all $h'\in S_2$. Since $X$ in indecomposable, $Q^X$ is connected. Inductively, there is a infinite subset $S_n\subseteq S_{n-1}\subseteq \Gamma'$, such that $h'\cdot a$ have common image for all $h'\in S_n$ and vertices $a\in Q^X$. Thus, pick $h'\neq h''\in S_n$, we have $h'|_{Q^X}=h''|_{Q^X}$. Therefore, we have $h'{h''}^{-1}|_{Q^X}=id_{Q^X}$, which contradicts to Lemma \ref{fix} which says that any element in $\Gamma'$ has no fixed points.

Thus, in both cases, $\mathcal C^Q_h$ is $Hom$-finite.
\end{proof}

\begin{Proposition}\label{is-hom}
Let $X\in \mathcal C^Q$ be indecomposable and $h\in \Gamma$. Then, there exists an $h'\in \Gamma$ such that $h'\cdot X\cong h\cdot X$ and $\mathcal C^Q_{h'}$ is $Hom$-finite.
\end{Proposition}

\begin{proof}
If $h$ has no fixed points, then according to Lemma \ref{Hom-finite}, $\mathcal C^Q_h$ is $Hom$-finite. In this case, choose $h'=h$.
Otherwise, if $h$ has fixed points, by Lemma \ref{finiteorder}, there exists $h'\in \Gamma$ such that the order of $h'$ is finite and $h\cdot X\cong h'\cdot X$. By Lemma \ref{Hom-finite}, $\mathcal C^Q_{h'}$ is $Hom$-finite.
\end{proof}

Until now, we have proved that the action of  $\Gamma$ on $\mathcal C^Q$ satisfying the assumptions at the beginning of section \ref{action}.

A subcategory $\mathcal S$ of $\mathcal C^Q$ is called {\bf finitely non-Hom-orthogonal in $\mathcal C^Q$} if for each $X\in \mathcal C^Q$, there are only finite indecomposable objects $T$ (up to isomorphism) in $\mathcal S$ such that $Hom_{\mathcal C^Q}(X, T)\neq 0$ or $Hom_{\mathcal C^Q}(T,X)\neq 0$. It is easy to see a finitely non-Hom-orthogonal subcategory is functorially finite.

\begin{Lemma}\label{functorially finite}
Let $\mathcal T=add(\{hX_i\;|\;i=1,\cdots,n, h\in \Gamma\})$ be a maximal $\Gamma$-stable rigid subcategory of $\mathcal C^Q$ with each $X_i$ indecomposable, containing all projective-injective objects. If $\mathcal T$ is finitely non-Hom-orthogonal in  $\mathcal C^Q$, then $\mu_{X_i}(\mathcal T)$ is also finitely non-Hom-orthogonal in $\mathcal C^Q$ for each non-projective-injective indecomposable $X_i\in \mathcal T$.

In particular, $\mu_{X_i}(\mathcal T)$ is functorially finite.
\end{Lemma}

\begin{proof}
By Lemma \ref{mutation}, there exists indecomposable $Y_i$ such that $(X_i,Y_i)$ are neighbours. By Corollary \ref{neighbour}, there exist two admissible short exact sequences $0\rightarrow X_i\rightarrow T\rightarrow Y_i\rightarrow 0$ and $0\rightarrow Y_i\rightarrow T'\rightarrow X_i\rightarrow 0$ with $T,T'\in \mathcal T$. By the assumption of the conditions on $\mathcal T$, to prove our result, it suffices to prove that there are only finite indecomposable $h\cdot Y_i$ (up to isomorphism), $h\in \Gamma$ such that $Hom_{\mathcal C^Q}(h\cdot Y_i,X)\neq 0$ or $Hom_{\mathcal C^Q}(X,h\cdot Y_i)\neq 0$ for each $X\in \mathcal C^Q$. Since $Hom_{\mathcal C^Q}(h\cdot Y_i,X)\subseteq Hom_{\mathcal C^Q}(h\cdot T,X)$ and $Hom_{\mathcal C^Q}(X,h\cdot Y_i)\subseteq Hom_{\mathcal C^Q}(X,h\cdot T')$, it suffices to prove that there are only finite $h$ (up to the isomorphic of $h\cdot T$ and $h\cdot T'$) such that $Hom_{\mathcal C^Q}(h\cdot T,X)\neq 0$ or $Hom_{\mathcal C^Q}(X,h\cdot T')\neq 0$.

We only prove that there are only finite $h$ (up to the isomorphic of $hT$) such that $Hom_{\mathcal C^Q}(hT,X)\neq 0$. The case for $Hom_{\mathcal C^Q}(X,h\cdot T')\neq 0$ can be proved dually.

Assume that $T=\bigoplus\limits_{i=1}^m T_i^{n_i}$ for indecomposable objects $T_i$. Thus, $Hom_{\mathcal C^Q}(h\cdot T,X)\neq 0$ if and only if $Hom_{\mathcal C^Q}(h\cdot T_i,X)\neq 0$ for some $i$. Since $T_i\in \mathcal T$, there are only finite $h$ (up to the isomorphic of $h\cdot T_i$) such that $Hom_{\mathcal C^Q}(h\cdot T_i,X)\neq 0$, we may assume that $Hom_{\mathcal C^Q}(h_j\cdot T_i,X)\neq 0$ for $j=1,\cdots,s$.

Now we show there are only finite $h$ (up to the isomorphic of $h\cdot T$) such that $h\cdot T_i\cong h_j\cdot T_i$ for each $j=1,\cdots,s$. If we can do this, our result follows according to the above discussion. For each $j=1,\cdots, s$, we may assume that $h_j=e$ without loss of generality. For each $h\in \Gamma$ such that $h\cdot T_i\cong T_i$, since $T_i$ is indecomposable, by Lemma \ref{fix}, $h$ has fixed points. By Observation \ref{extend} (1), we may assume that $h$ fix a vertex labelled by 1. Then it follows by Lemma \ref{ext} immediately that there are only finite $h$ (up to the isomorphic of $h\cdot T$) such that $h\cdot T_i\cong T_i$.
\end{proof}

\begin{Proposition}\label{quivermu}
Let us keep the forgoing notations. Suppose that $\mathcal T$ is a maximal $\Gamma$-stable rigid subcategory which has no $\Gamma$-loops. Assume $Q'$ is the Gabriel quiver of $\mathcal T$ and $Q''$ is the Gabriel quiver of $\mu_{X_i}(\mathcal T)$ for a non-projective object $X_i$. Then
$Q''=\widetilde{\mu}_{[i]}(Q').$
\end{Proposition}

\begin{proof}
Since $Q$ is strongly almost finite, the proof is parallel with the proof of Theorem 7.8, \cite{GLS1}. In fact, we only need to replace the matrix $S$ in section 7.1 of \cite{GLS1} to $S=(s_{jk})_{j,k\in Q_0}$, where $\begin{array}{ccl} s_{jk} &=&
         \left\{\begin{array}{ll}
             -\delta_{jk}+\frac{|b_{jk}|-b_{jk}}{2}, &\mbox{if $j\in [i]$}, \\
             \delta_{jk}, &\mbox{otherwise.}
         \end{array}\right.
      \end{array}$
\end{proof}

This result means the preservability of Gabriel quivers under orbit mutations.

\begin{Proposition}\label{extention2}(\cite {D}, Proposition 3.31; \cite{B}, P463)
Let $Q$ be a strongly almost finite  quiver and $I$ be an admissible ideal of $KQ$.  Denote by $J$ the
ideal of $KQ$ generated by all of its arrows. Let $i,j\in Q_0$. Then,
$dim e_j(I/(IJ+JI))e_i=dim Ext^2_{modKQ/I}(S_i,S_j)$ for $S_i$, $S_j$ simple $KQ/I$-modules supported on vertices $i$ and $j$.
\end{Proposition}

\begin{proof}
Although the corresponding result in (\cite{B}, P463) deals with the algebra $KQ/I$ in the case of finite dimension from a finite quiver, we can see its proof can be applied to $KQ/I$ similarly when $Q$ is strongly almost finite. Precisely, since $Q$ is strongly almost finite, for any simple modules $S_i, S_j$  and a projective resolution of $P_3\rightarrow P_2\rightarrow P_1\rightarrow P_0\rightarrow S_i\rightarrow 0$ of $S_i$ in $modKQ/I$, we can choose a finite acyclic subquiver $Q'$ of $Q$ such that $S_i,S_j$ and $ P_k$ are all in $modKQ'$ for $k=0,1,2,3$ and $e_jKQ'e_i=e_jKQe_i$ holds. In this case, we denote $I'=KQ'\cap I$. Then we can regard $P_k$, $k=0,1,2,3$ as projective and $S_i, S_j$ as simple in $modKQ'/I'$. Thus, following the definition of the $Ext^2$-group, we obtain  $Ext^2_{modKQ/I}(S_i,S_j)=Ext^2_{modKQ'/I'}(S_i,S_j)$.

Since $I'=I\cap KQ'$, we have $e_jI'e_i=e_j(I\cap KQ')e_i=e_jIe_i\cap e_jKQ'e_i=e_jIe_i\cap e_jKQe_i=e_jIe_i$. Clearly, $e_jI'J'e_i\subseteq e_jIJe_i$. Write $\sum\limits_{s}k_se_ja_sb_se_i\in e_jIJe_i\subseteq e_jKQe_i$ with $0\neq k_s\in K$ and $a_s\in e_jIe_s$, $b_s\in e_sJe_i$, $a_sb_s\neq 0$. Since $e_jKQe_i=e_jKQ'e_i$, any vertex $s$ through a path in $Q$ from $i$ to $j$ is a vertex of $Q'$ and the paths from $i$ to $s$ and from $s$ to $j$ are in $Q'$. Hence, we have $a_s, b_s\in KQ'$. Thus, $a_s\in KQ'\cap I=I'$, $b_s\in KQ'\cap J=J'$, then $\sum\limits_{s}k_se_ia_sb_se_j\in e_jI'J'e_i$, this implies $e_jIJe_i\subseteq e_jI'J'e_i$. Hence, $e_jIJe_i=e_jI'J'e_i$. Similarly, $e_jJIe_i=e_jJ'I'e_i$. Then  $e_j(IJ+JI)e_i=e_j(I'J'+J'I')e_i$. Therefore, $e_j(I'/(I'J'+J'I'))e_i=e_j(I/(IJ+JI))e_i$. Applying the result in \cite{B} for $KQ'/I'$, we have $$dim Ext^2_{modKQ/I}(S_i,S_j)=dimExt^2_{modKQ'/I'}(S_i,S_j)=dim e_j(I'/(I'J'+J'I'))e_i=dim e_j(I/(IJ+JI))e_i.$$
\end{proof}

For a strongly almost finite quiver $Q$, we can view $KQ$ as a Krull-Schmidt $K$-category whose objects are vertices of $Q$ and their all finite direct sums $\bigoplus\limits_{i\in Q_0} i^{n_i}$ (here $n_i\in\mathbb N$) and whose morphisms $$Hom_{KQ}(\bigoplus\limits_{i\in Q_0} i^{n_i},\bigoplus\limits_{j\in Q_0} j^{m_j}):=\bigoplus\limits_{i,j\in Q_0} KQ(i,j)^{n_im_j}$$ for any objects $\bigoplus\limits_{i\in Q_0} i^{n_i},\bigoplus\limits_{j\in Q_0} j^{m_j}$ of $KQ$.
 Trivially, vertices of $Q$ are just all indecomposable objects in $KQ$.
 Any admissible ideal $I$ of the algebra $KQ$ can be view as a two-sided ideal in $KQ$ as category. Thus, $KQ/I$ can be also regarded as a quotient category.

\begin{Theorem}\label{Gabr}(Categorical Gabriel's Theorem)
Let $\mathcal A$ be a $Hom$-finite Krull-Schmidt category and $Q$ be its Gabriel quiver. If $Q$ is strongly almost finite, then there exists an admissible ideal $I$ of $KQ$ such that $\mathcal A$ is categorically equivalent to $KQ/I$.
\end{Theorem}

\begin{proof}
 The following proof is parallel with the Gabriel Theorem for basic algebras, see Chapter 2, Theorem 3.7 of \cite{AS}. According to the definition of Gabriel quiver, we can write the set of indecomposable objects (up to isomorphism) of $\mathcal A$ as  $\{\overline{i}\;|\;i\in Q_0\}$, and any path $p:i\rightarrow j$ in $Q$ corresponds to a morphism $\overline{p}:\overline{i}\rightarrow \overline{j}$. Define an additive functor $F:KQ\rightarrow \mathcal A$ as follows: $F(i)=\overline{i}$ and $F(p)=\overline{p}$ for any  $i,j\in Q_0$ and path $p:i\rightarrow j$. For paths $p_1:i\rightarrow j$ and $p_2:j\rightarrow k$ in $Q$, obviously $\overline{p_2}~\overline{p_1}=\overline{p_2p_1}$. Thus, we have $F(p_2p_1)=\overline{p_2}\overline{p_1}=F(p_2)F(p_1)$. Its additivity is defined via linear expansions. Therefore, $F$ is an additive functor.

Since $Q$ is strongly almost finite, we know that $Q(i,j)$ is finite for any $i,j\in Q_0$. Then there exists $m\in \mathbb N$ such that $rad^{m}_{\mathcal A}(\overline{i},\overline{j})=0$. For any morphism $f:\overline{i}\rightarrow \overline{j}$ of $\mathcal A$, we may assume that $f\in rad^{n}_{\mathcal A}(\overline{i},\overline{j})\setminus rad^{n+1}_{\mathcal A}(\overline{i},\overline{j})$ for some $n\in \mathbb N$, and decompose  $f=f^n\cdots f^1$ for irreducible morphisms $f^k=\left(\begin{array}{c} f^k_{ts}\end{array}\right)_{t\in \Lambda_{k+1},s\in \Lambda_{k}}\in Hom_{\mathcal A}(\bigoplus\limits_{s\in \Lambda_{k}} \overline{i_s}, \bigoplus\limits_{t\in \Lambda_{k+1}} \overline{i_t})$, $1\leq k\leq n+1$, where $\Lambda_k$ are finite index sets. In particular, $|\Lambda_1|=|\Lambda_{n+1}|=1$. So, write $\Lambda_1=\{s_1\}$ and $\Lambda_{n+1}=\{s_{n+1}\}$. Then $f=\left(\begin{array}{c} f^n_{ts}\end{array}\right)_{t\in \Lambda_{n+1},s\in \Lambda_{n}}\cdots \left(\begin{array}{c} f^1_{ts}\end{array}\right)_{t\in \Lambda_{2},s\in \Lambda_{1}}$ is exactly a $1\times 1$ matrix, or say, is just an element. Throught the multiplication of matrices, we can write $f$ into two parts, i.e.     $$f=\sum\limits_{(s_{n},\cdots,s_{2})\in \Lambda'}f^n_{s_{n+1}s_{n}}\cdots f^1_{s_{2}s_1}+\sum\limits_{(s_{n},\cdots,s_{2})\in \Lambda''}f^n_{s_{n+1}s_{n}}\cdots f^1_{s_{2}s_1},$$ where $\Lambda'=\{(s_{n},\cdots,s_{2})\;|\;s_k\in \Lambda_k, f^k_{s_{k+1}s_{k}}\;\text{are irreducible for all}\;k\}$, and $\Lambda''=\{(s_{n},\cdots,s_{2})\;|\;s_k\in \Lambda_k\}\setminus \Lambda'$. Thus, $f-\sum\limits_{(s_{n},\cdots,s_{2})\in \Lambda'}f^n_{s_{n+1}s_{n}}\cdots f^1_{s_{2}s_1}\in rad^{n+1}_{\mathcal A}(\overline{i},\overline{j})$.

We claim that there exists $\beta\in Hom_{KQ}(i, j)$ such that $f-F(\beta)\in rad^{n+1}_{\mathcal A}(\overline{i},\overline{j})$. In fact, for any $k$, by the definition of $Q$, the set $B_{i_{s_k}i_{s_{k+1}}}=\{F(\alpha_{s_{k}})+rad^2_{\mathcal A}(\overline{i_{s_k}},\overline{i_{s_{k+1}}})\;|\;\alpha_{s_{k}}\in Q_1(i_{s_k},i_{s_{k+1}})\}$ is a basis of $rad_{\mathcal A}(\overline{i_{s_k}},\overline{i_{s_{k+1}}})/rad^2_{\mathcal A}(\overline{i_{s_k}},\overline{i_{s_{k+1}}})$, where $Q_1(i_{s_k},i_{s_{k+1}})$ denotes the set of arrows from $i_{s_k}$ to $i_{s_{k+1}}$.  Then there exist $a_{s_{k}}\in K$ such that
\begin{equation}\label{rad2}
f^{k}_{s_{k+1}s_k}-\sum a_{s_{k}}F(\alpha_{s_{k}})\in rad^2_{\mathcal A}(\overline{i_{s_k}},\overline{i_{s_{k+1}}}).
\end{equation}
 Therefore, for any $(s_n,\cdots,s_2)\in \Lambda'$, we have
\[\begin{array}{ccl} && \prod\limits_{k=1}^nf^k_{s_{k+1}s_k} -F(\prod\limits_{k=1}^n\sum a_{s_{k}}\alpha_{s_{k}}) \\
  &=&   \prod\limits_{k=1}^n(f^k_{s_{k+1}s_k}-\sum a_{s_{k}}F(\alpha_{s_{k}})+\sum a_{s_{k}}F(\alpha_{s_{k}}))-\prod\limits_{k=1}^n(\sum a_{s_{k}}F(\alpha_{s_{k}}))\\
  &=&   \prod\limits_{k=1}^n(f^k_{s_{k+1}s_k}-\sum a_{s_{k}}F(\alpha_{s_{k}}))+E+\prod\limits_{k=1}^n(\sum a_{s_{k}}F(\alpha_{s_{k}})) -\prod\limits_{k=1}^n(\sum a_{s_{k}}F(\alpha_{s_{k}}))\\
   &=&   \prod\limits_{k=1}^n(f^k_{s_{k+1}s_k}-\sum a_{s_{k}}F(\alpha_{s_{k}}))+E\in rad^{n+1}_{\mathcal A}(\overline{i},\overline{j}),
\end{array}\]
by (\ref{rad2}) and since in the expansion $E$ each monomial includes at least once a factor in the form $f^k_{s_{k+1}s_k}-\sum a_{s_{k}}F(\alpha_{s_{k}})$.
Thus, we have $f-F(\sum\limits_{(s_n,\cdots,s_2)\in \Lambda'}\prod\limits_{k=n}^1\sum a_{s_{k}}\alpha_{s_{k}})=f-\sum\limits_{(s_n,\cdots,s_2)\in \Lambda'}F(\prod\limits_{k=n}^1\sum a_{s_{k}}\alpha_{s_{k}})\in rad^{n+1}_{\mathcal A}(\overline{i},\overline{j}).$
Hence, the claim follows.

Using the fact that $rad^{m}_{\mathcal A}(\overline{i},\overline{j})=0$ and the above claim inductively, there exists $g\in Hom_{KQ}(i, j)$ such that $F(g)=f$. Therefore, $F(i,j): Hom_{KQ}(i,j)\rightarrow Hom_{\mathcal A}(\overline i,\overline j)$ is surjective for all $i,j\in Q_0$.

It is easy to see that $F$ makes a bijection from $Ob(KQ)$ to $Ob(\mathcal A)/\cong$, where $\cong$ means the isomorphisms of objects. Let $I$ be the ideal of $KQ$ generated by $ker(F(i,j))$ for all $i,j\in Q_0$. Then, $Hom_{KQ/I}(i,j)\cong Hom_{\mathcal A}(\overline i,\overline j)$ for all $i,j\in Q_0$. Thus, as categories, we have $KQ/I\cong \mathcal A$.

It remains to prove that $I$ is admissible. Denote $rad_{KQ}$ the radical of $KQ$ as a category, which is generated by all arrows in $Q$. For any $s,t\in Q_0$ and $\sum\limits_{i\in Q_0} a_iid_{i}+\sum\limits_{\alpha_l\in Q_1} b_l\alpha_l+y\in ker(F(s,t))$ with  $y\in rad_{KQ}^2(s,t)$ and $a_i,b_l\in K$, it suffices to prove $a_i$=0 and $b_l=0$ for all $i,l$.

 In fact, $\sum a_i id_{\overline i}+\sum\limits_{\alpha_l\in Q_1}b_l F(\alpha_l)+F(y)=0$ as morphism in $\mathcal A$. Thus, $\sum a_i id_{\overline i}=-\sum\limits_{\alpha_l\in Q_1}b_lF(\alpha_l)-F(y)\in rad_{\mathcal A}(\overline s,\overline t)$. Since $\mathcal A$ is $Hom$-finite, we have $-\sum\limits_{\alpha_l\in Q_1}b_lF(\alpha_l)-F(y)$ is nilpotent. Thus, $a_i=0$ for all $i$. Then,  $\sum\limits_{\alpha_l\in Q_1}b_lF(\alpha_l)=-F(y)\in rad^2(\mathcal A)(\overline s,\overline t)$. Hence, $(\sum\limits_{\alpha_l\in Q_1}b_lF(\alpha_l))+rad^2(\mathcal A)(\overline s,\overline t)=0$ holds in $rad_{\mathcal A}(\overline s,\overline t)/rad^2_{\mathcal A}(\overline s,\overline t)$. By the definition of the Gabriel quiver $Q$, we have $\{F(\alpha_l)+rad^2_{\mathcal A}(\overline s,\overline t)\;|\;\alpha_l\in Q_1\}$ are linearly independent. Therefore, $b_l=0$ for all $l$.
\end{proof}

Note that an admissible ideal $I$ of the algebra $KQ$ is just an admissible ideal $I$ (that is, which is included in $rad^2_{KQ}$) of $KQ$ as a category.

\begin{Theorem}\label{cluster tilting} Keep the forgoing notations.
Suppose that $\mathcal T=add(\{h\cdot \bigoplus\limits_{i=1}^n X_i\;|\;h\in \Gamma\})\in \mathfrak{Add}(\mathcal C^{Q})^{\Gamma}$, with $X_i$ indecomposable, is a maximal $\Gamma$-stable rigid subcategory of $\mathcal C^Q$ without $\Gamma$-loops, $\mathcal T$ contains all projective-injective objects. If $\mathcal T$ is finitely non-Hom-orthogonal, then:

(1) $gl.dim(mod\mathcal T)=3$;

(2) $\mathcal T$ is a cluster tilting subcategory of $\mathcal C^Q$;

(3) For all simple $\mathcal T$-modules $S$ and $S'$ such that $add(\bigoplus\limits_{h\in \Gamma}h\cdot S)=add(\bigoplus\limits_{h\in \Gamma}h\cdot S')$, we have
$Ext_{mod\mathcal T}^1(S,S')=Ext_{mod\mathcal T}^2(S,S')=0$;

(4) $\mathcal T$ has no $\Gamma$-$2$-cycles;

(5) If $X_i$ is non-projective for some $i\in\{1,\cdots,n\}$, then $\mu_{X_i}(\mathcal T)$ has no $\Gamma$-loops and moreover it is a cluster tilting subcategory (certainly is maximal rigid).
\end{Theorem}

\begin{proof}
(1) It follows by Proposition \ref{globaldim'};

(2) It follows from Theorem \ref{keylemma} and (1);

(3) By the relation between $S$ and $S'$, we know they are in an orbit under action of $\Gamma$. So, from the fact $\mathcal T$ has no $\Gamma$-loops, we get $Ext^1_{mod\mathcal T}(S,S')=0$.

Concerning $Ext^2_{mod\mathcal T}(S,S')$, assume $S=S_{X_i}$ for some $X_i\in \mathcal T$, set with $\mathcal T'_{(i)}=add(\{h\cdot \bigoplus\limits_{j\neq i}^n X_j\;|\;h\in \Gamma\})$. If $X_i\in \mathcal T$ is not projective,
 $S=S_{X_i}$ has the following projective resolution, given in the proof of Proposition \ref{globaldim}:
$$0\rightarrow Hom_{\mathcal C^Q}(X_i,-)\rightarrow Hom_{\mathcal C^Q}(T'',-)\rightarrow Hom_{\mathcal C^Q}(T',-)\rightarrow Hom_{\mathcal C^Q}(X_i,-)\rightarrow S\rightarrow 0$$ for $T',T''\in \mathcal T'_{(i)}$.
Since $add(\mathcal T'_{(i)})\cap add(\{h\cdot X_i\;|\;h\in \Gamma\})=0$, we have $Hom_{mod\mathcal T}(Hom_{\mathcal C}(T'',-), S')=0$, therefore,
$Ext_{mod\mathcal T}^2(S,S')=0.$

If $X_i\in \mathcal T$ is projective, then due to the proof of Proposition \ref{globaldim},  $S=S_{X_i}$ has the  projective resolution:
$0\rightarrow Hom_{\mathcal C^Q}(Z,-)\rightarrow Hom_{\mathcal C^Q}(T',-)\rightarrow Hom_{\mathcal C^Q}(X_i,-)\rightarrow S\rightarrow 0$ for $T'\in \mathcal T'_{(i)}$ and $Z\in \mathcal T$, and there is an admissible epimorphism $T'\twoheadrightarrow Z$.

If $Z$ has a direct summand $h\cdot X_i$ for some $h\in \Gamma$, then there is an admissible epimorphism $Z\twoheadrightarrow h\cdot X_i$. Therefore, $T'\rightarrow Z\rightarrow h\cdot X_i$ is an admissible epimorphism. Since $h\cdot X_i$ is projective, then $T'\twoheadrightarrow h\cdot X_i$ is split. It contradicts to the fact $T'\in \mathcal T'_{(i)}$. So, $Z$ has no direct summand isomorphic to $h\cdot X_i$ for any $h\in \Gamma$. Thus, $add(Z)\cap add(\{h\cdot X_i\;|\;h\in \Gamma\})=0$. Then $Hom_{mod\mathcal T}(Hom_{\mathcal C}(Z,-), S')=0$. Therefore, $Ext_{mod\mathcal T}^2(S,S')=0.$

(4) Assume that $\mathcal T$ admits a $\Gamma$-$2$-cycle. Without loss of generality, assume there are irreducible morphisms $b:X_i\rightarrow X_j$, $a:X_j\rightarrow h\cdot X_i$ for some $i\neq j$ and $h\in \Gamma$. By Proposition \ref{is-hom}, we may assume that $\mathcal C^Q_h(\mathcal T)$ is $Hom$-finite. By the definition of $J$ in (\ref{J}) and the irreducibility of $a,b$, we have $a\circ b\in J(End_{\mathcal C^Q_h}(X_i))$. Then using Lemma \ref{radical}, $a\circ b$ is nilpotent.  According to Theorem \ref{nil},
we obtain $a\circ b=\sum\limits_{i=1}^m \lambda_i[u_i,v_i]$ as morphisms in $\mathcal C^Q_h(\mathcal T)$
 since $gl.dim(\mathcal C^Q_h(\mathcal T))\leq 3$ by Proposition \ref{globaldim}, where $u_i, v_i$ are compositions of irreducible morphisms in $\mathcal T$ and $[u_i,v_i]=u_i\circ v_i-v_i\circ u_i$. In general, we can write that $a\circ b=\lambda_1[a,b]+\sum\limits_{i\geq 2}^m \lambda_i[u_i,v_i]$, and $\{u_i, v_i\}\neq\{a, b\}$ for $i\geq 2$. By the composition of morphisms in $\mathcal C^Q_h(\mathcal T)$, $a\circ b=ab$ (as the composition in $\mathcal T$), $b\circ a=(h\cdot b)a$.

If $\lambda_1\neq 1$, then
\[\begin{array}{ccl} ab=a\circ b
 & = &  id_{h\cdot X_i} a\circ b id_{X_i}\\
 & = &  \frac{1}{1-\lambda_1}(-\lambda_1 id_{h\cdot X_i} b\circ a id_{X_i}+\sum\limits_{i\geq 2}id_{h\cdot X_i}[u_i,v_i] id_{X_i})\\
 & = &  \frac{1}{1-\lambda_1}\sum\limits_{i\geq 2}id_{h\cdot X_i}[u_i,v_i] id_{X_i}.
\end{array}\]
  which is an equality as morphisms in $\mathcal T$. Let $Q_{\mathcal T}$ be the Gabriel quiver of $\mathcal T$. Then by Theorem \ref{Gabr}, $\mathcal T\cong KQ_{\mathcal T}/I$ for an admissible ideal $I$. Hence, $ab-\frac{1}{1-\lambda_1}\sum\limits_{i\geq 2}id_{h\cdot X_i}[u_i,v_i] id_{X_i}\in I\setminus \{0\}$ in $KQ_{\mathcal T}$, where $\{0\}$ is deleted according to that $\{u_i, v_i\}\neq\{a, b\}$ for $i\geq 2$. Thus, we have $\bar 0\neq ab-\frac{1}{1-\lambda_1}\sum\limits_{i\geq 2}id_{h\cdot X_i}[u_i,v_i] id_{X_i}+e_{h\cdot X_i}(IJ+JI)e_{X_i}\in e_{h\cdot X_i}(I/(IJ+JI))e_{X_i}$ for $J$ the ideal of $KQ_{\mathcal T}$ generated by all arrows of $Q_{\mathcal T}$. Therefore, $e_{h\cdot X_i}(I/(IJ+JI))e_{X_i}\neq \bar 0$.

If $\lambda_1=1$,
$$(h\cdot b)a=b\circ a=id_{h\cdot X_j} a\circ b id_{X_j}=-\sum\limits_{i\geq 2}id_{h\cdot X_j}[u_i,v_i] id_{X_j}$$
which is also an equality as morphisms in $\mathcal T$. Similarly, we have $e_{h\cdot X_j}(I/(IJ+JI))e_{X_j}\neq \bar 0$.

In both cases, by Proposition \ref{extention2}, we obtain $Ext^2_{mod\mathcal T}(S_{X_i}, S_{h\cdot X_i})\neq 0$ or $Ext^2_{mod\mathcal T}(S_{X_j}, S_{h\cdot X_j})\neq 0$. But, it contradicts to (3).

(5) By Proposition \ref{quivermu}, we have $Q'=\widetilde{\mu}_{i}(Q)$, where $Q$ is the Gabriel quiver of $\mathcal T$ and $Q'$ is the Gabriel quiver of $\mu_{X_i}(\mathcal T)$. Moreover, by (4), $Q$ has no $\Gamma$-$2$-cycles, then $Q'$ has no $\Gamma$-loops, as well as $\mu_{X_i}(\mathcal T)$.

Recall that in\cite{IY}, given a functorially finite subcategory $\mathcal D$ of a triangulated category $\mathcal A$, a pair of subcategories $(\mathcal X,\mathcal Y)$ of $\mathcal A$ is called a {\bf $\mathcal D$-mutation pair} if $$\mathcal D\subseteq \mathcal X\subseteq \mu(\mathcal Y;\mathcal D):=add(\{T\;|\;T\rightarrow D \overset{b}{\rightarrow} Y, Y\in \mathcal Y, D\in \mathcal D, b\;\text{is a right}\;\mathcal D\text{-approximation}\}),$$ and $$\mathcal D\subseteq \mathcal Y \subseteq \mu^{-1}(\mathcal X;\mathcal D):=add(\{T\;|\;X\overset{a}{\rightarrow} D \rightarrow T, X\in \mathcal X, D\in \mathcal D, a\;\text{is a left}\;\mathcal D\text{-approximation}\}).$$
Since $\mathcal T$ is finitely non-Hom-orthogonal, the subcategory $\mathcal T'_{(i)}$ is finitely non-Hom-orthogonal. Hence, $\underline{\mathcal T'_{(i)}}$ is a finitely non-Hom-orthogonal subcategory of $\underline{\mathcal C^Q}$, then it is functorially finite. According to the definition of $\mu_{X_i}(\mathcal T)$, in the triangulated category $\underline{\mathcal C^Q}$, it is easy to see that $\underline{\mathcal T'_{(i)}}\subseteq \underline{\mathcal T}\subseteq \mu(\underline{\mu_{X_i}(\mathcal T)};\underline{\mathcal T'_{(i)}})$ and $\underline{\mathcal T'_{(i)}}\subseteq \underline{\mu_{X_i}(\mathcal T)}\subseteq \mu^{-1}(\underline{\mathcal T};\underline{\mathcal T'_{(i)}})$. Thus, $(\underline{\mathcal T},\underline{\mu_{X_i}(\mathcal T)})$ is a $\underline{\mathcal T'_{(i)}}$-mutation pair. By Theorem 5.1 of \cite{IY}, $\underline{\mathcal T}$ is a cluster tilting subcategory of $\underline{\mathcal C^Q}$ if and only if $\underline{\mu_{X_i}(\mathcal T)}$ is so. By Lemma II.1.3 of \cite{BIRT}, $\mathcal T$($\mu_{X_i}(\mathcal T)$) is a cluster tilting subcategory of $\mathcal C^Q$ if and only if  $\underline{\mathcal T}$ ($\underline{\mu_{X_i}(\mathcal T)}$) is a cluster tilting subcategory of $\underline{\mathcal C^Q}$. Thus, $\mu_{X_i}(\mathcal T)$ is a cluster tilting subcategory of $\mathcal C^Q$ since $\mathcal T$ is such one of $\mathcal C^Q$.
\end{proof}

Note that
the (1)-(4) of the above theorem is inspired by Theorem 3.33, \cite{D}, but they have various circumstances.

We are now already to give the proof of the main result Theorem \ref{mainlemma}: \\ \\
{\bf Proof of Theorem \ref{mainlemma}: }

For any sequence $([i_1],\cdots,[i_s],[i_{s+1}])$ of orbits, by Remark \ref{rem2.3} (3), if $\widetilde{\mu}_{[i_s]}\cdots\widetilde{\mu}_{[i_1]}(Q)$ has no $\Gamma$-$2$-cycles and no $\Gamma$-loops, then $\widetilde{\mu}_{[i_{s+1}]}\widetilde{\mu}_{[i_s]}\cdots\widetilde{\mu}_{[i_1]}(Q)$ has no $\Gamma$-loops. Hence,
following Definition \ref{unfolding} and Proposition \ref{covering}, it suffices to prove the claim that $Q$ has no $\Gamma$-$2$-cycles and no $\Gamma$-loops and $\widetilde{\mu}_{[i_s]}\cdots\widetilde{\mu}_{[i_1]}(Q)$ has no $\Gamma$-$2$-cycles for any sequence $([i_1],\cdots,[i_s])$ of orbits.

  In the case $Q$ is not a linear quiver of type $A$. From Construction \ref{construction}, it is easy to see that $Q$ has no $\Gamma$-$2$-cycles and no $\Gamma$-loops. By Lemma \ref{T_0} (4), the Gabriel quiver of $\underline{\mathcal T_0}$ is isomorphic to $Q$. Since the Gabriel quiver of $\underline{\mu_{X_{i_s}}\cdots\mu_{X_{i_1}}(\mathcal T_0)}$ is a subquiver of the Gabriel quiver of $\mu_{X_{i_s}}\cdots\mu_{X_{i_1}}(\mathcal T_0)$, then by Proposition \ref{quivermu}, it suffices to prove that $\mu_{X_{i_s}}\cdots\mu_{X_{i_1}}(\mathcal T_0)$ has no $\Gamma$-$2$-cycles, where $X_{i_j}$ are the indecomposable objects in $\mathcal C$ corresponding to vertices $i_j$ for all $1\leq j\leq s$. This follows inductively by Theorem \ref{cluster tilting} and Lemma \ref{functorially finite}.
   Also, we know that the Gabriel quiver of $\underline{\mu_{X_{i_s}}\cdots\mu_{X_{i_1}}(\mathcal T_0)}$ is isomorphic to  $\widetilde{\mu}_{[i_s]}\cdots\widetilde{\mu}_{[i_1]}(Q)$, where $i_j$ are the vertices in $Q$ correspondence to $X_{i_j}$ for $1\leq j\leq s$. Thus, the above claim holds.

 In the case $Q$ is a linear quiver of type $A$, since $Q$ is strongly almost finite and the it has no infinite path, thus it is a  finite quiver. Moreover, it is clear that $\Gamma=\{e\}$. In this case, our result follows at once.\;\;\; \;\;\;\;\;\;\;\;\;\;\;\;\;\;\;\;\;\;\;
$\square$
\\

\section{Applications to positivity and F-polynomials }

In this section, we apply the unfolding theorem (Theorem \ref{mainlemma}) to study positivity and F-polynomials of an acyclic \emph{sign-skew-symmetric} cluster algebra $\mathcal A(\Sigma)$ in a uniform method. That is, in order to prove a property for $\mathcal A(\Sigma)$, one can firstly try to check this property for the skew-symmetric cluster algebra $\mathcal A(\widetilde{\Sigma})$ of \emph{infinite rank,} where $\widetilde{\Sigma}$ is the unfolding of $\Sigma$; secondly, try to give a surjective algebra morphism $\pi: \; \mathcal A(\widetilde{\Sigma})\rightarrow \mathcal A(\Sigma)$ which preserves such property.

\subsection{The surjective algebra morphism $\pi$}.

Let  $B\in Mat_{n\times n}(\mathbb Z)$ be sign-skew-symmetric, $B'\in Mat_{m\times n}(\mathbb Z)$ and $\widetilde B=\left(\begin{array}{c}
B  \\
B'
\end{array}\right)\in Mat_{(n+m)\times n}(\mathbb Z)$.  Assume that the exchange matrix $B\in Mat_{n\times n}(\mathbb Z)$ is acyclic. Denote $B''=\left(\begin{array}{cc}
B  & -B'^{T}\\
B' & 0
\end{array}\right)$. Since $B$ is acyclic, $B''$ is acyclic. Let $(Q'', \Gamma)$ be an unfolding of $B''$. Let $\widetilde Q$ be the quiver obtained from $Q''$ by freezing all vertices of $Q''$ corresponding to $B'$. Denote by $Q$ the subquiver of $\widetilde Q$ generated by the exchangeable vertices. Clearly, $(Q,\Gamma)$ is just the unfolding of $B$. In fact, it is easy to verify that $Q$ is a covering of $B$ and  $\widetilde\mu_{[i_s]}\cdots\widetilde\mu_{[i_1]}(Q)$ is a subquiver of $\widetilde\mu_{[i_s]}\cdots\widetilde\mu_{[i_1]}(\widetilde Q)$ for any sequences $([i_1],\cdots,[i_s])$ of orbits.

Let $\widetilde\Sigma=\Sigma(\widetilde Q)=(\widetilde X, \widetilde Y, \widetilde Q)$ be the seed associated with $\widetilde Q$, where $\widetilde X=\{x_{u}\;|\;u\in Q_0\}$, $\widetilde Y=\{y_v\;|\;v\in \widetilde Q_0\setminus Q_0\}$. Let $\Sigma=\Sigma(\widetilde B)=(X, Y,\widetilde B)$ be the seed associated with $\widetilde B$, where $X=\{x_{[i]}\;|\;i=1,\cdots,n\}$, $Y=\{y_{[j]}\;|\;j=1,\cdots,m\}$. It is clear that there is a surjective algebra homomorphism:
\begin{equation}\label{pi}
\pi:\;\; \mathbb Q[x^{\pm 1}_i, y_j\;|\;i\in Q_0, j\in \widetilde Q_0\setminus Q_0]\rightarrow \mathbb Q[x^{\pm 1}_{[i]},y_{[j]}\;|\;1\leq i\leq n, 1\leq j\leq m]\;\;\;\text{with}\;\;\; x_{i}\mapsto x_{[i]}, y_j\mapsto y_{[j]}
\end{equation}
where $[i]$, $[j]$  is the orbits of $i$, $j$ respectively under the action of $\Gamma$.

For any cluster variable $x_u\in \widetilde X$, define $\widetilde\mu_{[i]}(x_u)=\mu_u(x_u)$ if $u\in [i]$; otherwise,  $\widetilde\mu_{[i]}(x_u)=x_u$ if $u\not\in [i]$. Formally, we define $\widetilde\mu_{[i]}(\widetilde X)=\{\widetilde\mu_{[i]}(x)\;|\;x\in \widetilde X\}$ and $\widetilde\mu_{[i]}({\widetilde X}^{\pm 1})=\{\widetilde\mu_{[i]}(x)^{\pm 1}\;|\;x\in \widetilde X\}$.

\begin{Lemma}\label{mutationorbit}
Keep the notations as above. Assume that $B$ is acyclic. If $[i]$ is an orbit of vertices with $i\in Q_0$, then we have

(1)~$\widetilde\mu_{[i]}(x_j)$ is a cluster variable of $\mathcal A(\widetilde Q)$ for any $j\in Q_0$,

(2)~$\widetilde\mu_{[i]}(\widetilde X)$ is algebraic independent over $\mathbb Q[y_j\;|\;j\in \widetilde Q_0\setminus Q_0]$.
\end{Lemma}

\begin{proof}
(1)~ It follows immediately by the definition.

(2)~ By the definition, we need to prove any finite variables in $\widetilde\mu_{[i]}(\widetilde X)$ is algebraic independent. It suffices to show that any finite variables in $\widetilde\mu_{[i]}(\widetilde X)$ are in a cluster of $\mathcal A(\widetilde\Sigma)$. Assume that $\{z_1,\cdots,z_s\}\subseteq \widetilde\mu_{[i]}(\widetilde X)$, by definition of $\widetilde\mu_{[i]}(\widetilde X)$, there exist $s$ vertices $i_t, t=1,\cdots,s$ in $[i]$ such that $z_t=\mu_{i_t}(x_t)$ for $t=1,\cdots, s$. Denote $I=\bigcup\limits_{t=1,\cdots,s}\{i_t\}$, since $\widetilde Q$ has no $\Gamma$-loops, we have $z_t=\mu_{i_t}(x_t)=\prod\limits_{i'\in I}\mu_{i'}(x_t)$, then $\{z_1,\cdots,z_s\}\subseteq \prod\limits_{i'\in I}\mu_{i'}(\widetilde X)$. Our result follows.
\end{proof}

By Lemma \ref{mutationorbit}, $\widetilde\mu_{[i]}(\widetilde \Sigma):=(\widetilde\mu_{[i]}(\widetilde X), \widetilde Y, \widetilde\mu_{[i]}(\widetilde Q))$ is a seed. Thus, we can define $\widetilde\mu_{[i_s]}\widetilde\mu_{[i_{s-1}]}\cdots\widetilde\mu_{[i_1]}(x)$ and $\widetilde\mu_{[i_s]}\widetilde\mu_{[i_{s-1}]}\cdots\widetilde\mu_{[i_1]}(\widetilde X)$ and $\widetilde\mu_{[i_s]}\widetilde\mu_{[i_{s-1}]}\cdots\widetilde\mu_{[i_1]}(\widetilde \Sigma)$ for any sequence $([i_1],[i_2],\cdots,[i_s])$ of orbits in $Q_0$.

\begin{Lemma}\label{co}
Keep the notations as above. Assume that $B$ is acyclic. If $a$ is a vertex of $Q$, then any finite cluster variables in a cluster of $\mathcal A(\widetilde \mu_{[a]}\widetilde \Sigma)$ is contained in a cluster of $\mathcal A(\widetilde \Sigma)$.
\end{Lemma}

\begin{proof}
Let $\Delta_1$ and $\Delta_2$ be two finite subsets of $Q_0$. For any finite cluster variables $z_j=\mu_{i_s}\cdots\mu_{i_1}(\widetilde \mu_{[a]}x_j)$, $j\in\Delta_1$ of $\mathcal A(\widetilde \mu_{[a]}\widetilde \Sigma)$ and finite vertices $\Delta_2$ of $\mu_{i_s}\cdots\mu_{i_1}(\widetilde \mu_{[a]}\widetilde Q)$, we prove inductively on $s$ that there exists a finite set $S\subseteq [a]$ such that (1) $z_j=\mu_{i_s}\cdots\mu_{i_1}(\prod\limits_{a'\in [S']} \mu_{a'}x_j)$ for any $j\in \Delta_1$ and finite set $S\subseteq S'\subseteq [a]$; (2) for any $a_j\in \Delta_2$, the subquiver of $\mu_{i_s}\cdots\mu_{i_1}(\widetilde \mu_{[a]}\widetilde Q)$ formed by the arrows incident with $a_j$ equals the subquiver of $\mu_{i_s}\cdots\mu_{i_1}(\prod\limits_{a'\in [S']} \mu_{a'}\widetilde Q)$ formed by the arrows incident with $a_j$.

For $s=0$, set $S=S_1\cup S_2$, where $S_1=\{j\in [a]\;|\; j\in \Delta_1\}$ and $S_2=\{c\in [a]\;|\; j\;\text{incidents with or in}\;\Delta_2\}$. For any finite set $S'$ such that $S\subseteq S'\subseteq [a]$, by the definition of $z_j=\widetilde\mu_{[a]}x_j$, it is easy to see that $z_j=\prod\limits_{a'\in [S']} \mu_{a'}x_j$, (1) holds; for any $a_j\in \Delta_2$, denote $\widetilde \mu_{[a]}(\widetilde Q)=(b'_{ij})$ and $\prod\limits_{a'\in [S']} \mu_{a'}(\widetilde Q)=(b''_{ij})$. If $a_j\in [a]$, since $a_j\in S'$, we have $b''_{a_jk}=-b_{a_jk}=b'_{a_jk}$. If $a_j\not\in [a]$, since all $c\in [a]$ incident with $a_j$ are in $S'$, we have $b''_{a_jk}=b_{a_jk}+\sum\limits_{a'\in[a]}\frac{|b_{a_ja'}|b_{a'k}+b_{a_ja'}|b_{a'k}|}{2}=b'_{a_jk}$. Thus, (2) holds.

Assume that the statements hold for $s-1$. For case $s$, set $\Delta'_1=\{j\in Q_0\;|\; j\;\text{incidents with or in}\;\Delta_1\}$ and $\Delta'_2=\Delta_1\cup\Delta_2$. Applying the assumption to $\Delta'_1$ and $\Delta'_2$, there exists a finite subset $S'\subseteq [a]$ such that (a) $\mu_{i_{s-1}}\cdots\mu_{i_1}(\widetilde \mu_{[a]}x_j)=\mu_{i_{s-1}}\cdots\mu_{i_1}(\prod\limits_{a'\in [S'']} \mu_{a'}x_j)$ for any $j$ incidents with or in $\Delta_1$ and finite set $S'\subseteq S''\subseteq [a]$ and (b) for any $a'_j\in \Delta'_2$, the subquiver of $\mu_{i_{s-1}}\cdots\mu_{i_1}(\widetilde \mu_{[a]}\widetilde Q)$ formed by the arrows incident with $a'_j$ equals the subquiver of $\mu_{i_{s-1}}\cdots\mu_{i_1}(\prod\limits_{a'\in [S']} \mu_{a'}\widetilde Q)$ formed by the arrows incident with $a'_j$. For any $j\in \Delta_1$, if $j\neq i_s$, since $\Delta_1\subseteq \Delta'_1$, by (a), we have $$\mu_{i_s}\cdots\mu_{i_1}(\widetilde \mu_{[a]}x_j)=\mu_{i_{s-1}}\cdots\mu_{i_1}(\widetilde \mu_{[a]}x_j)=\mu_{i_{s-1}}\cdots\mu_{i_1}(\prod\limits_{a'\in [S'']} \mu_{a'}x_j)=\mu_{i_{s}}\cdots\mu_{i_1}(\prod\limits_{a'\in [S'']} \mu_{a'}x_j);$$ if $j=i_s$, since $j\in \Delta'_1,\Delta'_2$, by (a), (b) and the definition of mutation of cluster variable, we have $\mu_{i_s}\cdots\mu_{i_1}(\widetilde \mu_{[a]}x_j)=\mu_{i_{s}}\cdots\mu_{i_1}(\prod\limits_{a'\in [S'']} \mu_{a'}x_j).$ (1) holds. For any $a_j\in \Delta_2$, since $\Delta_2\subseteq \Delta'_2$, by (b), we have the subquiver of $\mu_{i_s}\cdots\mu_{i_1}(\widetilde \mu_{[a]}\widetilde Q)$ formed by the arrows incident with $a_j$ equals the subquiver of $\mu_{i_s}\cdots\mu_{i_1}(\prod\limits_{a'\in [S']} \mu_{a'}\widetilde Q)$ formed by the arrows incident with $a_j$. (2) holds.
Therefore, the statements hold for all $s\in \mathbb N$.

In particular, for any finite cluster variables $\{z_1,\cdots,z_m\}$ in a cluster of $\mathcal A(\widetilde \mu_{[a]}\widetilde \Sigma)$, let $\Delta_1$ be the vertices correspond to the cluster variables and $\Delta_2=\emptyset$. Applying the statements to $\Delta_1$ and $\Delta_2$, our result follows.
\end{proof}

\begin{Lemma}\label{laurentpre}
Keep the above notations. Assume that $B$ is acyclic. Then for any sequence of orbits $([i_1],\cdots,[i_s])$ in $Q_0$, we have

(1)~ The cluster variables of $\mathcal A(\widetilde\mu_{[i_s]}\cdots\widetilde\mu_{[i_1]}(\widetilde \Sigma))$ are the same as the cluster variables of $\mathcal A(\widetilde \Sigma)$;

(2)~ Any finite variables in $\widetilde\mu_{[i_s]}\cdots\widetilde\mu_{[i_1]}(\widetilde X)$ is contained in a cluster of $\mathcal A(\widetilde \Sigma)$;

(3)~ Any variable in $\widetilde\mu_{[i_s]}\cdots\widetilde\mu_{[i_1]}(\widetilde X)$ is a cluster variable of $\mathcal A(\widetilde \Sigma)$;

(4)~ Any monomial with variables in $\widetilde\mu_{[i_s]}\cdots\widetilde\mu_{[i_1]}(\widetilde X)$ is a cluster monomial of $\mathcal A(\widetilde \Sigma)$.
\end{Lemma}

\begin{proof}
(1) Using Lemma \ref{co}, the cluster variables of $\mathcal A(\widetilde\mu_{[i_1]}(\widetilde \Sigma))$ are the cluster variables of $\mathcal A(\widetilde \Sigma)$. Since it is easy to see that $\widetilde\mu_{[i_1]}\widetilde\mu_{[i_1]}(\widetilde \Sigma)=\widetilde\Sigma$, dually, we have the cluster variables of $\mathcal A(\widetilde \Sigma)$ are the cluster variables of $\mathcal A(\widetilde\mu_{[i_1]}(\widetilde \Sigma))$. Therefore, The cluster variables of $\mathcal A(\widetilde\mu_{[i_1]}(\widetilde \Sigma))$ are the same as the cluster variables of $\mathcal A(\widetilde \Sigma)$. Our result follows.

(2) Since any finite variables of $\widetilde\mu_{[i_s]}\cdots\widetilde\mu_{[i_1]}(\widetilde X)$ are in the initial cluster of $\mathcal A(\widetilde\mu_{[i_s]}\cdots\widetilde\mu_{[i_1]}(\widetilde \Sigma))$, applying Lemma \ref{co} step by step, our result follows.

(3) It follows immediately by (2).

(4) It follows immediately by (1).
\end{proof}

\begin{Corollary}\label{Lau}
Keep the notations as above. Assume that $B$ is acyclic. Then for any sequence of orbits $([i_1],\cdots,[i_s])$ in $Q_0$, any cluster variable of $\mathcal A(\widetilde\Sigma)$ can be expressed as a Laurent polynomial of $\widetilde\mu_{[i_s]}\cdots\widetilde\mu_{[i_1]}(\widetilde X)$ with ground field $\mathbb Z[y_j\;|\;j\in \widetilde Q_0\setminus Q_0]$.
\end{Corollary}

\begin{proof}
By Lemma \ref{laurentpre} (1), any cluster variable of $\mathcal A(\widetilde\Sigma)$ is a cluster variable of $\mathcal A(\widetilde\mu_{[i_s]}\cdots\widetilde\mu_{[i_1]}(\widetilde \Sigma))$. By Laurent Phenomenon in \cite{fz1}, our result follows.
\end{proof}

\begin{Theorem}\label{sur}
Keep the notations as above with an acyclic sign-skew-symmetric matrix $B$ and $\pi$ defined in (\ref{pi}).
 Restricting $\pi$ to $\mathcal A(\widetilde\Sigma)$, then $\pi:\mathcal A(\widetilde\Sigma)\rightarrow \mathcal A(\Sigma)$ is a surjective algebra morphism satisfying that $\pi(\widetilde\mu_{[j_k]}\cdots\widetilde\mu_{[j_1]}(x_{a} ))=\mu_{[j_k]}\cdots\mu_{[j_1]}(x_{[i]})\in \mathcal A(\Sigma)$ and $\pi(\widetilde\mu_{[j_k]}\cdots\widetilde\mu_{[j_1]}(\widetilde X))=\mu_{[j_k]}\cdots\mu_{[j_1]}(X)$ for any sequences of orbits $[j_1],\cdots,[j_k]$ and any $a\in [i]$.
\end{Theorem}

\begin{proof}
For $k=1$, by the definition of $\pi$ and $b_{[j][i]}=\sum\limits_{j'\in[j]}b_{j'i}$, we have $\pi(\widetilde\mu_{[j_1]}(x_i))=\mu_{[j_1]}(x_{[i]})$. Assume that the equation holds for every $k<s$. Since $\mathcal A(\widetilde\mu_{[j_1]}(\widetilde\Sigma))\subseteq \mathbb Q[x^{\pm 1}_i,y_j\;|\;i\in Q_0,j\in \widetilde Q_0\setminus Q_0]$ and $\mathcal A(\mu_{[j_1]}(\Sigma))\cong \mathcal A(\Sigma)$, applying $\pi$ to $\mathcal A(\widetilde\mu_{[j_1]}(\widetilde\Sigma))$ and by induction, we have $$\pi(\widetilde\mu_{[j_k]}\cdots\widetilde\mu_{[j_1]}( x_i))=\pi(\widetilde\mu_{[j_k]}\cdots\widetilde\mu_{[j_2]}(\widetilde\mu_{[j_1]} x_i))=\mu_{[j_k]}\cdots\mu_{[j_2]}(\pi(\widetilde\mu_{[j_1]} x_{[i]}))=\mu_{[j_k]}\cdots(\mu_{[j_1]}(x_{[i]}))=x.$$
Thus, $\pi(\widetilde\mu_{[j_k]}\cdots\widetilde\mu_{[j_1]}(\widetilde X))=\mu_{[j_k]}\cdots\mu_{[j_1]}(X)$.

Now, to prove that $\pi$ is a surjective algebra homomorphism, it suffices to show that $\pi(\mathcal A(\widetilde \Sigma))\subseteq \mathcal A(\Sigma)$. For any cluster $X'=\mu_{[j_k]}\cdots\mu_{[j_1]}(X)\in \mathcal A(\Sigma)$, we have $\pi(\widetilde\mu_{[j_k]}\cdots\widetilde\mu_{[j_1]}(\widetilde X))=X'$. Thus,

$\pi(\bigcap\limits_{\widetilde\mu_{[j_k]}\cdots\widetilde\mu_{[j_1]}(\widetilde X)}\mathbb Z[y_j\;|\;j\in \widetilde Q_0\setminus Q_0][\widetilde\mu_{[j_k]}\cdots\widetilde\mu_{[j_1]}(\widetilde X)^{\pm 1}])$\\
$=\bigcap\limits_{\mu_{[j_k]}\cdots\mu_{[j_1]}(X)}\mathbb Z[y_{[j]}\;|\;1\leq j\leq m][\mu_{[j_k]}\cdots\mu_{[j_1]}(X)^{\pm 1}]=\mathcal U(\Sigma).$\\
 Therefore, due to Lemma \ref{laurentpre} (4) and Corollary \ref{A=U}, we have $$\pi(\mathcal A(\widetilde \Sigma))\subseteq\pi(\mathcal U(\widetilde \Sigma))\subseteq \pi(\bigcap\limits_{\widetilde\mu_{[j_k]}\cdots\widetilde\mu_{[j_1]}(\widetilde X)}\mathbb Z[y_j\;|\;j\in\widetilde Q_0\setminus Q_0][\widetilde\mu_{[j_k]}\cdots\widetilde\mu_{[j_1]}(\widetilde X)^{\pm 1}])=\mathcal U(\Sigma)=\mathcal A(\Sigma).$$ Our result follows.
\end{proof}

\subsection{Positivity conjecture}

\begin{Conjecture}(\cite{fz1}, Positivity conjecture)
Let $\mathcal A(\Sigma)$ be a cluster algebra over $\mathbb {ZP}$. For any cluster $X$, and any
cluster variable $x$, the Laurent polynomial expansion of $x$ in the cluster $X$ has coefficients
which are nonnegative integer linear combinations of elements in $\mathbb P$.
\end{Conjecture}

By Fomin-Zelevinsky¡¯s separation of addition formula (Theorem 3.7, \cite{fz4}), to verify positivity conjecture, it suffices to prove that it holds for cluster algebras of geometric type. By Theorem 4.4 of \cite{HLY}, see also Theorem 2.23 of \cite{CZ}, any cluster algebra of geometric type $\mathcal A(\Sigma_1)$ is a rooted cluster subalgebra of a cluster algebra $\mathcal A(\Sigma)$ with coefficients $\mathbb Z$ via freezing some exchangeable variables. Thus, any cluster variable of $\mathcal A(\Sigma_1)$ is a cluster variable of $\mathcal A(\Sigma)$ and a cluster of $\mathcal A(\Sigma_1)$ corresponds to a cluster of $\mathcal A(\Sigma)$. Then the positivity of $\mathcal A(\Sigma)$ implies the positivity of $\mathcal A(\Sigma_1)$. Thus,{\bf we only need to prove the  positivity of sign-skew-symmetric cluster algebras $\mathcal A(\Sigma)$ in the case $\mathcal A(\Sigma)$ is with coefficients $\mathbb Z$, that is, the exchange matrix of $\mathcal A(\Sigma)$ is a square matrix,} as which discussed in the foregoing.

\begin{Theorem}\label{po}(Theorem 7, \cite{G})
The positivity conjecture is true for skew-symmetric cluster algebras of
infinite rank.
\end{Theorem}

\begin{Theorem}\label{positivity}
The positivity conjecture is true for acyclic sign-skew-symmetric cluster algebras.
\end{Theorem}

\begin{proof}
 By the remark at the beginning of this subsection, it suffices to prove that positivity holds for cluster algebras with coefficients $\mathbb Z$. Let $\mathcal A(\Sigma)$ be an acyclic sign-skew-symmetric cluster algebra with coefficients $\mathbb Z$. For any cluster variable $z=\mu_{[j_k]}\cdots\mu_{[j_1]}(x_{[i]})\in \mathcal A(\Sigma)$ and cluster $\mu_{[i_t]}\cdots\mu_{[i_1]}(X)$ of $\mathcal A(\Sigma)$, applying Theorem \ref{sur} to the coefficients $\mathbb Z$ case, we have $\pi(\widetilde z)=z$ and $\pi(\widetilde{\mu}_{[i_t]}\cdots \widetilde{\mu}_{[i_1]}(\widetilde{X}))=\mu_{[i_t]}\cdots\mu_{[i_1]}(X)$, where $\widetilde z=\widetilde {\mu}_{[j_k]}\cdots\widetilde{\mu}_{[j_1]}(x_{i})$. By Corollary \ref{Lau} and Theorem \ref{po}, we have $\widetilde z\in \mathbb N[\widetilde{\mu}_{[i_t]}\cdots \widetilde{\mu}_{[i_1]}(\widetilde{X})^{\pm 1}]$. Thus, we obtain $z\in \mathbb N[\mu_{[i_t]}\cdots\mu_{[i_1]}(X)^{\pm 1}]$  by the definition of $\pi$ in (\ref{pi}).
\end{proof}

\subsection{$F$-polynomials}

Let $\Sigma=(X,Y,B)$ be a seed of rank $n$ with principal coefficients and the coefficients variables are $y_{1},\cdots,y_{n}$, that is, its  extended exchange matrix $\widetilde B=\left(\begin{array}{c}
B  \\
I_n
\end{array}\right)\in Mat_{2n\times n}(\mathbb Z)$, where $I_n$ is the $n\times n$ identity matrix. For each cluster variable $x\in \mathcal A(\Sigma)$, it can be expressed as a Laurent polynomial $x=x(x_{1},\cdots,x_{n},y_{1},\cdots,y_{n})\in \mathbb Z[x_{1}^{\pm 1},\cdots,x_{n}^{\pm 1},y_{1},\cdots,y_{n}]$ by Laurent phenomenon, whose $F$ polynomial is defined as $F(x)=x(1,\cdots,1,y_{1},\cdots,y_{n})\in \mathbb Z[y_{1},\cdots,y_{n}]$. For details, see \cite{fz4}.
It is conjectured in \cite{fz4} that

\begin{Conjecture}\label{$F$-pol}(\cite{fz4}, Conjecture 5.4) In a cluster algebra with principal coefficients, each $F$-polynomial has constant term $1$.
\end{Conjecture}

In view of Proposition 5.3 in \cite{fz4}, this conjecture means that each $F$-polynomial $F(x)$ has a unique monomial of maximal degree. Moreover, this monomial has coefficient 1, and it is divisible by all of the other occurring monomials. See Conjecture 5.5 of \cite{fz4}.

\begin{Remark}
The seeds defined in \cite{fz4} are skew-symmetrizable with finite rank, it is easy to see that the above definition and conjecture can also apply to cluster algebras with sign-skew-symmetric seeds with finite/infinite rank.
\end{Remark}

\begin{Theorem}\label{con}(\cite{DWZ}, Theorem 1.7)
Conjectures \ref{$F$-pol} holds under the condition that $B$ is skew-symmetric.
\end{Theorem}

\begin{Lemma}\label{inf}
Conjecture \ref{$F$-pol} holds for all  skew-symmetric cluster algebras of infinite rank with principal coefficients.
\end{Lemma}

\begin{proof}
Assume that $Q$ is an infinite locally finite quiver and $\widetilde\Sigma=(X,Y,Q)$ is the seed with principal coefficients associated with $Q$, where $X=\{x_i\;|\;i\in Q_0\}$ and $Y=\{y_i\;|\;i\in Q_0\}$. For any cluster variable $x=\mu_{i_s}\cdots\mu_{i_1}(x_{i_0})$ of $\mathcal A(\widetilde\Sigma)$, denote $Q'$ the subquiver of $Q$ generated by the vertices $\{i_0,i_1,\cdots,i_s\}$. Assume that $S$ be the set of vertices in $Q_0$ which are incident with a vertex of $Q'$ but not belong to $Q'_0$. Denote $Q''$ be the subquiver generated by $S\cup Q'_0$. Let $\Sigma'$ be the seed obtained from $\Sigma$ by frozen $x_i,i\in S$ and deleted $x_i,i\not\in Q''_0$ and $y_j,j\not\in Q'_0$. Using the analogue result of Theorem 4.4 of \cite{HLY}, we have $\mathcal A(\Sigma')$ is the rooted cluster subalgebra of $\mathcal A(\widetilde \Sigma)$. Since $\{i_0,i_1,\cdots,i_s\}= Q'_0$, we have $x\in \mathcal A(\Sigma')$. Let $\Sigma''=(X'',Y'',Q'')$ be the seed with principal coefficients associated with $Q''$, by Theorem 4.4 of \cite{HLY}, $\mathcal A(\Sigma')$ is a rooted cluster subalgebra of $\mathcal A(\Sigma'')$, then $x\in \mathcal A(\Sigma'')$. By Theorem \ref{con}, the $F$-polynomial has constant term $1$. The result follows.
\end{proof}

From now on, assume that $B\in Mat_{n\times n}(\mathbb Z)$ is an acyclic matrix and $(Q,\Gamma)$ is an unfolding of $B$. Let $\widetilde{\Sigma}= \Sigma(Q)=(\widetilde X,\widetilde Y, Q)$ and $\Sigma=\Sigma(B)=(X,Y,B)$ be the seeds with principal coefficients and extend matrix be the matrix associated with $Q$ and $B$ respectively, where $\widetilde X=\{x_i\;|\;i\in Q_0\}, \widetilde Y=\{y_i\;|\;i\in Q_0\}$ and $X=\{x_{[i]}\;|\;i=1,\cdots,n\}, Y=\{y_{[i]}\;|\;i=1,\cdots,n\}$.

Combing Theorem \ref{con} and Theorem \ref{sur} in the principal coefficients case, we can deduce that Conjectures \ref{$F$-pol}  hold for all acyclic cluster algebras.

\begin{Theorem}\label{poly}
In an acyclic sign-skew-symmetric cluster algebra with principal coefficients, each $F$-polynomial has constant term $1$.
\end{Theorem}

\begin{proof}
By Theorem \ref{sur} (1) and Proposition \ref{laurentpre} (3), any cluster variable $x\in \mathcal A(\Sigma)$ can lift to a cluster variable $\widetilde x\in \mathcal A(\widetilde \Sigma)$, that is $\pi(\widetilde x)=x$. Thus, we have $\pi(F(\widetilde{x}))=F(x)$, where $F(\widetilde x)$ ($F(x)$ respectively) means the $F$-polynomial associated with $\widetilde x$ ($x$ respectively). By Lemma \ref{inf}, our result follows.
\end{proof}

\appendix
\section{$\mathcal A=\mathcal U$ for acyclic sign-skew-symmetric cluster algebras}

Let $\Sigma=(X,Y,\widetilde B)$ be a sign-skew-symmetric seed, where $X=\{x_1,\cdots,x_n\}$, $Y=\{y_1,\cdots,y_m\}$.

In \cite{M}, the author considered the case that $\widetilde B$ is skew-symmetrizable and the ground ring is $\mathbb Z[y_1^{\pm 1},\cdots, y_m^{\pm 1}]$. It can be seen that the approach to the result on $\mathcal A=\mathcal U$ for acyclic $\widetilde B$ in \cite{M} can be applied to the totally sign-skew-symmetric case. So, we will cite the result in \cite M as that in the acyclic sign-skew-symmetric case.

In this appendix, we consider the cluster algebra $\mathcal A(\Sigma)$ in the case that $\widetilde B$ is sign-skew-symmetric and the ground ring is $\mathbb Z[y_1,\cdots,y_m]$. We will see here that according to the difference of the ground ring, it is only needed to do little modification in the definitions and proofs in \cite{M}. Let $\widetilde {\mathcal A}(\Sigma)=\mathcal A(\Sigma)[y_1^{-1},\cdots,y_m^{-1}]$, then  $\widetilde{\mathcal A}(\Sigma)$ is a cluster algebra with the ground ring $\mathbb Z[y_1^{\pm 1},\cdots, y_m^{\pm 1}]$.

For the seed $\Sigma=(X,Y,\widetilde B)$ and a subset $S\subseteq X$, the seed $\Sigma_{S,\emptyset}$ means the sub-seed obtained by freezing the variables in $S$. For details, see \cite{HLY}.

\begin{Lemma}\label{Lemma1}(\cite{M}, Lemma 1)
Let $\Sigma=(X,Y,\widetilde B)$ be a seed, $S\subseteq X$. Let $\mathcal U(\Sigma)$ be the upper cluster algebra of $\mathcal A(\Sigma)$ and $\mathcal U(\Sigma_{S,\emptyset})$ be the upper cluster algebra of $\mathcal A(\Sigma_{S,\emptyset})$. Then
$$\mathcal A(\Sigma_{S,\emptyset})[x^{-1}\;|\;x\in S]\subseteq \mathcal A(\Sigma)[x^{-1}\;|\;x\in S]\subseteq \mathcal U(\Sigma)[x^{-1}\;|\;x\in S]\subseteq \mathcal U(\Sigma_{S,\emptyset})[x^{-1}\;|\;x\in S].$$
\end{Lemma}

\begin{Definition}(\cite{M}, Definition 1)
Let $\Sigma=(X,Y,\widetilde B)$ is a seed, $S\subseteq X$. We say $\mathcal A(\Sigma_{S,\emptyset})$ to be a {\bf cluster localization} of $\mathcal A(\Sigma)$ if $\mathcal A(\Sigma)[x^{-1}\;|\;x\in S]=\mathcal A(\Sigma_{S,\emptyset})[x^{-1}\;|\;x\in S]$.
\end{Definition}

\begin{Definition}
For a cluster algebra $\mathcal A(\Sigma)$, a set $\{\mathcal A(\Sigma_{S_i,\emptyset})\}$ of cluster localizations of $\mathcal A(\Sigma)$ is called a {\bf cover} of $\mathcal A(\Sigma)$ if, for every prime ideal $P$ in $\mathcal A(\Sigma)$, there exist some $\mathcal A(\Sigma_{S_i,\emptyset})$ ($i\in I$) such that $\mathcal A(\Sigma_{S_i,\emptyset})[x^{- 1}\;|\;x\in S_i]P\subsetneq\mathcal A(\Sigma_{S_i,\emptyset})[x^{-1}\;|\;x\in S_i]$.
\end{Definition}

\begin{Proposition}\label{proposition 2}(\cite{M}, Proposition 2)
If  $\{\mathcal A(\Sigma_{S_i,\emptyset})\}_{i\in I}$  is a cover of $\mathcal A(\Sigma)$, then $$\mathcal A(\Sigma)=\bigcap_{i\in I}\mathcal A(\Sigma_{S_i,\emptyset})[x^{-1}\;|\;x\in S_i].$$
\end{Proposition}

\begin{Lemma}\label{lemma 2}(\cite{M}, Lemma 2)
Let $\{\mathcal A(\Sigma_{S_i,\emptyset})\}_{i\in I}$ be a cover of $\mathcal A(\Sigma)$. If $\mathcal A(\Sigma_{S_i,\emptyset})=\mathcal U(\Sigma_{S_i,\emptyset})$ for each $i\in I$, then $\mathcal A(\Sigma)=\mathcal U(\Sigma)$.
\end{Lemma}

\begin{proof}
Since $\mathcal A(\Sigma_{S_i,\emptyset})=\mathcal U(\Sigma_{S_i,\emptyset})$, we have $\mathcal A(\Sigma_{S_i,\emptyset})[x^{-1}\;|\;x\in S_i]=\mathcal U(\Sigma_{S_i,\emptyset})[x^{-1}\;|\;x\in S_i]$. By Lemma \ref{Lemma1}, $\mathcal U(\Sigma)\subseteq\mathcal U(\Sigma_{S_i,\emptyset})[x^{-1}\;|\;x\in S_i]$ for all $i$. Then $\mathcal U(\Sigma)\subseteq \bigcap_{i\in I}\mathcal U(\Sigma_{S_i,\emptyset})[x^{-1}\;|\;x\in S_i]$. Thus, by Proposition \ref{proposition 2}, $$\mathcal U(\Sigma)\subseteq \bigcap_{i\in I}\mathcal U(\Sigma_{S_i,\emptyset})[x^{-1}\;|\;x\in S_i]=\bigcap_{i\in I}\mathcal A(\Sigma_{S_i,\emptyset})[x^{-1}\;|\;x\in S_i]=\mathcal A(\Sigma)\subseteq \mathcal U(\Sigma).$$
\end{proof}

In \cite{M}, $\mathcal A(\Sigma)$ is called {\bf isolated} if the exchange matrix of $\Sigma$ is zero.

\begin{Proposition}\label{proposition 3}(\cite{M}, Proposition 3)
Let $\mathcal A(\Sigma)$ be an isolated cluster algebra. Then $\mathcal A(\Sigma)=\mathcal U(\Sigma)$.
\end{Proposition}

\begin{Proposition}\label{proposition 4} (\cite{M}, Proposition 4)
If $\mathcal A(\Sigma)$ is acyclic, then it admits a cover by isolated cluster algebras.
\end{Proposition}

\begin{Corollary}\label{co2}(\cite{M}, Corollary 2)\label{A=U}
If $\mathcal A(\Sigma)$ is acyclic, then $\mathcal A(\Sigma)=\mathcal U(\Sigma)$.
\end{Corollary}

\begin{proof}
It follows immediately by Proposition \ref{proposition 3}, Proposition \ref{proposition 4} and Lemma \ref{lemma 2}.
\end{proof}

\begin{Theorem}
An acyclic sign-skew-symmetric cluster algebra $\mathcal A(\Sigma)$ of
infinite rank equals to its upper cluster algebra $\mathcal U(\Sigma)$.
\end{Theorem}

\begin{proof}
It is clear that $\mathcal A(\Sigma)\subseteq \mathcal U(\Sigma)$ by the definition of upper cluster algebra;

On the other side of the inclusion, for any $f\in \mathcal U(\Sigma)\subseteq \mathbb Q[X^{\pm 1}]$, there exists a finite subset $\{x_1,\cdots,x_n\}\subseteq X$ such that $f=\frac{g(x_1,\cdots,x_n)}{x_1^{a_1}\cdots x_n^{a_n}}\in \mathbb Q[x_1^{\pm 1}, \cdots, x_n^{\pm 1}]$ for some non-negative integer $a_i$ and a polynomial $g$. Let $\mathcal A (\Sigma')$ be the rooted cluster subalgebra of finite rank of $\mathcal A(\Sigma)$ which contains $x_1,\cdots,x_n$ as initial cluster variables.

Now we show that $f\in \mathcal U(\Sigma')$. For any cluster $\{x'_1,\cdots,x'_m\}$  of $\mathcal A(\Sigma')$, there exists a cluster $X'$ of $\mathcal A(\Sigma)$ such that $\{x'_1,\cdots,x'_m\}\subseteq X'$. Since $f\in \mathcal U(\Sigma)\subseteq \mathbb Q[X'^{\pm 1}]$, $f$ can be written in the form $f=\frac{g'(x'_1,\cdots,x'_m,\mathbf{x})}{{x'_1}^{b_1}\cdots{x'_m}^{b_m}\mathbf{x}^{\mathbf b}}$ for some non-negative integer $b_i$, a non-negative integer vector $\mathbf b$, $\mathbf{x}=X'\backslash \{x'_1,\cdots,x'_m\}$ and a polynomial $g'$, where we may assume that $g'$ and ${x'_1}^{b_1}\cdots{x'_m}^{b_m}\mathbf{x}^{\mathbf b}$ are coprime. Since $x_1,\cdots,x_n\in\mathcal A(\Sigma')$, they can be represented as polynomials of $x'_1,\cdots,x'_m$, then we have $\frac{g(x_1,\cdots,x_n)}{x_1^{a_1}\cdots x_n^{a_n}}=\frac{h(x'_1,\cdots,x'_m)}{h'(x'_1,\cdots,x'_m)}$ for two coprime polynomials $h$ and $h'$. Thus, we have $\frac{g'(x'_1,\cdots,x'_m,\mathbf{x})}{{x'_1}^{b_1}\cdots{x'_m}^{b_m}\mathbf{x}^{\mathbf b}}=\frac{h(x'_1,\cdots,x'_m)}{h'(x'_1,\cdots,x'_m)}$. Since  $(g',{x'_1}^{b_1}\cdots{x'_m}^{b_m}\mathbf{x}^{\mathbf b})=1$ and $(h,h')=1$, we obtain that $g'=h$ and ${x'_1}^{b_1}\cdots{x'_m}^{b_m}\mathbf{x}^{\mathbf b}=h'$ up to multiple with a non-zero element in $\mathbb Q$. It follows that $\mathbf b=0$. Therefore, we get $f=\frac{h(x'_1,\cdots,x'_m)}{{x'_1}^{b_1}\cdots{x'_m}^{b_m}}$. Then $f\in \mathbb Q[{x'_1}^{\pm 1},\cdots,{x'_m}^{\pm 1}]$. Thus, we obtain $f\in \mathcal U(\Sigma')$ since the cluster $\{x'_1,\cdots,x'_m\}$ is arbitrary.

Since $\mathcal A(\Sigma)$ is acyclic, $\mathcal A(\Sigma')$ is also acyclic. By Corollary \ref{co2}, we get $f\in \mathcal U(\Sigma')=\mathcal A(\Sigma')\subseteq \mathcal A(\Sigma)$. Therefore, we have $\mathcal U(\Sigma)\subseteq \mathcal A(\Sigma)$.
\end{proof}

{\bf Acknowledgements:}\; This project is supported by the National Natural Science Foundation of China (No. 11671350 and No.11571173)  and the Zhejiang Provincial Natural Science Foundation of China (No.LZ13A010001).

\end{document}